\documentclass[11pt]{amsart}
\usepackage[utf8]{inputenc}
\usepackage[T1]{fontenc}
\usepackage{amsmath,amssymb,amsthm}
\usepackage[margin=1.1in]{geometry}
\usepackage{xcolor}
\usepackage{hyperref}
\hypersetup{colorlinks=true,linkcolor=blue!50!black,citecolor=blue!50!black,urlcolor=blue!50!black}

\theoremstyle{plain}
\newtheorem{theorem}{Theorem}[section]
\newtheorem{lemma}[theorem]{Lemma}
\newtheorem{proposition}[theorem]{Proposition}
\newtheorem{corollary}[theorem]{Corollary}
\theoremstyle{definition}
\newtheorem{definition}[theorem]{Definition}
\newtheorem{remark}[theorem]{Remark}
\newtheorem{example}[theorem]{Example}
\newtheorem{convention}[theorem]{Convention}

\newtheorem*{theoremA}{Theorem A}
\newtheorem*{theoremB}{Theorem B}
\newtheorem*{theoremC}{Theorem C}
\newtheorem*{theoremD}{Theorem D}
\newtheorem*{theoremE}{Theorem E}
 
\newcommand{\R}{\mathbb{R}}
\newcommand{\Z}{\mathbb{Z}}
\newcommand{\N}{\mathbb{N}}
\newcommand{\Q}{\mathbb{Q}}
\newcommand{\T}{\mathbb{T}}
\newcommand{\C}{\mathbb{C}}
\newcommand{\Dist}{\operatorname{Dist}}

\newcommand{\pre}{\widetilde{\mathcal{R}}}
\newcommand{\Ren}{\mathcal{R}}
\newcommand{\Norm}{\mathcal{N}}

\title[Simultaneous renormalization of commuting circle diffeomorphisms]{Simultaneous renormalization of commuting circle diffeomorphisms}

\author{Boris Petkovi\'c}
\address{University of Banja Luka, Faculty of Natural Sciences and Mathematics, Mladena Stojanovi\'ca 2, 78 000 Banja Luka, Bosnia and Herzegovina}
\email{boris.petkovic@pmf.unibl.org}
\date{February 2026}

\begin{document}

\begin{abstract}
We define a renormalization operator for pairs of commuting orientation preserving circle diffeomorphisms of class $C^3$. The operator is driven by the accelerated Brun multidimensional continued fraction algorithm. For rotation vectors of simultaneously bounded type, under a summable strong convergence condition satisfied by all eventually periodic types with Pisot period matrix, the renormalizations converge exponentially fast in $C^2$ to pairs of rigid rotations. The pair is simultaneously $C^{1+\gamma}$ conjugate to the pair of rotations, with $\gamma>0$ depending only on the bounded type constant. A fast renormalization extends exponential convergence and $C^1$ rigidity to a class of rotation vectors of full Lebesgue measure, containing explicit vectors outside bounded type. We show that the operator is topologically hyperbolic. The partition of the renormalized pairs by rotation vector is an invariant lamination whose leaves are the topological conjugacy classes. All results hold for pairs of class $C^{2+\alpha}$.
\end{abstract}

\maketitle

\section{Introduction}\label{sec:intro}

Renormalization entered one dimensional dynamics through the work of Feigenbaum \cite{Feigenbaum} and of Coullet and Tresser \cite{CoulletTresser} on the period doubling cascade. Lanford \cite{Lanford} proved the existence and hyperbolicity of the fixed point. The theory was then developed into a general tool for rigidity. Sullivan \cite{Sullivan}, McMullen \cite{McMullen} and Lyubich \cite{Lyubich} established the hyperbolicity of renormalization for unimodal maps, Martens \cite{Martens} constructed the periodic points, and Avila and Lyubich \cite{AvilaLyubich} built the full renormalization horseshoe. For circle maps the corresponding theory grew out of the linearization theorems of Herman \cite{Herman} and Yoccoz \cite{Yoccoz}. Khanin and Sinai \cite{KhaninSinai} and Katznelson and Ornstein \cite{KatznelsonOrnstein} gave renormalization proofs of rigidity at finite regularity, Khanin and Teplinsky \cite{KhaninTeplinsky} sharpened them, and de Faria and de Melo \cite{deFariadeMelo} and Yampolsky \cite{Yampolsky} proved hyperbolicity for critical circle maps. A complete self contained account of the renormalization of circle diffeomorphisms appears in the book of Dong, Martens and Palmisano \cite{DMP} and in the exposition \cite{Dong}. The central analytic statement there reads as follows. For a $C^3$ circle diffeomorphism whose rotation number is of bounded type, the renormalizations converge to rotations at exponential speed, and exponential convergence yields $C^1$ rigidity.

All of these theories concern a single map. In this paper we construct a renormalization operator for an action of rank two, generated by two commuting circle diffeomorphisms, and we prove for it the convergence, rigidity and hyperbolicity statements that the rank one theory provides for a single map. To our knowledge this is the first renormalization operator for an action of rank larger than one, and the first hyperbolicity statement for such an operator.

The motivating result is the theorem of Fayad and Khanin \cite{FayadKhanin}. Let $f$ and $g$ be commuting orientation preserving $C^\infty$ circle diffeomorphisms with rotation numbers $\alpha$ and $\beta$. If the vector $(\alpha,\beta)$ satisfies the simultaneous Diophantine condition, that is
\[
\max\big(\|q\alpha\|,\|q\beta\|\big)\ \ge\ C\,q^{-\sigma}\qquad\text{for all } q\in\N
\]
and some $C>0$, $\sigma>0$, then $f$ and $g$ are simultaneously $C^\infty$ conjugate to the rotations $R_\alpha$ and $R_\beta$. The condition is imposed on the vector and not on the coordinates. Each of $\alpha$ and $\beta$ may be a Liouville number, so that no element of the action need be linearizable by the theory of Herman and Yoccoz, and yet the whole action is rigid. This phenomenon was first observed by Moser \cite{Moser}. The proof of Fayad and Khanin is a KAM argument and works at infinite smoothness. The present paper provides a renormalization approach at finite smoothness which was missing from the litereature.

Two questions guide the construction. How does one renormalize a commuting pair, and which arithmetic class of rotation vectors plays the role of bounded type. Our answers are as follows. Lifting the pair together with the unit translation produces a free abelian group of rank three acting on the line, and the translation number is an injective homomorphism from this group to $\R$. A prerenormalization is a change of basis of this group, dictated by the accelerated Brun algorithm applied to the ordered triple of translation numbers. Renormalization is the prerenormalization followed by restriction to a shrinking fundamntal interval and affine rescaling. A rotation vector is of simultaneously bounded type when the three translation numbers produced by the algorithm remain uniformly comparable at every stage. The class contains every vector whose Brun expansion is eventually periodic, for instance the Tribonacci vector and the supergolden vector.

The commuting setting has two structural advantages over the rank one theory. The first is that the renormalized objects are smooth commuting pairs on a quotient circle, the quotient of the line by the largest element of the current basis. The other generators commute with that element, so they descend to smooth diffeomorphisms of the quotient circle, with no condition on their derivatives at the identification point. This is in contrast with the commuting pairs of the rank one theory, where the two maps have to be glued into a circle map by hand and agreement of all derivatives at the gluing point has to be imposed as a hypothesis. The second is that every element of the lifted group with positive translation number moves every point of the line forward. This produces self similar tilings of the circle, which replace the dynamical partitions of the rank one theory.

\subsection*{Main results}
Throughout, $f$ and $g$ are commuting orientation preserving circle diffeomorphisms of class $C^3$ with rotation numbers $\alpha$ and $\beta$, and $1$, $\alpha$, $\beta$ are linearly independent over $\Q$. The renormalization operator $\Ren$ acts on normalized pairs $\xi_n=\Ren^n\xi_0$ of two branch maps of $[0,1)$, and $\mathcal{O}_n$ denotes the pair of rigid rotations with the same combinatorial data as $\xi_n$. The distance $d_1$ is the $C^1$ distance of the normalized two branch maps, which also controls the distance of the cut points. The construction depends on the choice of a base point $x_0$ on the circle. The proofs require the base point to be admissible in the sense of Convention \ref{con:basepoint}, a condition satisfied by $\mu$ almost every point of the circle, where $\mu$ is the invariant measure of the action, but the conclusion of Theorem A then holds for every base point, by Proposition \ref{prop:allbasepoints}. All notions are defined in Sections \ref{sec:operator} and \ref{sec:realbounds}.

\begin{theoremA}
Suppose that $(\alpha,\beta)$ is of simultaneously bounded type with constant $A$ and summably convergent. Then, for every admissible base point, there are $\xi\in(0,1)$, which can be chosen to depend only on $A$, and $C$ such that
\[
d_1\big(\xi_n,\mathcal{O}_n\big)\ \le\ C\,\xi^{\,n}\qquad\text{for all } n\ge0 .
\]
\end{theoremA}

The convergence holds in the $C^2$ sense as well, by Proposition \ref{prop:C2}.

\begin{theoremB}
Under the hypotheses of Theorem A the homeomorphism $h$ of the circle with $hfh^{-1}=R_\alpha$ and $hgh^{-1}=R_\beta$ is a $C^{1+\gamma}$ diffeomorphism, for some $\gamma>0$ depending only on $A$.
\end{theoremB}

Beyond bounded type we introduce a fast renormalization along the good stages of the Brun expansion, the stages at which the expansion looks, for a bounded number of steps ahead, like an expansion of bounded type. A vector is fast renormalizable if good stages recur with positive frequency, the drops of scale between consecutive good stages grow subexponentially, and a digit weighted form of summable convergence holds.

\begin{theoremC}
Suppose that $(\alpha,\beta)$ is fast renormalizable. Then, for every admissible base point, the renormalizations along the good stages converge exponentially to the rotation pairs, and $f$ and $g$ are simultaneously $C^1$ conjugate to $R_\alpha$ and $R_\beta$.
\end{theoremC}

\begin{theoremD}
For Lebesgue almost every vector $(\alpha,\beta)$, every pair of commuting orientation preserving $C^3$ circle diffeomorphisms with rotation vector $(\alpha,\beta)$ is simultaneously $C^1$ conjugate to the rotation pair.
\end{theoremD}

The operator loses derivatives on $C^3$ pairs, so hyperbolicity is formulated topologically, in the lamination sense of \cite{DMP} and of the renormalization horseshoes of \cite{Martens,AvilaLyubich}. The operator acts on the normalized pairs of all pointed commuting pairs, that is, commuting pairs of orientation preserving $C^3$ circle diffeomorphisms whose rotation vector is rationally independent together with $1$, together with a base point. Let $\mathcal{X}^3_{A}$ be the space of those pointed pairs whose rotation vector is of simultaneously bounded type with constant $A$ and summably convergent and whose base point is admissible, and let $\mathcal{N}_A$ be the set of their normalized pairs. The periodic points of the statement are the rotation pairs of purely periodic type with Pisot period matrix.

\begin{theoremE}
The partition of the normalized pairs by rotation vector is an $\Ren$ invariant lamination whose leaves are the topological conjugacy classes. On $\mathcal{N}_A$ any two pairs on the same leaf are contracted exponentially towards each other by $\Ren$, at a rate depending only on $A$. The family of all rotation pairs is a global transversal, on which $\Ren$ acts as the accelerated Brun map. On the leaf space of $\mathcal{N}_A$ the induced action is the shift on the corresponding renormalization types. Through every periodic point there is a two dimensional local unstable manifold contained in the transversal, with multipliers the Galois quotients of the Perron eigenvalue of the period matrix, all of modulus larger than one.
\end{theoremE}

We also settle the reconstruction problem for the renormalization type. The type of a vector of simultaneously bounded type determines the vector. A vector has purely periodic type if and only if it is the normalized Perron eigenvector of the period matrix of the type, and the Perron eigenvalue is then an algebraic unit, a Pisot number whenever its conjugates are nonreal.

\subsection*{Organization}
Section \ref{sec:brun} recalls the accelerated Brun algorithm as a dynamical system. Section \ref{sec:triples} sets up the group of lifts. Section \ref{sec:operator} defines the prerenormalization and renormalization operators and the class of simultaneously bounded type. Section \ref{sec:type} treats the renormalization type and the reconstruction problem. Section \ref{sec:towers} constructs the tower tilings. Section \ref{sec:realbounds} proves the real bounds. Section \ref{sec:moebius} proves Theorem A. Section \ref{sec:rigidity} proves Theorem B. Section \ref{sec:beyond} treats the lattice flow, fast renormalization, Theorems C and D, gives an explicit fast renormalizable vector which is not of simultaneously bounded type, and compares the classes of vectors treated here with the coordinatewise classes of the rank one theory. Section \ref{sec:periodic} works out the supergolden and Tribonacci periodic points and a totally real periodic point, and records the contraction factors along the leaves at a periodic point. Section \ref{sec:topohyp} proves Theorem E.

\section{The accelerated Brun algorithm}\label{sec:brun}

We recall the algorithm of Brun \cite{Brun} in its accelerated ordered form, following \cite{Schweiger}. Let
\[
D\ =\ \{\,(t_1,t_2,t_3)\in\R^3\ \colon\ t_1>t_2>t_3>0\,\}
\]
be the open cone of strictly decreasing positive triples. Given $t\in D$ set $a=\lfloor t_1/t_2\rfloor\ge1$ and $r=t_1-at_2\in[0,t_2)$. The case $r=0$ occurs only when $t_1/t_2$ is rational and is excluded in all situations considered below, so that $r\in(0,t_2)$. The accelerated Brun step is
\[
\mathcal{S}(t)\ =\
\begin{cases}
(t_2,\,r,\,t_3) & \text{if } r>t_3,\\
(t_2,\,t_3,\,r) & \text{if } r<t_3 .
\end{cases}
\]
The first case is called case $A$ and the second case $B$. The pair $(a,\varepsilon)$ with $\varepsilon\in\{A,B\}$ is the symbol of the step. The step is undefined on the set $\{r=t_3\}$, a countable union of hyperplanes, which is avoided in all situations considered below. The step subtracts the second largest entry from the largest one as many times as possible and reorders. It is the natural two dimensional analogue of the Gauss map, in which the digit is collected in one accelerated step rather than through a run of subtractive steps.

\subsection{Step matrices}
For an integer $a\ge1$ define the nonnegative integer matrices
\[
M_A(a)\ =\ \begin{pmatrix}a&1&0\\ 1&0&0\\ 0&0&1\end{pmatrix},
\qquad
M_B(a)\ =\ \begin{pmatrix}a&0&1\\ 1&0&0\\ 0&1&0\end{pmatrix}.
\]
They are obtained by inverting the step. If $t'=\mathcal{S}(t)$ with symbol $(a,A)$, then $t_1=at_1'+t_2'$, $t_2=t_1'$, $t_3=t_3'$, so $t=M_A(a)t'$. If the symbol is $(a,B)$, then $t_1=at_1'+t_3'$, $t_2=t_1'$, $t_3=t_2'$, so $t=M_B(a)t'$. Their determinants are $\det M_A(a)=-1$ and $\det M_B(a)=1$.

\begin{lemma}\label{lem:cylinders}
For every $a\ge1$ and $\varepsilon\in\{A,B\}$ the matrix $M_\varepsilon(a)$ maps $D$ bijectively onto the set $D_\varepsilon(a)$ of triples in $D$ whose step has symbol $(a,\varepsilon)$. Consequently every finite word in the symbols is realized by an open nonempty set of triples, and the set of triples whose expansion begins with the word $(a_1,\varepsilon_1)\cdots(a_p,\varepsilon_p)$ is the image of $D$ under $M_{\varepsilon_1}(a_1)\cdots M_{\varepsilon_p}(a_p)$.
\end{lemma}

\begin{proof}
Let $t'\in D$ and $t=M_A(a)t'=(at_1'+t_2',\,t_1',\,t_3')$. Then $t_1>t_1'=t_2>t_3'=t_3>0$, so $t\in D$. Since $0<t_2'/t_1'<1$ we have $\lfloor t_1/t_2\rfloor=a$, and the remainder is $t_2'$, which exceeds $t_3'=t_3$, so the symbol is $(a,A)$ and $\mathcal{S}(t)=t'$. Conversely if $t\in D_A(a)$ then $t=M_A(a)\mathcal{S}(t)$ by the inversion above. The case $B$ is identical with $t=(at_1'+t_3',\,t_1',\,t_2')$, where the remainder $t_3'$ is smaller than $t_2'=t_3$. The statement about words follows by induction on the length.
\end{proof}

\subsection{The projectivized map and its ergodic theory}
The step commutes with scaling. Normalizing by the first coordinate identifies $D$ modulo scaling with the open triangle
\[
\Delta\ =\ \{\,(x,y)\ \colon\ 1>x>y>0\,\},
\]
and the step induces the map
\[
T(x,y)\ =\
\begin{cases}
\big(\tfrac{1-ax}{x},\ \tfrac{y}{x}\big) & \text{if } 1-ax>y,\\[2pt]
\big(\tfrac{y}{x},\ \tfrac{1-ax}{x}\big) & \text{if } 1-ax<y,
\end{cases}
\qquad a=\lfloor 1/x\rfloor .
\]
This is the accelerated Brun map on the projective simplex. By Lemma \ref{lem:cylinders} it is a piecewise projective map whose branches are bijections of the cylinders $\Delta_\varepsilon(a)$ onto $\Delta$. We recall the facts about $T$ that will be used in Section \ref{sec:beyond}.

\begin{theorem}\label{thm:brunergodic}
The map $T$ preserves the probability measure $\nu$ on $\Delta$ with density $c\,h$, where
\[
h(x,y)\ =\ \frac{1}{(1+x)(1+x+y)}+\frac{1}{(1+y)(1+x+y)}\ =\ \frac{2+x+y}{(1+x)(1+y)(1+x+y)}
\]
and $c$ is a normalizing constant, and $T$ is ergodic with respect to $\nu$. Since $\tfrac13\le h\le2$ on $\Delta$, the measure $\nu$ is comparable to Lebesgue measure on all of $\Delta$. In particular $\nu\big(\{(x,y)\in\Delta\colon\lfloor1/x\rfloor\ge k\}\big)\le c\,k^{-2}$ for every $k\ge1$, and the function $(x,y)\mapsto\log\lfloor 1/x\rfloor$ is $\nu$ integrable. Moreover $T$ satisfies the Pisot condition, that is, the second Lyapunov exponent of the cocycle of step matrices is strictly negative, so the algorithm is strongly convergent with exponential rate for $\nu$ almost every point.
\end{theorem}

The existence of an absolutely continuous invariant probability measure and the ergodicity of $T$ are due to Schweiger, see \cite[Chapter 6]{Schweiger}, and the invariant density was identified by Arnoux and Nogueira \cite{ArnouxNogueira}. The invariance of $h$ can also be verified directly. The inverse branches of $T$ are, by Lemma \ref{lem:cylinders}, the maps $(u,v)\mapsto\big(\tfrac{1}{a+u},\tfrac{v}{a+u}\big)$ and $(u,v)\mapsto\big(\tfrac{1}{a+v},\tfrac{u}{a+v}\big)$, $a\ge1$, with Jacobians $(a+u)^{-3}$ and $(a+v)^{-3}$, so the transfer operator $\mathcal{L}$ of $T$ acts on $h$ by
\begin{align*}
\mathcal{L}h(u,v)\ =\ \sum_{a\ge1}\Big[&\frac{1}{(a+u)(a+u+1)(a+u+v+1)}+\frac{1}{(a+u)(a+u+v)(a+u+v+1)}\\
&+\frac{1}{(a+v)(a+v+1)(a+u+v+1)}+\frac{1}{(a+v)(a+u+v)(a+u+v+1)}\Big] .
\end{align*}
Expanding each summand in partial fractions with respect to $a$, the coefficients of $\frac{1}{a+u}$ and $\frac{1}{a+u+1}$ are $\frac1v$ and $-\frac1v$, those of $\frac{1}{a+v}$ and $\frac{1}{a+v+1}$ are $\frac1u$ and $-\frac1u$, and those of $\frac{1}{a+u+v}$ and $\frac{1}{a+u+v+1}$ are $-\frac{u+v}{uv}$ and $\frac{u+v}{uv}$, so all three series telescope and
\[
\mathcal{L}h(u,v)\ =\ \frac{1}{v(1+u)}+\frac{1}{u(1+v)}-\frac{u+v}{uv(1+u+v)}\ =\ h(u,v) .
\]
The region $\{\lfloor 1/x\rfloor\ge k\}=\{x<1/k\}$ has Lebesgue measure $1/(2k^2)$, which gives the bound on its $\nu$ measure and the integrability of $\log\lfloor1/x\rfloor$. The negativity of the second Lyapunov exponent, equivalently the almost everywhere strong convergence of the algorithm, is proved by Hardcastle \cite{Hardcastle}, with the reduction to a finite computation set up in \cite{HardcastleKhanin}, and further estimates on the second exponent appear in \cite{BST}. We shall use only the statements collected in the theorem, in the form quoted, and we refer to the cited references for their proofs.

\section{Commuting triples}\label{sec:triples}

We write $\T=\R/\Z$ and $T(x)=x+1$. Let $f$ and $g$ be commuting orientation preserving circle diffeomorphisms of class $C^3$ with rotation numbers $\alpha$ and $\beta$, and assume that $1$, $\alpha$, $\beta$ are linearly independent over $\Q$. Let $F$ and $G$ be the lifts of $f$ and $g$ normalized by $F(0)\in[0,1)$ and $G(0)\in[0,1)$. For an increasing homeomorphism $W$ of $\R$ commuting with $T$ the translation number is
\[
\tau(W)\ =\ \lim_{n\to\infty}\frac{W^n(x)-x}{n}.
\]
The limit exists, is independent of $x$, and satisfies $|W^n(x)-x-n\tau(W)|<1$ for all $n$ and $x$.

\begin{lemma}\label{lem:liftscommute}
The lifts $F$ and $G$ commute.
\end{lemma}

\begin{proof}
The map $G^{-1}FG$ is a lift of $g^{-1}fg=f$, so $G^{-1}FG=T^cF$ for some integer $c$. Translation numbers are invariant under conjugation. Indeed, with $y=G(x)$ we have $G^{-1}F^nG(x)-x=(G^{-1}(F^ny)-F^ny)+(F^ny-y)+(G(x)-x)$. The first and third summands are bounded uniformly in $n$ because $G^{-1}-\mathrm{id}$ and $G-\mathrm{id}$ are continuous and periodic. Dividing by $n$ gives $\tau(G^{-1}FG)=\tau(F)=\alpha$. On the other hand $T^c$ and $F$ commute, so $(T^cF)^n=T^{cn}F^n$ and $\tau(T^cF)=\alpha+c$. Hence $c=0$.
\end{proof}

By Lemma \ref{lem:liftscommute} the maps $F$, $G$, $T$ generate an abelian group
\[
\Gamma\ =\ \{\,F^kG^lT^m\ \colon\ (k,l,m)\in\Z^3\,\}
\]
of increasing homeomorphisms of $\R$, each of which is a lift of a circle homeomorphism.

\begin{lemma}\label{lem:tauhom}
The map $\tau\colon\Gamma\to\R$ is a homomorphism, $\tau(F^kG^lT^m)=k\alpha+l\beta+m$, and it is injective. In particular $\Gamma\cong\Z^3$.
\end{lemma}

\begin{proof}
Let $U,V\in\Gamma$. Since $U$ and $V$ commute, $(UV)^n=U^nV^n$, and
\[
(UV)^n(x)-x\ =\ \big(U^n(V^nx)-V^nx\big)+\big(V^n(x)-x\big).
\]
Using $|U^n(y)-y-n\tau(U)|<1$ uniformly in $y$ and dividing by $n$ gives $\tau(UV)=\tau(U)+\tau(V)$. The formula on generators follows from $\tau(F)=\alpha$, $\tau(G)=\beta$, $\tau(T)=1$. Injectivity is the rational independence of $1$, $\alpha$, $\beta$.
\end{proof}

\begin{lemma}\label{lem:forward}
Let $W\in\Gamma$. If $\tau(W)>0$ then $W(x)>x$ for every $x\in\R$. If $\tau(W)<0$ then $W(x)<x$ for every $x$.
\end{lemma}

\begin{proof}
Let $\tau(W)>0$. If $W(x_0)\le x_0$ for some $x_0$ then, since $W$ is increasing, the sequence $W^n(x_0)$ is nonincreasing. Hence $W^n(x_0)-x_0\le0$ for all $n$ and $\tau(W)\le0$, a contradiction. The second statement follows by applying the first to $W^{-1}$.
\end{proof}

Lemma \ref{lem:forward} is elementary and decisive. The order of $\Gamma$ given by translation numbers is realized pointwise on the line.

\begin{lemma}\label{lem:toplin}
There is an orientation preserving homeomorphism $h$ of $\T$ with $hfh^{-1}=R_\alpha$ and $hgh^{-1}=R_\beta$. Moreover $h$ is unique up to composition with a rotation.
\end{lemma}

\begin{proof}
Since $\alpha$ is irrational and $f$ is of class $C^2$, Denjoy's theorem \cite{Denjoy} provides a homeomorphism $h$ with $hfh^{-1}=R_\alpha$. The map $g'=hgh^{-1}$ commutes with $R_\alpha$, so the continuous function $g'(x)-x$ is $R_\alpha$ invariant, hence constant by minimality of $R_\alpha$. Thus $g'$ is a rotation, and it is $R_\beta$ because rotation numbers are preserved. Any self conjugacy of the minimal rotation $R_\alpha$ commutes with it and is therefore a rotation by the same argument, which gives uniqueness up to rotation.
\end{proof}

We write $\mu$ for the unique probability measure on $\T$ invariant under both $f$ and $g$. It is the pull back of Lebesgue measure by $h$, and it is the unique $f$ invariant probability measure. We also write $\mu$ for its lift to $\R$, a $\sigma$ finite measure invariant under every element of $\Gamma$ and assigning mass $\tau(W)$ to every interval $[x,Wx)$ with $\tau(W)>0$.

\section{The prerenormalization and renormalization operators}\label{sec:operator}

\begin{definition}\label{def:triple}
A positive triple is an ordered basis $u=(u_1,u_2,u_3)$ of $\Gamma$ whose translation numbers $t_i=\tau(u_i)$ satisfy $t_1>t_2>t_3>0$. The vector $t(u)=(t_1,t_2,t_3)\in D$ is the scale vector of the triple.
\end{definition}

After replacing $F$ by $FT^{-m}$ and $G$ by $GT^{-m'}$ we may and do assume $0<\beta<\alpha<1$, exchanging the names of $f$ and $g$ if necessary. Then $(T,F,G)$ is a positive triple with scale vector $(1,\alpha,\beta)$. Strictness of the inequalities and positivity at every stage below follow from the injectivity of $\tau$ and from the linear independence of basis vectors.

\begin{definition}\label{def:prestep}
Let $u=(u_1,u_2,u_3)$ be a positive triple with scale vector $t$, let $(a,\varepsilon)$ be the symbol of $t$ under the Brun step, and put $v=u_1u_2^{-a}$, so that $\tau(v)=t_1-at_2\in(0,t_2)$. The prerenormalization of $u$ is the positive triple
\[
\pre(u)\ =\
\begin{cases}
(u_2,\,v,\,u_3) & \text{if } \varepsilon=A,\\
(u_2,\,u_3,\,v) & \text{if } \varepsilon=B .
\end{cases}
\]
\end{definition}

The scale vector of $\pre(u)$ is $\mathcal{S}(t)$, so $t(u)=M_\varepsilon(a)\,t(\pre u)$. Note that $\tau(v)\neq t_3$, since $\tau(v)=t_3$ would force $u_1u_2^{-a}u_3^{-1}=\mathrm{id}$ by injectivity of $\tau$, contradicting the linear independence of the basis. The triple $\pre(u)$ is again a basis of $\Gamma$, since the passage matrix is unimodular. Prerenormalization is a change of basis of the group $\Gamma$ and nothing else. It involves neither a return map nor a rescaling. We write
\[
u^{(n)}\ =\ \pre^{\,n}(T,F,G),
\]
with scale vector $t(n)=(t_1(n),t_2(n),t_3(n))$ and symbols $(a_n,\varepsilon_n)$, where $(a_n,\varepsilon_n)$ is the symbol of the step from $u^{(n-1)}$ to $u^{(n)}$. Thus $t(n-1)=M_{\varepsilon_n}(a_n)\,t(n)$ and the sequence $\{(a_n,\varepsilon_n)\}_{n\ge1}$ is the Brun expansion of $(1,\alpha,\beta)$. We fix once and for all a base point $x_0\in\R$ and set
\[
I_n^{(i)}\ =\ \big[x_0,\,u_i^{(n)}(x_0)\big),\qquad i=1,2,3 .
\]
By Lemma \ref{lem:forward} these are intervals with $I_n^{(3)}\subset I_n^{(2)}\subset I_n^{(1)}$, and $I_{n+1}^{(1)}=I_n^{(2)}$ for every $n$. The intervals $I_n^{(1)}$ are nested and shrink towards $x_0$. The base point is arbitrary for the constructions of this section and of Section \ref{sec:towers}. From Section \ref{sec:realbounds} on it will be required to be admissible in the sense of Convention \ref{con:basepoint}, a condition satisfied by $\mu$ almost every point of the circle.

\begin{lemma}\label{lem:decomposition}
Let $u$ be a positive triple with digit $a$, write $I^{(i)}=[x_0,u_ix_0)$ and $I'=[x_0,vx_0)$ with $v=u_1u_2^{-a}$. Then
\[
I^{(1)}\ =\ \bigsqcup_{j=0}^{a-1}u_2^{\,j}\big(I^{(2)}\big)\ \sqcup\ u_2^{\,a}\big(I'\big),
\]
a partition into consecutive intervals.
\end{lemma}

\begin{proof}
The points $x_0<u_2x_0<\dots<u_2^ax_0$ increase by Lemma \ref{lem:forward} applied to $u_2$. Since $u_1u_2^{-a}$ has positive translation number, Lemma \ref{lem:forward} gives $u_1u_2^{-a}(u_2^ax_0)>u_2^ax_0$, that is $u_1x_0>u_2^ax_0$. The interval $[u_2^jx_0,u_2^{j+1}x_0)$ equals $u_2^j(I^{(2)})$ for $j<a$, and $[u_2^ax_0,u_1x_0)=u_2^a\big([x_0,u_2^{-a}u_1x_0)\big)=u_2^a(I')$. These consecutive intervals partition $[x_0,u_1x_0)$.
\end{proof}

\subsection{Quotient circles and normalized pairs}
Since $u_2^{(n)}$ and $u_3^{(n)}$ commute with $u_1^{(n)}$, they descend to $C^3$ diffeomorphisms $P_n$ and $Q_n$ of the quotient circle
\[
\mathcal{C}_n\ =\ \R/\langle u_1^{(n)}\rangle ,
\]
a smooth circle of total length $|I_n^{(1)}|$ in the metric induced from $\R$. The two ends of the chart project to a single point of $\mathcal{C}_n$, the image of the base point. Near this point the map $P_n$ is represented both by $u_2^{(n)}$ and by $(u_1^{(n)})^{-1}u_2^{(n)}u_1^{(n)}$, and these agree identically because the elements commute. No matching of derivatives of any order has to be imposed there. In the fundamental chart $I_n^{(1)}=[x_0,u_1^{(n)}x_0)$ the map $P_n$ has the two branches
\[
x\ \mapsto\ u_2^{(n)}(x)\quad\text{on } [x_0,c_2(n)),\qquad
x\ \mapsto\ (u_1^{(n)})^{-1}u_2^{(n)}(x)\quad\text{on } [c_2(n),u_1^{(n)}x_0),
\]
with cut point $c_2(n)=(u_2^{(n)})^{-1}u_1^{(n)}(x_0)$, which lies in the interior of $I_n^{(1)}$ by Lemma \ref{lem:forward}. Similarly $Q_n$ has cut point $c_3(n)=(u_3^{(n)})^{-1}u_1^{(n)}(x_0)$. The cut points depend on the base point $x_0$ and carry no intrinsic information. They mark where the chart breaks, not where the map breaks.

\begin{definition}\label{def:renorm}
The normalization $\Norm$ rescales the chart $I_n^{(1)}$ affinely onto $[0,1)$. The normalized pair $\xi_n=\Norm(P_n,Q_n)$ is the pair of two branch maps of $[0,1)$ so obtained. The renormalization operator is $\Ren=\Norm\circ\pre$, so that $\xi_{n+1}=\Ren\xi_n$. The rotation model $\mathcal{O}_n$ is the pair of rigid rotations by $t_2(n)/t_1(n)$ and $t_3(n)/t_1(n)$, written in the same two branch form on $[0,1)$.
\end{definition}

In the linear model, where $f$ and $g$ are themselves rotations, every $\xi_n$ equals $\mathcal{O}_n$. In general the two pairs share their combinatorial data, namely the rotation numbers of $P_n$ and $Q_n$ on $\mathcal{C}_n$, which are $t_2(n)/t_1(n)$ and $t_3(n)/t_1(n)$ because $\mu$ assigns mass $t_i(n)$ to $I_n^{(i)}$.

\begin{remark}\label{rem:welldefined}
The normalized pair $\xi_n$ determines $\xi_{n+1}$, so that $\Ren$ is a well defined self map of the set of normalized pairs arising from pointed pairs. Indeed, $\xi_n$ is a chart of the pointed pair $(\mathcal{C}_n;P_n,Q_n;x_0)$: the two branches of $P_n$ glue to a homeomorphism of the circle $[0,1)/\{0\sim1\}$, and the smooth structure of $\mathcal{C}_n$ at the point $0\sim1$ is the one for which the gluing map, namely the first branch of $P_n$ composed with the inverse of the second branch, which is $u_1^{(n)}$ in the chart, is a diffeomorphism. The universal cover of $\mathcal{C}_n$, with the deck transformation and with the lifts of $P_n$ and $Q_n$ chosen so that they move the base point into the fundamental domain, is the line with the basis $u^{(n)}$ of $\Gamma$, up to an affine change of coordinates, and stage $n+1$ is built from these data and from the base point alone.
\end{remark}

\subsection{Simultaneously bounded type}

Write $u_i^{(n)}=F^{k_i(n)}G^{l_i(n)}T^{m_i(n)}$, and call the vector $e_i(n)=(k_i(n),l_i(n))\in\Z^2$ the exponent vector of $u_i^{(n)}$. The height of the stage $n$ basis is
\[
H(n)\ =\ \max_{i=1,2,3}\big(|k_i(n)|+|l_i(n)|\big).
\]
Since $t(0)=B_nt(n)$ with $B_n=M_{\varepsilon_1}(a_1)\cdots M_{\varepsilon_n}(a_n)$, the exponent vectors are read off the rows of $B_n^{-1}$, and the heights are determined by the type.

\begin{definition}\label{def:sbt}
The vector $(\alpha,\beta)$ is of simultaneously bounded type with constant $A\ge1$ if
\[
t_1(n)\ \le\ A\,t_3(n)\qquad\text{for every } n\ge0 .
\]
\end{definition}

The second hypothesis we shall need is not a condition on the scales but on the quality of the approximation produced by the algorithm.

\begin{definition}\label{def:sc}
The vector $(\alpha,\beta)$ is summably convergent if
\[
\Sigma\ :=\ \sum_{n\ge0}t_1(n)\,H(n)\ <\ \infty .
\]
\end{definition}

The product $t_1(n)H(n)$ measures the quality of the approximation of zero by the linear forms $k\alpha+l\beta+m$ delivered by the algorithm at stage $n$, the rows of $B_n^{-1}$. This is the dual of the simultaneous approximation of $(\alpha,\beta)$ by the columns of $B_n$. The algorithm is strongly convergent at $(\alpha,\beta)$ when $t_1(n)H(n)\to0$, and summable convergence is strong convergence with a summable rate. For $\nu$ almost every vector the Oseledets theorem gives $\lim_n n^{-1}\log\big(t_1(n)H(n)\big)=\theta_2$, where $\theta_2<0$ is the second Lyapunov exponent of the cocycle of step matrices, see Theorem \ref{thm:brunergodic}, so summable convergence holds for almost every vector by the theorem of Hardcastle \cite{Hardcastle}, and it holds for every vector whose type is eventually periodic with a Pisot period matrix, by Theorem \ref{thm:periodic} below. We shall also use the digit weighted series
\[
\Sigma^{\ast}\ :=\ \sum_{n\ge0}a_{n+1}\,t_1(n)\,H(n) ,
\]
which is at most $A\Sigma$ for a vector of simultaneously bounded type with constant $A$, by Lemma \ref{lem:scales} below.

\begin{lemma}\label{lem:scales}
Suppose that $(\alpha,\beta)$ is of simultaneously bounded type with constant $A$. Then $a_n\le A$ for every $n$. With $S(n)=t_1(n)+t_2(n)+t_3(n)$ and $\lambda=1-\tfrac{1}{3A}$ we have
\[
S(n)\ \le\ S(0)\,\lambda^{\,n},\qquad \tfrac{1}{3A}\,S(n)\ \le\ t_i(n)\ \le\ S(n)\quad (i=1,2,3).
\]
Moreover $t_2(n)/t_1(n)\le A/(A+1)$ and $\tau(v_n)\ge t_2(n)/A$, where $v_n=u_1^{(n)}(u_2^{(n)})^{-a_{n+1}}$ is the element created at stage $n$.
\end{lemma}

\begin{proof}
First $a_{n+1}\le t_1(n)/t_2(n)\le t_1(n)/t_3(n)\le A$. The step removes $a_{n+1}t_2(n)$ from the sum, so
\[
S(n+1)\ =\ S(n)-a_{n+1}t_2(n)\ \le\ S(n)-t_2(n)\ \le\ S(n)\Big(1-\tfrac{1}{3A}\Big),
\]
using $t_2(n)\ge t_3(n)\ge t_1(n)/A\ge S(n)/(3A)$. The comparability of the $t_i(n)$ with $S(n)$ follows from $t_1(n)\ge S(n)/3$ and $t_3(n)\ge t_1(n)/A$. The number $\tau(v_n)$ is either $t_2(n+1)$ or $t_3(n+1)$, and both are at least $t_1(n+1)/A=t_2(n)/A$. Finally $t_1(n)=a_{n+1}t_2(n)+\tau(v_n)\ge t_2(n)(1+1/A)$.
\end{proof}

\begin{lemma}\label{lem:heights}
For every $n$ one has $t_1(n)H(n)^2\ge1/6$. If $(\alpha,\beta)$ is of simultaneously bounded type with constant $A$ then every product $w=\prod_{m<n}(u_2^{(m)})^{j_m}$ with $0\le j_m\le a_{m+1}$ has exponent vector of height at most $A\sum_{m<n}H(m)$.
\end{lemma}

\begin{proof}
The vectors $b_i(n)=(t_i(n),e_i(n))\in\R^3$ form a basis of the lattice $\{(k\alpha+l\beta+m,k,l)\}$, which has covolume one. Expanding the determinant of the matrix with rows $b_i(n)$, each of the six products contains exactly one scale and two exponents, so $1\le6\,t_1(n)H(n)^2$. The second statement holds because the exponent vector of $w$ is $\sum_mj_me_2(m)$ and $j_m\le a_{m+1}\le A$ by Lemma \ref{lem:scales}.
\end{proof}

\begin{remark}\label{rem:sbtarith}
The two hypotheses play different roles. Simultaneously bounded type controls the geometry of one step, so it gives the uniform contraction of the scales and the comparability of the three nested intervals. Summable convergence controls the total nonlinearity accumulated over all stages, and it is what replaces the disjointness of the dynamical partition of the rank one theory. Proposition \ref{prop:symbolic} shows that the first condition is decided by the digits and by the gaps between occurrences of the letter $B$. The second is determined by the type as well, since for a vector of simultaneously bounded type the type determines the vector by Theorem \ref{thm:reconstruction}, but it is a condition on the whole type and not on finitely many symbols, and we do not know whether it follows from the first. Section \ref{sec:beyond} explains why the arithmetic of a vector is visible in the space of lattices rather than in the symbolic data.
\end{remark}

\section{The renormalization type and the reconstruction problem}\label{sec:type}

\begin{definition}\label{def:type}
The simultaneous renormalization type of the pair $(f,g)$, or of the vector $(\alpha,\beta)$, is the sequence of symbols
\[
(a_n,\varepsilon_n)\ \in\ \N_+\times\{A,B\},\qquad n\ge1,
\]
of the Brun expansion of $(1,\alpha,\beta)$.
\end{definition}

The terminology follows the rank one theory of \cite{DMP}. In the rank one theory the renormalization operator is the first return map to an interval around the cutting point, and at every stage the renormalized map lies in one of two chambers, according to the position of the cutting point relative to the image of the endpoint. This produces a sign sequence, and the renormalization type of \cite{DMP} is the sequence of run lengths of the sign sequence. A continued fraction digit is collected only at a chamber switch. The type determines the rotation number, $\rho=[a_1+1,a_2,a_3,\dots]$, and conversely.

Every item of this structure has a counterpart here. In the subtractive form of the Brun algorithm, in which one copy of $u_2$ is removed at each move and the triple is then reordered, each move records one of three letters, according to the position of the remainder in the reordered triple. A maximal run of the letter meaning that the remainder is still the largest has length $a_n-1$ and is closed by one of the letters $A$ or $B$. So the accelerated symbol $(a_n,\varepsilon_n)$ is the run length structure of the subtractive move sequence, the only difference with rank one being that a run can be closed in two ways rather than one. The two chambers correspond to the two cases of the reordering, the sign sequence to the case sequence, the run lengths to the digits, and the two return times of the rank one theory to the three tower heights of Section \ref{sec:towers}. Periodicity of the type corresponds to periodicity under renormalization. The golden mean, of constant type one, in rank one theory is replaced by the supergolden vector with symbol $(1,B)$ at every step and by the Tribonacci vector with period three, treated in Section \ref{sec:periodic}.

\subsection{Positivity of blocks}

\begin{lemma}\label{lem:positivity}
The letter $B$ occurs infinitely often in the type of every vector. The product of the step matrices over any block of the expansion containing at least four occurrences of $B$ is a strictly positive matrix. If a block product is strictly positive, then so is its product with any further step matrices on either side.
\end{lemma}

\begin{proof}
Suppose that the letter $A$ occurred at every step from some stage $n$ on. Then $t_3(m)=t_3(n)$ for all $m\ge n$, while the pair $(t_1(m),t_2(m))$ undergoes the subtractive Euclidean algorithm of the irrational ratio $t_1(n)/t_2(n)$, whose entries tend to zero. Eventually $t_2(m)<t_3(n)$, which contradicts the ordering of the triple. Hence $B$ occurs infinitely often.

For the second statement we track the supports of the columns of the growing product $B_k=M_1\cdots M_k$, where $M_j$ is the step matrix of the $j$th step of the block. Let $S_i(k)\subset\{1,2,3\}$ be the set of indices of nonzero entries of the $i$th column of $B_k$, with $S_i(0)=\{i\}$. Multiplying on the right by $M_A(a)$ replaces the columns $(c_1,c_2,c_3)$ by $(ac_1+c_2,\,c_1,\,c_3)$, and multiplying by $M_B(a)$ replaces them by $(ac_1+c_2,\,c_3,\,c_1)$. Since all entries are nonnegative, the supports are replaced by $(S_1\cup S_2,\,S_1,\,S_3)$ in case $A$ and by $(S_1\cup S_2,\,S_3,\,S_1)$ in case $B$. The first step of either kind gives $S_1=\{1,2\}$. At the first occurrence of $B$ afterwards, $S_3$ becomes the previous $S_1$, which contains $1$, and $S_2$ becomes the previous $S_3=\{3\}$. The following step of either kind replaces $S_1$ by $S_1\cup S_2\supset\{1,2,3\}$. From then on $S_1=\{1,2,3\}$, since $S_1$ can only grow. The next occurrence of $B$ transports the full support into $S_3$, and the following occurrence of $B$ transports it from $S_3$ into $S_2$. Four occurrences of $B$ therefore suffice, allowing for the case in which the first $B$ is the very first step. The last statement holds because no step matrix has a zero row or a zero column, and the product of a strictly positive matrix with such a matrix, on either side, is strictly positive.
\end{proof}

\subsection{A symbolic characterization of simultaneously bounded type}

\begin{proposition}\label{prop:symbolic}
The scales of $(\alpha,\beta)$ satisfy $t_1(n)\le At_3(n)$ for all $n$ and some $A$ if and only if its type has bounded digits and the gaps between consecutive occurrences of the letter $B$ are bounded.
\end{proposition}

\begin{proof}
Suppose the scales satisfy $t_1(n)\le At_3(n)$ for all $n$. Digits are bounded by $A$ by Lemma \ref{lem:scales}. During a run of the letter $A$ starting at stage $n$ the third scale is fixed, while $S$ contracts by the factor $\lambda$ of Lemma \ref{lem:scales} at each step. Since $t_1(m)\ge t_3(m)=t_3(n)$ and $t_1(m)\le S(m)\le S(n)\lambda^{m-n}\le 3A\,t_3(n)\lambda^{m-n}$, the run has length at most $\log(3A)/\log(1/\lambda)$.

Conversely suppose that the digits are bonded by $a^*$ and that the gaps between occurrences of $B$ are at most $L$. Then every block of $4L$ consecutive steps contains at least four occurrences of $B$, so its matrix product $B$ is strictly positive by Lemma \ref{lem:positivity}, and its entries lie between $1$ and $(a^*+1)^{4L}$. For every $n$ we have $t(n)=B\,t(n+4L)$ with such a block product, so every coordinate of $t(n)$ lies between the sum of the coordinates of $t(n+4L)$ and $3(a^*+1)^{4L}$ times that sum. Hence $t_1(n)\le 3(a^*+1)^{4L}t_3(n)$ for every $n$.
\end{proof}

\begin{remark}\label{rem:periodictype}
If a vector has type of the form $(a_n,\varepsilon_n)_n$ where $a_n=a$, for some integer $a \geq 1$, and $\varepsilon_n=A$ for every $n$ we say that it has purely periodic type and write it as $(a,A)^{\infty}$. Similarly, for the constant case $B$ we denote the type with $(a,B)^{\infty}$.
\end{remark}

\begin{corollary}\label{cor:cantor}
For every $a\ge1$ the vector of purely periodic type $(a,B)^{\infty}$ is of simultaneously bounded type and summably convergent. For fixed $a^{\ast}$ and $L$ the set $\mathcal{K}(a^{\ast},L)$ of vectors in $\Delta$ whose type has digits at most $a^{\ast}$ and gaps between occurrences of $B$ at most $L$ is a compact subset of $\Delta$ of Lebesgue measure zero. By Proposition \ref{prop:symbolic} the set of vectors of simultaneously bounded type with constant $A$ is contained in $\mathcal{K}(a^{\ast},L)$ for suitable $a^{\ast}=a^{\ast}(A)$ and $L=L(A)$, and $\mathcal{K}(a^{\ast},L)$ is contained in the set of vectors of simultaneously bounded type with constant $3(a^{\ast}+1)^{4L}$.
\end{corollary}

\begin{proof}
The matrix $M_B(a)$ has characteristic polynomial $x^3-ax^2-1$, which is irreducible over $\Q$ because it has no integer root, and $M_B(a)^4$ is strictly positive by Lemma \ref{lem:positivity}. Its Perron eigenvalue $\lambda$ exceeds $a$, and the remaining eigenvalues have sum $a-\lambda<0$ and product $\lambda^{-1}$. Since $f(x)=x^3-ax^2-1$ satisfies $f(a)<0<f(a+a^{-2})$, we have $\lambda<a+a^{-2}$, hence $(a-\lambda)^2<a^{-4}<4/(a+1)<4/\lambda$. So the discriminant of the quadratic factor is negative and the eigenvalues are nonreal, and Theorem \ref{thm:periodic} below gives the first statement.

For the second statement let $\mathcal{S}$ be the set of sequences in $\big(\{1,\dots,a^{\ast}\}\times\{A,B\}\big)^{\N}$ in which the gaps between occurrences of $B$, including the gap before the first occurrence, are at most $L$. It is a compact subshift. For $s\in\mathcal{S}$ let $C_n(s)=M_{\varepsilon_1}(a_1)\cdots M_{\varepsilon_n}(a_n)(D)$. The closed cones $\overline{C_n(s)}$ are nested, and their intersections with the unit sphere are nonempty compact sets whose diameters tend to zero by the proof of Theorem \ref{thm:reconstruction} below, so $\bigcap_n\overline{C_n(s)}$ is a single ray $\R_+\omega(s)$. We claim that $\omega(s)$ lies in every open cone $C_n(s)$. Every block of $4L$ consecutive symbols of $s$ contains at least four occurrences of $B$, so the product of the step matrices over it is strictly positive by Lemma \ref{lem:positivity}, and the computation in the proof of Theorem \ref{thm:periodic} shows that a product $M_{\varepsilon}(a)M_{\varepsilon'}(a')B''$ with $B''$ strictly positive maps $\overline{D}\setminus\{0\}$ into $D$. Hence $\overline{C_{n+4L+2}(s)}\setminus\{0\}\subset C_n(s)$, which proves the claim. By Lemma \ref{lem:cylinders} the expansion of $\omega(s)$ is then $s$, and by Theorem \ref{thm:reconstruction} every vector of $\mathcal{K}(a^{\ast},L)$ is of the form $\omega(s)$. The map $s\mapsto\omega(s)$, normalized so as to take values in $\Delta$, is continuous, because two sequences agreeing on their first $4Lk$ symbols give vectors in a common cone of projective diameter at most $2\log K\tanh(\log K/2)^{k-1}$, with $K=(a^{\ast}+1)^{4L}$, see the proof of Theorem \ref{thm:reconstruction}. So $\mathcal{K}(a^{\ast},L)=\omega(\mathcal{S})$ is compact. The two inclusions are the two halves of Proposition \ref{prop:symbolic}. Finally, every vector of $\mathcal{K}(a^{\ast},L)$ satisfies $t_1(n)\le A't_3(n)$ for all $n$ with $A'=3(a^{\ast}+1)^{4L}$, so its $T$ orbit stays in the set $\{y\ge1/A'\}$, whose complement in $\Delta$ is open and nonempty, hence of positive $\nu$ measure. By the ergodicity of $T$, $\nu$ almost every orbit enters this complement, so $\nu(\mathcal{K}(a^{\ast},L))=0$, and $\mathcal{K}(a^{\ast},L)$ has Lebesgue measure zero because $\nu$ is equivalent to Lebesgue measure.
\end{proof}

\subsection{Reconstruction of the vector from the type}

\begin{theorem}\label{thm:reconstruction}
Two vectors of simultaneously bounded type with the same simultaneous renormalization type coincide. More precisely, if the digits of the type are bounded by $a^*$ and the gaps between occurrences of $B$ by $L$, then the intersection of the nested cones $M_{\varepsilon_1}(a_1)\cdots M_{\varepsilon_n}(a_n)(D)$, $n\ge1$, is a single ray, and $(1,\alpha,\beta)$ spans it.
\end{theorem}

\begin{proof}
By Lemma \ref{lem:cylinders} the vector $(1,\alpha,\beta)$ lies in every cone $C_n=M_{\varepsilon_1}(a_1)\cdots M_{\varepsilon_n}(a_n)(D)$, and these cones are nested. Let $d$ denote the Hilbert projective metric on the open positive cone $C^+$. Recall that a strictly positive matrix $B$ maps $C^+$ into itself and contracts $d$ by a factor $\tanh(\Delta(B)/4)<1$, where $\Delta(B)=\sup_{x,y\in C^+}d(Bx,By)$ is the projective diameter of the image, which is finite for strictly positive $B$, see Birkhoff \cite{Birkhoff}. If the entries of $B$ lie in an interval $[1,K]$ then $\Delta(B)\le 2\log K$, so the contraction factor is bounded by $\tanh(\log K/2)<1$. Group the expansion into consecutive blocks of $4L$ steps. By Lemma \ref{lem:positivity} and the bound on the digits every block product is strictly positive with entries in $[1,(a^*+1)^{4L}]$. The $d$ diameter of $D$ is infinite, but the diameter of the image of $D$ under a strictly positive matrix is finite and bounded by $2\log K$. Hence the projective diameters of the cones $C_{4Lk}$ are bounded by $2\log K\cdot\tanh(\log K/2)^{k-1}$, which tends to zero. Two rays lying in all the $C_n$ therefore coincide.
\end{proof}

\subsection{Periodic types and Perron eigenvectors}

For a finite word $w=(a_1,\varepsilon_1)\cdots(a_p,\varepsilon_p)$ we write $B_w=M_{\varepsilon_1}(a_1)\cdots M_{\varepsilon_p}(a_p)$ for its period matrix. It is a nonnegative integer matrix of determinant $\pm1$.

\begin{theorem}\label{thm:periodic}
Let $w$ be a finite word containing the letter $B$, so that $B_w^4$ is strictly positive by Lemma \ref{lem:positivity}, and let $\lambda$ be the Perron eigenvalue of $B_w$ with Perron eigenvector $\omega=(\omega_1,\omega_2,\omega_3)$, all coordinates positive. Then $\lambda>1$, $\omega\in D$, the vector $(1,\omega_2/\omega_1,\omega_3/\omega_1)$ has purely periodic type $w^\infty$, its scale vectors satisfy $t(kp)=\lambda^{-k}t(0)$, and its scales are uniformly comparable. Conversely every vector with purely periodic type $w^\infty$ is proportional to the Perron eigenvector of $B_w$. Let $\mu_2$ and $\mu_3$ be the two remaining eigenvalues of $B_w$, so that $\lambda|\mu_2||\mu_3|=1$. The three eigenvalues are distinct. The vector is summably convergent in the sense of Definition \ref{def:sc} if and only if $\lambda\min(|\mu_2|,|\mu_3|)>1$, that is, if and only if $\max(|\mu_2|,|\mu_3|)<1$. This holds in particular whenever $\mu_2$ and $\mu_3$ are nonreal, since then $|\mu_2|=|\mu_3|=\lambda^{-1/2}$ and $t_1(kp+j)H(kp+j)$ is comparable to $\lambda^{-k/2}$. The quantity $t_1(n)H(n)^2$ stays bounded if and only if $|\mu_2|=|\mu_3|$. The same statements hold for every vector with eventually periodic type whose periodic block has these properties. If the characteristic polynomial of $B_w$ is irreducible over $\Q$, then $1$, $\omega_2/\omega_1$, $\omega_3/\omega_1$ are linearly independent over $\Q$ and span a real cubic field. The eigenvalue $\lambda$ is a unit of this field. It is a Pisot number whenever the two remaining eigenvalues of $B_w$ lie in the open unit disk, in particular whenever they are nonreal.
\end{theorem}

\begin{proof}
By Lemma \ref{lem:cylinders} the matrix $B_w$ maps $D$ into $D$, so by continuity it maps the closed cone $\overline{D}$ into itself, and the induced map on the compact triangle $\overline{D}/\R_+$ has a fixed point by the Brouwer theorem. Let $t\in\overline{D}\setminus\{0\}$ be a corresponding eigenvector, $B_wt=\lambda't$ with $\lambda'>0$. All coordinates of $B_w^4t=\lambda'^4t$ are positive because $B_w^4$ is strictly positive and $t$ is nonzero and nonnegative, so $t$ has positive coordinates, and it is the Perron eigenvector of $B_w$ by the uniqueness of the positive eigenvector of a primitive matrix. Hence $\lambda'=\lambda$ and $\omega\in\overline{D}$ with positive coordinates. Since $B_w^4$ is an integer matrix with all entries at least one, its Perron eigenvalue $\lambda^4$ is at least three, so $\lambda>1$.

We show that $\omega$ lies in the open cone $D$. Put
\[
s\ =\ M_{\varepsilon_2}(a_2)\cdots M_{\varepsilon_p}(a_p)\,\omega/\lambda ,
\]
so that $\omega=M_{\varepsilon_1}(a_1)s$, and note that $s$ lies in $\overline{D}$ and has positive coordinates. In case $\varepsilon_1=A$ we have $\omega_1-\omega_2=(a_1-1)s_1+s_2>0$ and $\omega_2-\omega_3=s_1-s_3\ge0$, with equality only if $s_1=s_2=s_3$. In case $\varepsilon_1=B$ we have $\omega_1-\omega_2=(a_1-1)s_1+s_3>0$ and $\omega_2-\omega_3=s_1-s_2\ge0$, with equality only if $s_1=s_2$. In both cases equality forces $s_1=s_2$. But $s=M_{\varepsilon_2}(a_2)r$ with $r\in\overline{D}$ of positive coordinates, and the same computation gives $s_1-s_2>0$ if $p\ge2$. If $p=1$ then $s=\omega/\lambda$ and $s_1=s_2$ would give $\omega_1=\omega_2$, so $\omega_1-\omega_2=0$, which contradicts the strict inequality already established. Hence $\omega\in D$.

Since $\omega\in B_w(D)$ and $B_w(D)$ is the cylinder of the word $w$ by Lemma \ref{lem:cylinders}, the first $p$ symbols of the expansion of $\omega$ form the word $w$, and the scale vector after $p$ steps is $B_w^{-1}\omega=\lambda^{-1}\omega$, which is proportional to $\omega$. So the type is $w^\infty$ and $t(kp)=\lambda^{-k}t(0)$. Comparability of the scales holds at every stage with constant $\max_{0\le n<p}t_1(n)/t_3(n)$. The three eigenvalues of $B_w$ are distinct: a repeated root of the characteristic polynomial $\chi$ would make $\chi$ reducible over $\Q$, hence $\chi$ would have a rational root, which divides the constant term $\pm1$. So $\pm1$ would be a root, and then either $\pm1$ is the repeated root and $\lambda=\pm1$, or the other two roots $\lambda$ and $\mu$ satisfy $\lambda\mu=\pm1$ with $\lambda>1$, so they are distinct. Both cases are impossible. For the heights, the exponent vectors at stage $kp+j$ are read off the rows of $B_{kp+j}^{-1}=B_j'^{-1}B_w^{-k}$ with $B_j'$ one of finitely many initial products, so $H(kp+j)$ is bounded above and below by constant multiples of $\|B_w^{-k}\|$. Since $B_w^{-1}$ is diagonalizable with eigenvalues $\lambda^{-1},\mu_2^{-1},\mu_3^{-1}$, the norm $\|B_w^{-k}\|$ is bounded above and below by constant multiples of $\min(|\mu_2|,|\mu_3|)^{-k}$. Since $t_1(kp+j)$ is comparable to $\lambda^{-k}$, the terms $t_1(n)H(n)$ are comparable to $\big(\lambda\min(|\mu_2|,|\mu_3|)\big)^{-\lfloor n/p\rfloor}$, so the series $\sum_nt_1(n)H(n)$ converges if and only if $\lambda\min(|\mu_2|,|\mu_3|)>1$. Since $\lambda|\mu_2||\mu_3|=1$, this is the same as $\max(|\mu_2|,|\mu_3|)<1$. When $\mu_2,\mu_3$ are nonreal they are complex conjugate, so $|\mu_2|=|\mu_3|=\lambda^{-1/2}<1$ and $\lambda\min(|\mu_2|,|\mu_3|)=\lambda^{1/2}$. Similarly $t_1(n)H(n)^2$ is comparable to $\big(\lambda\min(|\mu_2|,|\mu_3|)^2\big)^{-\lfloor n/p\rfloor}$, which stays bounded if and only if $\min(|\mu_2|,|\mu_3|)^2\ge\lambda^{-1}=|\mu_2||\mu_3|$, that is, if and only if $|\mu_2|=|\mu_3|$. Conversely if a vector has type $w^\infty$ then $t(0)=B_wt(p)$ and $t(p)$ has the same type, so $t(0)$ and $t(p)$ lie on the same ray by Theorem \ref{thm:reconstruction}, whose hypotheses are satisfied because a periodic type has bounded digits and bounded gaps. Thus $t(0)$ is a positive eigenvector of $B_w$, hence the Perron eigenvector. A vector with eventually periodic type has, after finitely many steps, a purely periodic type, and comparability at the finitely many initial stages is automatic.

Suppose the characteristic polynomial $\chi$ of $B_w$ is irreducible over $\Q$. Then $\lambda$ has degree three and $\omega$ can be chosen with coordinates in $\Q(\lambda)$. If $c_1\omega_1+c_2\omega_2+c_3\omega_3=0$ with rational $c_i$ not all zero, then applying the two nontrivial embeddings of $\Q(\lambda)$ shows that the three eigenvectors of $B_w$ lie in the rational plane $c_1x_1+c_2x_2+c_3x_3=0$, which is impossible since they span $\C^3$. So $1$, $\omega_2/\omega_1$, $\omega_3/\omega_1$ are rationally independent, and they lie in the cubic field $\Q(\lambda)$. Since $\det B_w=\pm1$, the constant term of $\chi$ is $\pm1$ and $\lambda$ is a unit. The remaining two eigenvalues have product $\pm\lambda^{-1}$. If they are nonreal they are complex conjugate of modulus $\lambda^{-1/2}<1$, and $\lambda$ is a Pisot number. The same conclusion holds whenever both lie in the open unit disk.
\end{proof}

\begin{remark}\label{rem:pisotnot}
Theorem \ref{thm:periodic} is stated so as to cover period matrices with a real conjugate of modulus larger than one, but we do not know whether such period matrices exist. Let $\chi(x)=\det(xI-B_w)$. Since $\lambda>1$ and the product of the two remaining eigenvalues has modulus $\lambda^{-1}<1$, the matrix $B_w$ is Pisot if and only if no eigenvalue other than $\lambda$ is real of modulus at least one, that is, if and only if $\chi(1)<0$ and $\chi(-1)<0$. As $\chi(1)=-\det(B_w-I)$ and $\chi(-1)=-\det(B_w+I)$, the Pisot property is equivalent to the two inequalities
\[
\det(B_w-I)\ >\ 0\qquad\text{and}\qquad\det(B_w+I)\ >\ 0 .
\]
The inequalities can be sharp. For $w=(a,B)(1,B)$ one has $\det(B_w+I)=1$ for every $a$, and the characteristic polynomial $x^3-ax^2-(a+1)x-1$ has, for large $a$, a real root $-1+(a+2)^{-1}+O(a^{-2})$, so the second eigenvalue tends to $-1$. Whether every period matrix is Pisot is left open. If it is, every vector with eventually periodic type satisfies both hypotheses of Section \ref{sec:operator}. In the examples of Section \ref{sec:periodic} the remaining eigenvalues are nonreal in the first two cases and real in the third, and the Perron eigenvalues are Pisot numbers in all three.
\end{remark}

\begin{remark}\label{rem:asymmetry}
The dictionary with the rank one theory is complete at the topological level and breaks at the arithmetic level. In rank one the renormalization type is a complete arithmetic invariant. Bounded type, the Diophantine classes and the Liouville numbers are all read off the continued fraction digits, so the symbolic dynamics of renormalization carries the arithmetic. In rank two the type determines the vector by Theorem \ref{thm:reconstruction}, but the arithmetic class is not readable from the symbols, as Remark \ref{rem:sbtarith} explains. Section \ref{sec:beyond} introduces the space of lattices to repair this.
\end{remark}

\begin{remark}\label{rem:sbt_and_summability}
For a vector with eventually periodic type the two hypotheses of Section \ref{sec:operator} read as follows. Simultaneously bounded type always holds, with constant the maximum of $t_1(n)/t_3(n)$ over the preperiod and one period, since $t(n+p)=\lambda^{-1}t(n)$ along the periodic part. Summable convergence holds if and only if $\lambda$ is a Pisot number, that is, if and only if the two non-Perron eigenvalues of the period matrix lie in the open unit disk. In that case $t_1(n)H(n)$ decays at the exponential rate $\big(\lambda\min(|\mu_2|,|\mu_3|)\big)^{-n/p}$. In the examples of Section \ref{sec:periodic} the non-Perron eigenvalues are nonreal, so both hypotheses hold. The stronger property that $t_1(n)H(n)^2$ stays bounded, which is a reducedness property of the Brun basis discussed in Section \ref{sec:beyond}, holds exactly when $|\mu_2|=|\mu_3|$, in particular in the nonreal case, and it fails for totally real Pisot period matrices, which exist. For instance the word $(1,B)(1,A)(1,B)(1,A)$ has a strictly positive period matrix with eigenvalues approximately $5.049$, $0.643$ and $0.308$, so the corresponding vector is of simultaneously bounded type and summably convergent while $t_1(n)H(n)^2$ is unbounded along its expansion.
\end{remark}

\begin{remark}\label{rem:mixedpisot}
The set of vectors that are both of simultaneously bounded type and summably convergent is much larger than the countable set of eventually periodic types with Pisot period matrix. It contains a Cantor set. Let $w_1,\dots,w_r$, $r\ge2$, be finite words containing the letter $B$ whose period matrices $B_{w_i}$ have Pisot Perron eigenvalue, for instance the words $(1,B)$ and $(1,A)(1,B)(1,A)$ of Section \ref{sec:periodic}. Then there is $k_0$ such that every sequence
\[
s\ =\ w_{i_1}^{k_1}\,w_{i_2}^{k_2}\,w_{i_3}^{k_3}\cdots,\qquad i_j\in\{1,\dots,r\},\quad k_j\ge k_0 ,
\]
is the type of exactly one vector $\omega(s)\in\Delta$, and $\omega(s)$ is of simultaneously bounded type and summably convergent. The vectors obtained with all $k_j=k_0$ already form a Cantor set, the image of the full shift on $r$ symbols under the continuous injective map $s\mapsto\omega(s)$, and all but countably many of them have a type that is not eventually periodic. Theorems \ref{thm:main} and \ref{thm:rigidity} therefore apply to a Cantor set of rotation vectors, not only to the eventually periodic ones.

Indeed, every word $w_i$ contains the letter $B$, so $s$ has digits bounded by the largest digit of the words and gaps between occurrences of $B$ bounded by twice the largest length of the words. By the proof of Corollary \ref{cor:cantor} such a sequence is the type of a unique vector $\omega(s)$, which is of simultaneously bounded type by Proposition \ref{prop:symbolic}, and $s\mapsto\omega(s)$ is continuous and injective. For summable convergence fix the $j$th run, write $w=w_{i_j}$, $p=|w|$, let $\lambda>1$ be the Perron eigenvalue of $B_w$ and $\mu_2,\mu_3$ its other eigenvalues, and put $\gamma_w=\lambda\min(|\mu_2|,|\mu_3|)=\max(|\mu_2|,|\mu_3|)^{-1}>1$, by Theorem \ref{thm:periodic}. Let $n_{j-1}$ and $n_j=n_{j-1}+k_jp$ be the stages at which the run begins and ends, so that $t(n_{j-1}+mp)=B_w^{\,k_j-m}\,t(n_j)$ for $0\le m\le k_j$. We use the norm $\|t\|_\infty$, which equals $t_1$ on $\overline D$. Since $B_w$ is diagonalizable with spectral radius $\lambda$, $\|B_w^{k}t\|_\infty\le C\lambda^{k}\|t\|_\infty$, and if $\ell$ is the positive left Perron eigenvector of $B_w$, then $\ell(B_w^{k}t)=\lambda^{k}\ell(t)$ gives $\|B_w^{k}t\|_\infty\ge c\lambda^{k}\|t\|_\infty$ for $t\ge0$, with $c=\min_i\ell_i/\sum_i\ell_i$. Applying the two bounds with $k=k_j-m$ and $k=k_j$ yields
\[
t_1(n_{j-1}+mp)\ \le\ (C/c)\,\lambda^{-m}\,t_1(n_{j-1}),\qquad 0\le m\le k_j .
\]
The rows of $B_n^{-1}$ consist of the exponent vectors and of the exponents of $T$, and the latter are bounded by $H(n)+1$ because $t_i(n)\in(0,1)$, so $H(n)\le\|B_n^{-1}\|_\infty\le3H(n)$. Since $B_{n_{j-1}+mp}^{-1}=B_w^{-m}B_{n_{j-1}}^{-1}$ and $B_w^{-1}$ is diagonalizable with spectral radius $\min(|\mu_2|,|\mu_3|)^{-1}$,
\[
H(n_{j-1}+mp)\ \le\ 3C\,\min(|\mu_2|,|\mu_3|)^{-m}\,H(n_{j-1}) .
\]
Multiplying the two displays, $t_1(n)H(n)\le C_1\gamma_w^{-m}\,t_1(n_{j-1})H(n_{j-1})$ for $n=n_{j-1}+mp$, with $C_1$ depending only on the words. At the intermediate stages $n_{j-1}+mp+i$, $0<i<p$, the scale $t_1$ is not larger and $H$ is larger by at most a factor $C_2$, the largest norm of a product of fewer than $p$ inverse step matrices with digits from the words. Choose $k_0$ with $C_1\gamma_{w_i}^{-k_0}\le\tfrac12$ for all $i$. Then $t_1(n_j)H(n_j)\le\tfrac12\,t_1(n_{j-1})H(n_{j-1})$, hence $t_1(n_j)H(n_j)\le2^{-j}\,t_1(0)H(0)$, and the stages of the $j$th run contribute at most $pC_1C_2\,2^{1-j}\,t_1(0)H(0)\sum_{m\ge0}\gamma_w^{-m}$ to $\Sigma$. Summing over $j$ gives $\Sigma<\infty$.
\end{remark}

\section{Tower tilings}\label{sec:towers}

In the rank one theory the distortion estimates rest on the dynamical partitions, whose intervals are pairwise disjoint because renormalization is a first return map. In the simultaneous setting the substitute is a self similar system of tilings of the circle produced by Lemma \ref{lem:decomposition}. We work on the original circle $\T$ and identify subintervals of $[x_0,x_0+1)$ with their projections.

\begin{theorem}\label{thm:tiling}
For every $n\ge0$ there are finite subsets $\mathcal{W}_n^{(1)},\mathcal{W}_n^{(2)},\mathcal{W}_n^{(3)}\subset\Gamma$ such that
\[
[x_0,x_0+1)\ =\ \bigsqcup_{i=1}^{3}\ \bigsqcup_{w\in\mathcal{W}_n^{(i)}}\ w\big(I_n^{(i)}\big)
\]
is a partition. The families are defined recursively by $\mathcal{W}_0^{(1)}=\{\mathrm{id}\}$, $\mathcal{W}_0^{(2)}=\mathcal{W}_0^{(3)}=\emptyset$, and at each step every tile of type $1$ is decomposed according to Lemma \ref{lem:decomposition} while the tiles of types $2$ and $3$ are relabelled according to the reordering of the triple. The tiling $\mathcal{T}_{n+1}$ refines $\mathcal{T}_n$.
\end{theorem}

\begin{proof}
The proof goes by induction on $n$. For $n=0$ the statement is the trivial partition of the fundamental domain by itself, since $I_0^{(1)}=[x_0,x_0+1)$. Assume the partition holds at stage $n$. Apply Lemma \ref{lem:decomposition} to the triple $u^{(n)}$. For each $w\in\mathcal{W}_n^{(1)}$ the tile $w(I_n^{(1)})$ decomposes as the disjoint union of $w\,(u_2^{(n)})^{j}(I_n^{(2)})$ for $0\le j<a_{n+1}$ and of $w\,(u_2^{(n)})^{a_{n+1}}(I')$, where $I'$ is the interval of the new element $v_n$. Now $I_n^{(2)}=I_{n+1}^{(1)}$, the interval $I_n^{(3)}$ equals $I_{n+1}^{(3)}$ in case $A$ and $I_{n+1}^{(2)}$ in case $B$, and $I'$ is the remaining stage $n+1$ interval. Therefore the collection of all pieces of all decomposed type $1$ tiles, together with all old tiles of types $2$ and $3$ relabelled by the stage $n+1$ types of their intervals, is a partition of $[x_0,x_0+1)$ into images of the three stage $n+1$ intervals under elements of $\Gamma$. This defines the sets $\mathcal{W}_{n+1}^{(i)}$. Every stage $n+1$ tile is contained in a stage $n$ tile by construction.
\end{proof}

\begin{lemma}\label{lem:tilecount}
Let $N_i(n)$ be the number of tiles of type $i$ at stage $n$, and let $N(n)$ be the column vector with entries $N_1(n)$, $N_2(n)$, $N_3(n)$. Then
\[
N(n+1)\ =\ M_{\varepsilon_{n+1}}(a_{n+1})^{\top}N(n),
\]
so that $N(n)$ is the first row of $M_{\varepsilon_1}(a_1)\cdots M_{\varepsilon_n}(a_n)$, written as a column. The identities
\[
\sum_{i=1}^3 N_i(n)\,t_i(n)\ =\ 1,\qquad \sum_{i=1}^3\ \sum_{w\in\mathcal{W}_n^{(i)}}\big|w(I_n^{(i)})\big|\ =\ 1
\]
hold at every stage.
\end{lemma}

\begin{proof}
In the recursion each type $1$ tile produces $a_{n+1}$ tiles over $I_{n+1}^{(1)}$ and one tile over $I'$, and the old type $2$ tiles become type $1$. In case $A$ the tiles over $I'$ are of type $2$ and the old type $3$ tiles remain of type $3$, giving $N_1'=a_{n+1}N_1+N_2$, $N_2'=N_1$, $N_3'=N_3$, which is multiplication by $M_A(a_{n+1})^{\top}$, a symmetric matrix. In case $B$ the old type $3$ tiles become type $2$ and the tiles over $I'$ are of type $3$, giving $N_1'=a_{n+1}N_1+N_2$, $N_2'=N_3$, $N_3'=N_1$, which is multiplication by $M_B(a_{n+1})^{\top}$. The first identity is the statement that $\mu$, which assigns mass $t_i(n)$ to each tile of type $i$, has total mass one on the circle. The second identity is the partition of the fundamental domain.
\end{proof}

\begin{lemma}\label{lem:factorization}
Every $w\in\mathcal{W}_n^{(i)}$ is a product
\[
w\ =\ \big(u_2^{(0)}\big)^{j_0}\big(u_2^{(1)}\big)^{j_1}\cdots\big(u_2^{(n-1)}\big)^{j_{n-1}},\qquad 0\le j_m\le a_{m+1},
\]
and for $J\subset I_n^{(i)}$ the image $\big(u_2^{(m)}\big)^{j_m}\cdots\big(u_2^{(n-1)}\big)^{j_{n-1}}(J)$ lies in $I_m^{(1)}$ for every $m\le n$. Within a fixed stage $m$ the pieces $\big(u_2^{(m)}\big)^{k}(J_m)$, $0\le k<j_m$, where $J_m$ denotes the image of $J$ under the factors of stages larger than $m$, are pairwise disjoint subintervals of $I_m^{(1)}$, contained in the pairwise disjoint intervals $\big(u_2^{(m)}\big)^{k}(I_m^{(2)})$ of Lemma \ref{lem:decomposition}.
\end{lemma}

\begin{proof}
By the recursion of Theorem \ref{thm:tiling} every word of stage $n+1$ is either an old word of stage $n$ of type $2$ or $3$, unchanged, or a word $w'(u_2^{(n)})^{j}$ with $w'\in\mathcal{W}_n^{(1)}$ and $0\le j\le a_{n+1}$. Induction on $n$ gives the factorization. For the statement about images we argue as follows. The base interval $I_{n+1}^{(i)}$ is contained in $I_n^{(1)}$, and $(u_2^{(n)})^{j}$ maps $I_{n+1}^{(1)}=I_n^{(2)}$ onto the piece $(u_2^{(n)})^{j}(I_n^{(2)})\subset I_n^{(1)}$ for $j<a_{n+1}$ and maps $I'$ onto $(u_2^{(n)})^{a_{n+1}}(I')\subset I_n^{(1)}$ for $j=a_{n+1}$, by Lemma \ref{lem:decomposition}. So the image under the finest factor lies in $I_n^{(1)}$, and the claim for the coarser factors follows by induction, since $w'\in\mathcal{W}_n^{(1)}$ is a word of the previous stage applied to a subset of $I_n^{(1)}$. Disjointness within a stage is the disjointness in Lemma \ref{lem:decomposition}.
\end{proof}

\section{Real bounds}\label{sec:realbounds}

\subsection{Distortion and nonlinearity}
For a $C^1$ diffeomorphism $\varphi$ of an interval $J$ onto its image put
\[
\Dist(\varphi|_J)\ =\ \sup_{x,y\in J}\log\frac{D\varphi(x)}{D\varphi(y)} ,
\]
and for a $C^2$ diffeomorphism let $\eta_\varphi=D^2\varphi/D\varphi=D\log D\varphi$ be its nonlinearity. The following holds
\[
\Dist(\varphi|_J)\le\int_J|\eta_\varphi|,\quad
\Dist(\varphi\circ\psi|_J)\le\Dist(\psi|_J)+\Dist(\varphi|_{\psi(J)}),\quad
\Dist(\varphi^{-1}|_{\varphi(J)})=\Dist(\varphi|_J),
\]
and the chain rule is $\eta_{\varphi\circ\psi}=(\eta_\varphi\circ\psi)\,D\psi+\eta_\psi$. Every element of $\Gamma$ commutes with $T$, so $\eta_W$ is $1$ periodic for $W\in\Gamma$. Set
\[
E\ =\ \int_\T|\eta_F|+\int_\T|\eta_G|,\qquad
E'\ =\ \max_{\ell}\sup_\T|\eta_\ell|,\qquad
E''\ =\ \max_{\ell}\sup_\T|D\eta_\ell| ,
\]
where $\ell$ runs over the four letters $F^{\pm1},G^{\pm1}$. All three constants are finite because $f$ and $g$ are of class $C^3$. The third constant is used only in Section \ref{sec:moebius}.

Commutativity gives the following identity which is used repeatedly. If $U,W\in\Gamma$ then $D(UW)=D(WU)$, so
\begin{equation}\label{eq:symmetry}
\log DU(Wx)-\log DU(x)\ =\ \log DW(Ux)-\log DW(x)\qquad\text{for all } x\in\R .
\end{equation}
Equivalently $\int_{[x,Wx)}\eta_U=\int_{[x,Ux)}\eta_W$ whenever both translation numbers are positive.

\begin{lemma}\label{lem:zeromean}
For every $W\in\Gamma$ one has $\int_\T\log DW\,d\mu=0$. In particular $\log DW$ vanishes at some point of $\T$.
\end{lemma}

\begin{proof}
The element $W$ is the lift of a circle diffeomorphism $\bar W$ whose rotation number $\tau(W)$ modulo one is irrational, so $\bar W$ is uniquely ergodic with invariant measure $\mu$. The function $\log D\bar W$ is continuous, so by unique ergodicity $\tfrac1n\log D\bar W^n=\tfrac1n\sum_{j<n}\log D\bar W\circ\bar W^j$ converges uniformly to $c=\int\log D\bar W\,d\mu$. If $c>0$ then $D\bar W^n\ge e^{nc/2}$ on $\T$ for large $n$, which contradicts $\int_\T D\bar W^n\,dx=1$. If $c<0$ the same argument applies to $\bar W^{-1}$. The last statement follows since $\log D\bar W$ is continuous with vanishing $\mu$ mean and $\mu$ has full support.
\end{proof}

\subsection{Total nonlinearity, enlarged intervals and the choice of base point}
The Denjoy estimate of the rank one theory rests on the disjointness of the intervals $f^j(I_n)$, $0\le j<q_n$, of the dynamical partition. In rank two the images of an interval under the prefixes of a word are not disjoint, and their multiplicity is not bounded. We replace disjointness by three elementary facts. The first bounds the total nonlinearity of a word over the whole circle by its length. The second is a nesting property of of certain enlargements of the fundamental intervals, which keeps all the pieces produced by the chain rule inside an interval of $\mu$ mass at most four times the mass of the fundamental domain. The third averages the nonlinearity over these intervals with respect to the base point, which is at our disposal.

\begin{lemma}\label{lem:totalnl}
Let $W\in\Gamma$ be a word of length $L$ in the letters $F^{\pm1}$, $G^{\pm1}$. Then $\int_\T|\eta_W|\le L\,E$.
\end{lemma}

\begin{proof}
Write $W=\ell_L\circ\cdots\circ\ell_1$ and $p_s=\ell_s\circ\cdots\circ\ell_1$, with $p_0$ the identity. The chain rule $\eta_{\varphi\circ\psi}=(\eta_\varphi\circ\psi)D\psi+\eta_\psi$ applied to $p_s=\ell_s\circ p_{s-1}$ gives $\eta_{p_s}=(\eta_{\ell_s}\circ p_{s-1})Dp_{s-1}+\eta_{p_{s-1}}$, and therefore
\[
\eta_W\ =\ \sum_{s=1}^{L}\big(\eta_{\ell_s}\circ p_{s-1}\big)\,Dp_{s-1}.
\]
Each $p_{s-1}$ lies in $\Gamma$, so it commutes with $T$ and descends to an orientation preserving circle diffeomorphism. The substitution $z=p_{s-1}(x)$, whose Jacobian is $Dp_{s-1}$, gives
\[
\int_\T\big|\eta_{\ell_s}(p_{s-1}x)\big|\,Dp_{s-1}(x)\,dx\ =\ \int_\T\big|\eta_{\ell_s}(z)\big|\,dz .
\]
Finally $\eta_{F^{-1}}=-(\eta_F\circ F^{-1})DF^{-1}$, so $\int_\T|\eta_{F^{-1}}|=\int_\T|\eta_F|$, and likewise for $G$. Summing over $s$ gives $\int_\T|\eta_W|\le\sum_s\int_\T|\eta_{\ell_s}|\le LE$.
\end{proof}

The word length of $u_i^{(n)}$ is $|k_i(n)|+|l_i(n)|\le H(n)$, so Lemma \ref{lem:totalnl} gives
\begin{equation}\label{eq:Nbound}
\int_\T\big|\eta_{u_i^{(n)}}\big|\ \le\ H(n)\,E\qquad\text{for } i=1,2,3 .
\end{equation}
This is the only place where the heights enter, and it enters through an integral over the whole circle rather than through a count of overlaps.

\medskip
\noindent\textit{Enlarged intervals.} For every $n$ put
\[
\Omega_n\ =\ (u_1^{(n)})^{2}\,u_2^{(n)}\,u_3^{(n)}\ \in\ \Gamma,\qquad
\widetilde I_n(x)\ =\ [x,\ \Omega_nx),\qquad \widetilde I_n\ =\ \widetilde I_n(x_0),\qquad D_n(x)\ =\ [x,\ u_1^{(n)}x),
\]
so that $D_n=D_n(x_0)=I_n^{(1)}$. Since $\tau(\Omega_n)=2t_1(n)+t_2(n)+t_3(n)$, the interval $\widetilde I_n$ has $\mu$ mass at most $4t_1(n)$. The point of these intervals is the following nesting property, which holds for every rationally independent vector.

\begin{lemma}\label{lem:windows}
For every $n\ge0$ the following hold.
\begin{enumerate}
\item[(a)] $D_n\subset\widetilde I_n$, $u_j^{(n)}(D_n)\subset\widetilde I_n$ for $j=1,2,3$, and $(u_2^{(n)})^{k}(D_n)\subset\widetilde I_n$ for $0\le k\le a_{n+1}$.
\item[(b)] $(u_2^{(n)})^{j}(\widetilde I_{n+1})\subset\widetilde I_n$ for $0\le j\le a_{n+1}$, in particular $\widetilde I_{n+1}\subset\widetilde I_n$.
\end{enumerate}
\end{lemma}

\begin{proof}
Write $u_i=u_i^{(n)}$, $t_i=t_i(n)$, $a=a_{n+1}$ and $v=v_n=u_1u_2^{-a}$, so that $u_1=u_2^{a}v$ and $\tau(v)\in(0,t_2)$. All intervals in question have left endpoint $Ux_0$ with $\tau(U)\ge0$, hence not to the left of $x_0$ by Lemma \ref{lem:forward}, so it suffices to compare right endpoints, and $Ux_0<U'x_0$ as soon as $\tau(U'U^{-1})>0$, again by Lemma \ref{lem:forward}.

(a) The right endpoint of $u_j(D_n)$ is $u_ju_1x_0$, and $\Omega_n(u_ju_1)^{-1}=u_1u_2u_3u_j^{-1}$ has translation number $t_1+t_2+t_3-t_j>0$. The right endpoint of $u_2^{k}(D_n)$ is $u_2^{k}u_1x_0$, and $\Omega_n(u_2^{k}u_1)^{-1}=u_1u_2^{1-k}u_3$ has translation number $t_1+(1-k)t_2+t_3=(a+1-k)t_2+\tau(v)+t_3>0$ for $k\le a$.

(b) The basis of stage $n+1$ is $(u_2,v,u_3)$ or $(u_2,u_3,v)$, so in both cases $\Omega_{n+1}=u_2^{2}vu_3$ and the right endpoint of $u_2^{j}(\widetilde I_{n+1})$ is $u_2^{j+2}vu_3x_0$. Since $\Omega_n=u_2^{2a}v^{2}u_2u_3$, we get $\Omega_n(u_2^{j+2}vu_3)^{-1}=u_2^{2a-1-j}v$, whose translation number is $(2a-1-j)t_2+\tau(v)\ge(a-1)t_2+\tau(v)>0$ for $j\le a$.
\end{proof}

For $x\in\R$, $n\ge0$ and $i=1,2,3$ put
\[
\mu_i(n;x)\ =\ \int_{\widetilde I_n(x)}\big|\eta_{u_i^{(n)}}\big| ,\qquad \mu_i(n)\ =\ \mu_i(n;x_0),
\]
the nonlinearity of the $i$th basis element over $\widetilde I_n(x)$. The base point was fixed arbitrarily in Section \ref{sec:operator} and no result so far depends on the choice. The next lemma shows that on average over the base point these quantities are small.

\begin{lemma}\label{lem:fubini}
Let $W\in\Gamma$ and let $U\in\Gamma$ with $\tau(U)>0$. Then
\[
\int_\T\bigg(\int_{[x,\,Ux)}\big|\eta_W\big|\bigg)\,d\mu(x)\ =\ \tau(U)\int_\T\big|\eta_W\big| .
\]
The same identity holds with $|\eta_W|$ replaced by any nonnegative measurable $1$ periodic function, for instance by $DW$.
\end{lemma}

\begin{proof}
Lift $\mu$ to a $T$ invariant measure on $\R$ and integrate $x$ over $[0,1)$. Since $U$ is increasing, $y\in[x,Ux)$ if and only if $U^{-1}y<x\le y$. By Fubini,
\[
\int_{[0,1)}\int_\R\big|\eta_W(y)\big|\,\mathbf 1_{[x,Ux)}(y)\,dy\,d\mu(x)\ =\ \int_\R\big|\eta_W(y)\big|\ \mu\big((U^{-1}y,\,y]\cap[0,1)\big)\,dy .
\]
Write $y=y'+k$ with $y'\in[0,1)$ and $k\in\Z$. Since $U$ commutes with $T$ and $\mu$ is $T$ invariant, $\mu((U^{-1}y,y]\cap[0,1))=\mu((U^{-1}y',y']\cap[-k,1-k))$, and $|\eta_W(y)|=|\eta_W(y')|$. Summing over $k$ gives $\int_0^1|\eta_W(y')|\,\mu((U^{-1}y',y'])\,dy'$, and $\mu((U^{-1}y',y'])=\tau(U)$ because $\mu$ has no atoms and assigns mass $\tau(U)$ to every interval $[z,Uz)$.
\end{proof}

\begin{proposition}\label{prop:basepoint}
Suppose that $\Sigma^{\ast}=\sum_{n\ge0}a_{n+1}t_1(n)H(n)<\infty$, which holds in particular if $(\alpha,\beta)$ is of simultaneously bounded type and summably convergent, since then $\Sigma^{\ast}\le A\Sigma$. Then
\[
\int_\T\ \sum_{n\ge0}\sum_{i=1}^{3}a_{n+1}\,\mu_i(n;x)\,d\mu(x)\ \le\ 12\,E\,\Sigma^{\ast}
\qquad\text{and}\qquad
\int_\T\ \sum_{n\ge0}\sum_{s}\big|p_s\big(\widetilde I_n(x)\big)\big|\,d\mu(x)\ \le\ 24\,\Sigma^{\ast} ,
\]
where for each $n$ the sum over $s$ runs over the prefixes $p_s$ of fixed factorizations of $u_1^{(n)}$, $u_2^{(n)}$ and $u_3^{(n)}$ into letters. In particular both series converge for $\mu$ almost every $x\in\T$.
\end{proposition}

\begin{proof}
By Lemma \ref{lem:fubini} applied with $U=\Omega_n$ and $W=u_i^{(n)}$, and by \eqref{eq:Nbound},
\[
\int_\T\mu_i(n;x)\,d\mu(x)\ =\ \tau(\Omega_n)\int_\T\big|\eta_{u_i^{(n)}}\big|\ \le\ 4E\,t_1(n)H(n).
\]
Multiplying by $a_{n+1}$, summing over $i$ and $n$ and using monotone convergence gives the first bound. For the second, $|p_s(\widetilde I_n(x))|=\int_{\widetilde I_n(x)}Dp_s$, and Lemma \ref{lem:fubini} with $U=\Omega_n$ and $Dp_s$ in place of $|\eta_W|$ gives $\int_\T|p_s(\widetilde I_n(x))|\,d\mu(x)=\tau(\Omega_n)\int_\T Dp_s=\tau(\Omega_n)\le4t_1(n)$, since $p_s$ descends to a circle diffeomorphism. Each $u_i^{(n)}$ has at most $H(n)+1\le2H(n)$ prefixes, so the three factorizations have at most $6H(n)$ prefixes in total, and the second bound follows because $a_{n+1}\ge1$. A nonnegative function with finite integral is finite almost everywhere.
\end{proof}

\begin{convention}\label{con:basepoint}
A base point $x_0\in\T$ is admissible for the pair $(f,g)$ if the two series
\[
K_0\ :=\ \sum_{n\ge0}\sum_{i=1}^{3}a_{n+1}\,\mu_i(n;x_0),\qquad
K_1\ :=\ \sum_{n\ge0}\sum_{s}\big|p_s(\widetilde I_n)\big|
\]
converge. By Proposition \ref{prop:basepoint}, when $\Sigma^{\ast}<\infty$ the admissible base points form a set of full $\mu$ measure. From now on the base point $x_0$ fixed in Section \ref{sec:operator} is assumed admissible, once and for all. All objects of Sections \ref{sec:operator} and \ref{sec:towers} are built on this base point, and all constants below depend on the pair and on the base point only through $A$, $E$, $E'$, $E''$, $K_0$ and $K_1$. The conjugacy of Lemma \ref{lem:toplin} and the measure $\mu$ do not depend on the base point.
\end{convention}

Since $D_n\subset\widetilde I_n$, the series $\sum_n\sum_s|p_s(D_n)|$ is bounded by $K_1$ as well. Two immediate consequences of the definition of $K_0$ will be used constantly. Since $a_{n+1}\ge1$ and $D_n\subset\widetilde I_n$,
\begin{equation}\label{eq:K0basic}
\int_{D_n}\big|\eta_{u_i^{(n)}}\big|\ \le\ \mu_i(n)\ \le\ K_0\qquad\text{for all } n\ge0 \text{ and } i=1,2,3 ,
\end{equation}
and $\sum_{m\ge n}a_{m+1}\mu_2(m)\to0$ as $n\to\infty$. Note also that admissibility is a statement about the tails of two convergent series, so the estimates below apply from any stage on. This is used in Section \ref{sec:topohyp}.

\subsection{The Denjoy estimate}
The estimate below and its consequences in Proposition \ref{prop:derivbound} hold at every stage and do not use bounded type. This is what makes the extension of Section \ref{sec:beyond} possible.

\begin{proposition}\label{prop:denjoy}
Suppose that $\Sigma^{\ast}<\infty$ and that the base point is admissible. For every $n$, every $i$ and every tile word $w\in\mathcal{W}_n^{(i)}$, written $w=\prod_{m<n}(u_2^{(m)})^{j_m}$ as in Lemma \ref{lem:factorization},
\[
\Dist\big(w|_{\widetilde I_n}\big)\ \le\ \sum_{m<n}j_m\,\mu_2(m)\ \le\ K_0 ,
\]
and in particular $\Dist(w|_{\widetilde I_n})\le\sum_{m\ge n_0}a_{m+1}\mu_2(m)$ whenever $w$ involves no factor of stage below $n_0$. Moreover $\Dist(u_j^{(n)}|_{\widetilde I_n})\le\mu_j(n)\le K_0$ for $j=1,2,3$, and the four branch words of $\xi_n$ have distortion at most $2K_0$ on their domains.
\end{proposition}

\begin{proof}
Let $J\subset\widetilde I_n$ be an interval. Put $J_n=J$ and $J_m=(u_2^{(m)})^{j_m}(J_{m+1})$ for $m=n-1,\dots,0$. Since $w=(u_2^{(0)})^{j_0}\circ\cdots\circ(u_2^{(n-1)})^{j_{n-1}}$ we have $w(J)=J_0$. By Lemma \ref{lem:windows}(b) and induction, $J_{m}\subset\widetilde I_{m}$ for every $m$, and $(u_2^{(m)})^{k}(J_{m+1})\subset\widetilde I_m$ for $0\le k\le j_m$. The chain rule and $\Dist(\varphi|_I)\le\int_I|\eta_\varphi|$ give
\[
\Dist(w|_J)\ \le\ \sum_{m<n}\Dist\big((u_2^{(m)})^{j_m}\big|_{J_{m+1}}\big)\ \le\ \sum_{m<n}\ \sum_{0\le k<j_m}\int_{(u_2^{(m)})^{k}(J_{m+1})}\big|\eta_{u_2^{(m)}}\big|\ \le\ \sum_{m<n}j_m\,\mu_2(m),
\]
and $j_m\le a_{m+1}$ gives the bound $K_0$. The distortion of $u_j^{(n)}$ on $\widetilde I_n$ is at most $\int_{\widetilde I_n}|\eta_{u_j^{(n)}}|=\mu_j(n)$. The branch words of $\xi_n$ are $u_j^{(n)}$ on $[x_0,c_j(n))$ and $(u_1^{(n)})^{-1}u_j^{(n)}$ on $[c_j(n),u_1^{(n)}x_0)$, $j=2,3$. The second branch maps its domain onto $[x_0,u_j^{(n)}x_0)\subset D_n$, and $u_j^{(n)}$ maps that domain onto $u_1^{(n)}([x_0,u_j^{(n)}x_0))$, so
\[
\Dist\big((u_1^{(n)})^{-1}u_j^{(n)}\big|_{[c_j(n),u_1^{(n)}x_0)}\big)\ \le\ \Dist\big(u_j^{(n)}|_{D_n}\big)+\Dist\big(u_1^{(n)}|_{[x_0,u_j^{(n)}x_0)}\big)\ \le\ \mu_j(n)+\mu_1(n)\ \le\ 2K_0 .\qedhere
\]
\end{proof}

\begin{proposition}\label{prop:derivbound}
Under the hypotheses of Proposition \ref{prop:denjoy}, for every $n$:
\begin{enumerate}
\item[(i)] $|\log Du_j^{(n)}|\le2K_0$ on $\widetilde I_n$, for $j=1,2,3$.
\item[(ii)] $|\log D(u_2^{(n)})^{j}|\le2jK_0$ on $D_n$, for $0\le j\le a_{n+1}$.
\item[(iii)] if $W\in\mathcal{W}_n^{(i)}$ is a tile word with $W(D_n)\subset\widetilde I_n$, then $|\log DW|\le2K_0$ on $\widetilde I_n$.
\end{enumerate}
\end{proposition}

\begin{proof}
Let $W$ be $u_j^{(n)}$ or a tile word as in (iii). In both cases $W(D_n)\subset\widetilde I_n$, by Lemma \ref{lem:windows}(a) or by hypothesis, and $\Dist(W|_{\widetilde I_n})\le K_0$ by Proposition \ref{prop:denjoy}. By Lemma \ref{lem:zeromean} there is $z\in\T$ with $\log DW(z)=0$, and by Theorem \ref{thm:tiling} there are $i$ and a tile word $w\in\mathcal{W}_n^{(i)}$ with $z\in w(I_n^{(i)})$. Put $x=w^{-1}(z)\in I_n^{(i)}\subset D_n$. Identity \eqref{eq:symmetry} with $U=W$ gives
\[
\log DW(wx)-\log DW(x)\ =\ \log Dw(Wx)-\log Dw(x),
\]
and the left hand side equals $-\log DW(x)$. Both $x$ and $Wx$ lie in $\widetilde I_n$, so $|\log DW(x)|\le\Dist(w|_{\widetilde I_n})\le K_0$. Adding the distortion bound gives $|\log DW|\le2K_0$ on $\widetilde I_n$, which proves (i) and (iii). For (ii), let $y\in D_n$ and $j\le a_{n+1}$. Then $\log D(u_2^{(n)})^{j}(y)=\sum_{k<j}\log Du_2^{(n)}\big((u_2^{(n)})^{k}y\big)$, and each point $(u_2^{(n)})^{k}y$ lies in $\widetilde I_n$ by Lemma \ref{lem:windows}(a), so (i) gives the bound $2jK_0$.
\end{proof}

\subsection{Internal ratios and the decay of the mesh}
Put
\[
r_i(n)\ =\ \frac{|I_n^{(i)}|}{|D_n|},\qquad i=1,2,3,
\]
so $r_1(n)=1$ and $r_3(n)\le r_2(n)$. In this subsection $(\alpha,\beta)$ is of simultaneously bounded type with constant $A$ and summably convergent, the base point is admissible, and $L=L(A)$ is an integer such that every $L$ consecutive letters of the type contain the letter $B$. Such an $L$ exists by the proof of Proposition \ref{prop:symbolic}, one more than the maximal length of a run of the letter $A$.

\begin{lemma}\label{lem:participation}
Within any $L+2$ consecutive stages each basis element is replaced. More precisely $I_n^{(3)}=D_{n'}$ for some $n'$ with $n<n'\le n+L+2$, and $I_n^{(2)}=D_{n+1}$.
\end{lemma}

\begin{proof}
The second identity is the definition of the step. For the first, $u_3^{(n)}$ stays in third position exactly as long as the letter $A$ is chosen. At the first occurrence of $B$, which happens within $L$ steps, it moves to second position, and one step later it is the first element. So $u_3^{(n)}=u_1^{(n')}$ for some $n'\le n+L+2$, and the corresponding intervals coincide.
\end{proof}

\begin{proposition}\label{prop:ratios}
Put
\[
\rho_1\ =\ \frac{e^{-2AK_0}}{A+1},\qquad \rho_2\ =\ \big(1+e^{-2AK_0}\rho_1^{\,L+2}\big)^{-1},\qquad \rho_3\ =\ e^{-K_0}\rho_1^{\,L+4},\qquad \sigma\ =\ (1-\rho_3)^{1/(L+3)} .
\]
Then $0<\rho_1\le\rho_2<1$, and for every $n$
\[
\rho_1\ \le\ r_2(n)\ \le\ \rho_2,\qquad r_3(n)\ \ge\ \rho_1^{\,L+2} .
\]
Moreover every stage $n+1$ tile contained in a stage $n$ tile of type $1$ occupies a fraction of it between $\rho_3$ and $1-\rho_3$, and consequently
\[
T_{\max}(n)\ :=\ \max\big\{|w(I_n^{(i)})|\colon i=1,2,3,\ w\in\mathcal{W}_n^{(i)}\big\}\ \le\ (1-\rho_3)^{-1}\sigma^{\,n}\qquad\text{for all } n .
\]
\end{proposition}

\begin{proof}
By Lemma \ref{lem:decomposition} the interval $D_n$ is the disjoint union of the $a_{n+1}$ intervals $(u_2^{(n)})^{j}(I_n^{(2)})$, $0\le j<a_{n+1}$, and of $(u_2^{(n)})^{a_{n+1}}(I')$, where $I'=[x_0,v_nx_0)\subset I_n^{(2)}$. By Proposition \ref{prop:derivbound}(ii) and $a_{n+1}\le A$, each of the maps $(u_2^{(n)})^{j}$ with $j\le a_{n+1}$ has $|\log D|\le2AK_0$ on $D_n$. Hence
\[
|D_n|\ \le\ (a_{n+1}+1)\,e^{2AK_0}\,|I_n^{(2)}|\ \le\ (A+1)\,e^{2AK_0}\,|I_n^{(2)}|,
\]
which gives $r_2(n)\ge\rho_1$. By Lemma \ref{lem:participation} the interval $I_n^{(3)}$ equals $D_{n'}$ with $n<n'\le n+L+2$, and $|D_{m+1}|/|D_m|=r_2(m)\ge\rho_1$ for every $m$, so $r_3(n)=|D_{n'}|/|D_n|\ge\rho_1^{L+2}$. The interval $I'$ is one of $I_{n+1}^{(2)}$, $I_{n+1}^{(3)}$, so $|I'|\ge\rho_1^{L+2}|D_{n+1}|=\rho_1^{L+2}|I_n^{(2)}|$. Hence
\[
|D_n|\ \ge\ |I_n^{(2)}|+\big|(u_2^{(n)})^{a_{n+1}}(I')\big|\ \ge\ |I_n^{(2)}|\,\big(1+e^{-2AK_0}\rho_1^{\,L+2}\big),
\]
which gives $r_2(n)\le\rho_2$. The same estimates show that each of the $a_{n+1}+1$ pieces of $D_n$ has length at least $e^{-2AK_0}\min(|I_n^{(2)}|,|I'|)\ge e^{-2AK_0}\rho_1^{\,L+2}\rho_1|D_n|\ge\rho_1^{\,L+4}|D_n|$, since $\rho_1\le e^{-2AK_0}$.

For the tiles, let $w(I_n^{(1)})$ be a tile of type $1$. By Theorem \ref{thm:tiling} its stage $n+1$ pieces are the images under $w$ of the pieces of $D_n$, and $\Dist(w|_{D_n})\le K_0$ by Proposition \ref{prop:denjoy}. Therefore each piece occupies a fraction at least $e^{-K_0}\rho_1^{\,L+4}=\rho_3$ of $w(I_n^{(1)})$, and since there are at least two pieces and they partition the tile, each occupies a fraction at most $1-\rho_3$. Tiles of types $2$ and $3$ are unchanged at the step. By Lemma \ref{lem:participation} every stage $n$ tile is of type $1$ at some stage $n'\le n+L+2$ and is split at the step from $n'$ to $n'+1$, so $T_{\max}(n+L+3)\le(1-\rho_3)T_{\max}(n)$ for every $n$, and iterating from $T_{\max}(0)\le1$ gives the stated bound.
\end{proof}

\begin{corollary}\label{cor:mesh}
Let $\theta(n)$ be the largest length of an interval $p_s(D_n)$ and $\widetilde\theta(n)$ the largest length of an interval $p_s(\widetilde I_n)$, where $p_s$ runs over the prefixes of the fixed factorizations of $u_1^{(n)}$, $u_2^{(n)}$ and $u_3^{(n)}$ into letters. Then $\theta(n)\le\widetilde\theta(n)\le(4A+2)\,T_{\max}(n)\le C\sigma^{\,n}$. More generally, without any hypothesis on the vector, $\theta(n)\le\big(t_1(m)/t_3(m)+2\big)\,T_{\max}(m)$ and $\widetilde\theta(n)\le\big(4t_1(m)/t_3(m)+2\big)\,T_{\max}(m)$ for all $m\le n$.
\end{corollary}

\begin{proof}
The interval $p_s(D_n)$ has $\mu$ measure $t_1(n)$ and the interval $p_s(\widetilde I_n)$ has $\mu$ measure $\tau(\Omega_n)\le4t_1(n)$, because $\mu$ is $\Gamma$ invariant, and $t_1(n)\le t_1(m)$ for $m\le n$. Every stage $m$ tile has $\mu$ measure $t_i(m)\ge t_3(m)$, so an interval of $\mu$ measure $\kappa t_1(n)$ contains at most $\kappa t_1(m)/t_3(m)$ tiles of stage $m$, meets at most $\kappa t_1(m)/t_3(m)+2$ of them, and is covered by the tiles it meets. Under simultaneously bounded type $t_1(n)/t_3(n)\le A$, and Proposition \ref{prop:ratios} gives the decay.
\end{proof}

Propositions \ref{prop:denjoy}, \ref{prop:derivbound} and \ref{prop:ratios} are the real bounds. They say that at every stage the normalized pair $\xi_n$ has cut points at a definite distance from the endpoints of the chart, branch derivatives in $[e^{-4K_0},e^{4K_0}]$ and branch distortion at most $2K_0$, and that the mesh of the tiling decays exponentially. This is the rank two analogue of the statement that the renormalizations of a rank one circle diffeomorphism of bounded type lie in a compact set of pairs.

\section{M\"obius approximation and exponential convergence}\label{sec:moebius}

\subsection{The M\"obius family}
For $w\in\R$ let
\[
m_w(x)\ =\ \frac{x}{(1-e^{w/2})x+e^{w/2}} ,
\]
an increasing diffeomorphism of $[0,1]$ onto itself with $m_w(0)=0$ and $m_w(1)=1$. A direct computation gives
\[
Dm_w(x)=\frac{e^{w/2}}{\big(e^{w/2}(1-x)+x\big)^2},\qquad
\eta_{m_w}(x)=\frac{2(e^{w/2}-1)}{e^{w/2}(1-x)+x},\qquad
\int_0^1\eta_{m_w}=w .
\]
The maps $m_w$ form a one parameter group, $m_w\circ m_{w'}=m_{w+w'}$, so the parameter is additive under composition. It is convenient to use the coordinate
\[
z\ =\ e^{-w/2}-1 ,
\]
which is a real analytic reparametrization with $z=0$ at the identity and with $dz/dw=-1/2$ at $w=0$. For two $C^2$ diffeomorphisms of an interval $J$ we write
\[
\operatorname{dist}_\eta(\varphi,\psi)\ =\ \int_J\big|\eta_\varphi-\eta_\psi\big| .
\]
If $\varphi$ maps $J$ onto $\varphi(J)$ we write $\widehat\varphi$ for its normalization, the conjugate of $\varphi$ by the affine maps taking $J$ and $\varphi(J)$ onto $[0,1]$. Then $\eta_{\widehat\varphi}(y)=|J|\,\eta_\varphi(x)$ with $y$ the normalized coordinate of $x$, so that $\int_0^1\eta_{\widehat\varphi}=\int_J\eta_\varphi$, and $\operatorname{dist}_\eta$ is unchanged by normalization. We record two facts.

\begin{lemma}\label{lem:normclose}
If $\operatorname{dist}_\eta(\widehat\varphi,\mathrm{id})=\int_0^1|\eta_{\widehat\varphi}|=\epsilon\le1$ then $\Dist(\widehat\varphi|_{[0,1]})\le\epsilon$ and $\|\widehat\varphi-\mathrm{id}\|_{C^1}\le3\epsilon$. For every real $w$ one has
\[
\sup_{[0,1]}|\eta_{m_w}|\ \le\ |w|\,e^{|w|/2},\qquad \sup_{[0,1]}|D\eta_{m_w}|\ \le\ \tfrac12\,w^2e^{2|w|},
\]
and if $|w|\le1$ then $\operatorname{dist}_\eta(m_w,\mathrm{id})\le3|w|$ and $|z|\le|w|$.
\end{lemma}

\begin{proof}
The first inequality is $\Dist\le\int|\eta|$. Since $\int_0^1D\widehat\varphi=1$ and $\log D\widehat\varphi$ has oscillation at most $\epsilon$, there is a point where $D\widehat\varphi=1$, hence $|\log D\widehat\varphi|\le\epsilon$ everywhere and $|D\widehat\varphi-1|\le2\epsilon$ for $\epsilon\le1$. Integrating from $0$ and from $1$ gives $|\widehat\varphi(y)-y|\le2\epsilon\min(y,1-y)\le\epsilon$. For the bounds on the M\"obius maps write $s=e^{w/2}$. The denominator $e^{w/2}(1-x)+x$ lies between $\min(1,s)\ge e^{-|w|/2}$ and $\max(1,s)$, and $|s-1|\le\tfrac{|w|}{2}e^{|w|/2}$, so $|\eta_{m_w}|=2|s-1|/(e^{w/2}(1-x)+x)\le|w|e^{|w|/2}$ and $|D\eta_{m_w}|=2(s-1)^2/(e^{w/2}(1-x)+x)^2\le\tfrac12w^2e^{2|w|}$. For $|w|\le1$ this gives $|\eta_{m_w}|\le3|w|$, and $|z|=|e^{-w/2}-1|\le|w|$.
\end{proof}

\begin{lemma}\label{lem:moebius}
Let $\varphi$ be a $C^3$ diffeomorphism of an interval $J$ with $|D\eta_\varphi|\le E''$ and $|\eta_\varphi|\le E'$ on $J$, and let $w=\int_J\eta_\varphi$. Then
\[
\sup_{[0,1]}\big|\eta_{\widehat\varphi}-\eta_{m_w}\big|\ \le\ C_2\,|J|^2 ,\qquad C_2=C_2(E',E'') ,
\]
and in particular $\operatorname{dist}_\eta(\widehat\varphi,m_w)\le C_2|J|^2$.
\end{lemma}

\begin{proof}
In the normalized coordinate $\eta_{\widehat\varphi}(y)=|J|\eta_\varphi(x_J+|J|y)$, so $\eta_{\widehat\varphi}$ has total variation at most $|J|^2E''$ on $[0,1]$ and mean $w$. Hence $|\eta_{\widehat\varphi}-w|\le|J|^2E''$ pointwise. Also $|w|\le|J|E'$, and from the formula above $|\eta_{m_w}-w|\le C(E')w^2\le C(E')E'^2|J|^2$ for $|w|\le|J|E'$. Adding the two estimates gives the pointwise bound, and integrating over $[0,1]$ gives the second claim.
\end{proof}

\begin{lemma}\label{lem:composition}
Let $\varphi=\varphi_N\circ\cdots\circ\varphi_1$ on an interval $J$, let $J_s=\varphi_s\circ\cdots\circ\varphi_1(J)$ with $J_0=J$, let $\widehat\varphi_s$ be the normalization of $\varphi_s|_{J_{s-1}}$, and suppose $\operatorname{dist}_\eta(\widehat\varphi_s,m_{w_s})\le\epsilon_s$ for each $s$, where $w_s=\int_{J_{s-1}}\eta_{\varphi_s}$. Put $\Theta=\sum_{s=1}^{N}\int_0^1|\eta_{\widehat\varphi_s}|=\sum_{s=1}^{N}\int_{J_{s-1}}|\eta_{\varphi_s}|$. Then
\[
\operatorname{dist}_\eta\big(\widehat\varphi,\,m_{w}\big)\ \le\ C(\Theta)\sum_{s=1}^{N}\epsilon_s,\qquad w=\sum_{s=1}^{N}w_s,\qquad C(\Theta)=\exp\Big(\Theta e^{3\Theta}\exp\big(2\Theta e^{\Theta/2}\big)\Big),
\]
where the constant $C(\Theta)$ depends only on $\Theta$ and is nondecreasing in $\Theta$. If moreover $\sup_{[0,1]}|\eta_{\widehat\varphi_s}-\eta_{m_{w_s}}|\le\epsilon_s^{\infty}$ for each $s$, then
\[
\sup_{[0,1]}\big|\eta_{\widehat\varphi}-\eta_{m_w}\big|\ \le\ C'(\Theta)\sum_{s=1}^{N}\epsilon_s^{\infty},\qquad C'(\Theta)=\exp\Big(\Theta+\Theta e^{3\Theta}\exp\big(\Theta+2\Theta e^{\Theta/2}\big)\Big),
\]
and $C'(\Theta)\ge C(\Theta)$ is again nondecreasing in $\Theta$.
\end{lemma}

\begin{proof}
The normalization of a composition is the composition of the normalizations, since the intermediate affine maps cancel, so $\widehat\varphi=\Phi_N$ with $\Phi_k=\widehat\varphi_k\circ\cdots\circ\widehat\varphi_1$. Put $M_k=m_{w_k}\circ\cdots\circ m_{w_1}=m_{W_k}$ with $W_k=w_1+\cdots+w_k$, and $d_k=\operatorname{dist}_\eta(\Phi_k,M_k)$, so that $d_0=0$. We first record an a priori bound. For diffeomorphisms $\psi_1,\psi_2$ of $[0,1]$ onto itself the chain rule $\eta_{\psi_2\circ\psi_1}=(\eta_{\psi_2}\circ\psi_1)D\psi_1+\eta_{\psi_1}$ and the change of variables $y=\psi_1(x)$ give $\int_0^1|\eta_{\psi_2\circ\psi_1}|\le\int_0^1|\eta_{\psi_2}|+\int_0^1|\eta_{\psi_1}|$, hence $\int_0^1|\eta_{\Phi_k}|\le\Theta$. Also $|W_k|\le\Theta$ and $\int_0^1|\eta_{M_k}|\le|W_k|e^{|W_k|/2}$ by Lemma \ref{lem:normclose}. Therefore $d_k\le\Theta(1+e^{\Theta/2})=:\Theta_1$ for all $k$.

Now fix $k\ge1$ and write $m=m_{w_k}$, $\Phi=\Phi_{k-1}$, $M=M_{k-1}$ and $d=d_{k-1}$. By the triangle inequality
\[
d_k\ \le\ \int_0^1\big|\eta_{\widehat\varphi_k\circ\Phi}-\eta_{m\circ\Phi}\big|+\int_0^1\big|\eta_{m\circ\Phi}-\eta_{m\circ M}\big| .
\]
The first integral equals $\int_0^1|(\eta_{\widehat\varphi_k}-\eta_m)\circ\Phi|\,D\Phi=\int_0^1|\eta_{\widehat\varphi_k}-\eta_m|\le\epsilon_k$. For the second, the chain rule gives
\[
\eta_{m\circ\Phi}-\eta_{m\circ M}\ =\ \big[(\eta_m\circ\Phi)D\Phi-(\eta_m\circ M)DM\big]+\big[\eta_\Phi-\eta_M\big],
\]
and the second bracket contributes at most $d$. For the first bracket,
\[
\big|(\eta_m\circ\Phi)D\Phi-(\eta_m\circ M)DM\big|\ \le\ \sup|D\eta_m|\,|\Phi-M|\,D\Phi+\sup|\eta_m|\,|D\Phi-DM| .
\]
The function $\log D\Phi-\log DM$ has derivative $\eta_\Phi-\eta_M$, hence oscillation at most $d$, and since $\int_0^1D\Phi=\int_0^1DM=1$ there is a point where $D\Phi=DM$. So $|\log D\Phi-\log DM|\le d$ everywhere, $|D\Phi-DM|\le(e^{d}-1)DM$, $\int_0^1|D\Phi-DM|\le e^{d}-1$ and $|\Phi-M|\le e^{d}-1$. Integrating the last display over $[0,1]$ and using Lemma \ref{lem:normclose} we get
\[
\int_0^1\big|(\eta_m\circ\Phi)D\Phi-(\eta_m\circ M)DM\big|\ \le\ c(w_k)\,(e^{d}-1)\ \le\ c(w_k)\,e^{\Theta_1}\,d ,
\]
where $c(w):=|w|e^{|w|/2}+\tfrac12w^2e^{2|w|}\le|w|e^{3|w|}$, and where we used $e^{d}-1\le de^{d}$ and $d\le\Theta_1$. Altogether $d_k\le\epsilon_k+\big(1+c(w_k)e^{\Theta_1}\big)d_{k-1}$, and by induction
\[
d_N\ \le\ \sum_{k=1}^{N}\epsilon_k\prod_{j>k}\big(1+c(w_j)e^{\Theta_1}\big)\ \le\ \exp\Big(e^{\Theta_1}\sum_{j}c(w_j)\Big)\sum_{k=1}^{N}\epsilon_k .
\]
Finally $|w_j|\le\int_0^1|\eta_{\widehat\varphi_j}|$, so $\sum_jc(w_j)\le e^{3\Theta}\sum_j|w_j|\le\Theta e^{3\Theta}$, and $e^{\Theta_1}\le\exp(2\Theta e^{\Theta/2})$, which gives the constant $C(\Theta)$.

For the supremum norm put $e_k=\sup_{[0,1]}|\eta_{\Phi_k}-\eta_{M_k}|$, so that $d_k\le e_k$. The three terms of the decomposition above are now estimated pointwise. Since $\Phi$ is a diffeomorphism of $[0,1]$ onto itself with $\int_0^1D\Phi=1$, we have $|\log D\Phi|\le\int_0^1|\eta_\Phi|\le\Theta$, and therefore $|\eta_{\widehat\varphi_k\circ\Phi}-\eta_{m\circ\Phi}|=|(\eta_{\widehat\varphi_k}-\eta_m)\circ\Phi|\,D\Phi\le e^{\Theta}\epsilon_k^{\infty}$. The second bracket is bounded by $e_{k-1}$. The first bracket is bounded pointwise by $c(w_k)\,(e^{d}-1)\max(\sup D\Phi,\sup DM)$, and $|\log DM|\le\int_0^1|\eta_M|\le\Theta e^{\Theta/2}$ by Lemma \ref{lem:normclose}, so it is at most $c(w_k)e^{\Theta_1}e^{\Theta e^{\Theta/2}}d\le c(w_k)e^{\Theta_1+\Theta e^{\Theta/2}}e_{k-1}$. Hence
\[
e_k\ \le\ e^{\Theta}\epsilon_k^{\infty}+\big(1+c(w_k)e^{\Theta_1+\Theta e^{\Theta/2}}\big)\,e_{k-1},
\]
and the same induction gives $e_N\le e^{\Theta}\exp\big(e^{\Theta_1+\Theta e^{\Theta/2}}\sum_jc(w_j)\big)\sum_k\epsilon_k^{\infty}$. Since $\sum_jc(w_j)\le\Theta e^{3\Theta}$ and $\Theta_1+\Theta e^{\Theta/2}=\Theta+2\Theta e^{\Theta/2}$, this is the stated bound.
\end{proof}

\begin{proposition}\label{prop:basemoebius}
Suppose that $\Sigma^{\ast}<\infty$ and that the base point is admissible. Put
\[
W_i(n)\ =\ \int_{D_n}\eta_{u_i^{(n)}} ,\qquad \widetilde W_i(n)\ =\ \int_{\widetilde I_n}\eta_{u_i^{(n)}} ,\qquad i=1,2,3 ,
\]
and let $\widehat{u_i^{(n)}}$ be the normalization of $u_i^{(n)}$ on $D_n$ and $\widetilde{u_i^{(n)}}$ its normalization on $\widetilde I_n$. Then for every $n$
\[
\sup_{[0,1]}\big|\eta_{\widehat{u_i^{(n)}}}-\eta_{m_{W_i(n)}}\big|\ \le\ C_3\,\theta(n),\qquad
\sup_{[0,1]}\big|\eta_{\widetilde{u_i^{(n)}}}-\eta_{m_{\widetilde W_i(n)}}\big|\ \le\ C_3\,\widetilde\theta(n),
\]
with $C_3=C'(E'K_1)\,C_2K_1$, where $\theta(n)$ and $\widetilde\theta(n)$ are as in Corollary \ref{cor:mesh}, $C_2=C_2(E',E'')$ is the constant of Lemma \ref{lem:moebius} and $C'(\cdot)$ is the function of Lemma \ref{lem:composition}. In particular $\operatorname{dist}_\eta(\widehat{u_i^{(n)}},m_{W_i(n)})\le C_3\theta(n)$.
\end{proposition}

\begin{proof}
Let $J$ be $D_n$ or $\widetilde I_n$, let $\theta_J$ be $\theta(n)$ or $\widetilde\theta(n)$ accordingly, and let $\widehat\varphi$ be the normalization of $u_i^{(n)}$ on $J$. Factor $u_i^{(n)}$ into letters as in Lemma \ref{lem:totalnl}. The intermediate intervals are the $p_s(J)$, of length at most $\theta_J$, and by Lemma \ref{lem:moebius}, applied to the letter $\ell_s$ on $p_{s-1}(J)$, the normalized factor at step $s$ is within $C_2|p_{s-1}(J)|^2$ of a M\"obius map in the supremum norm. The quantity $\Theta$ of Lemma \ref{lem:composition} is at most $\sum_s\int_{p_{s-1}(J)}|\eta_{\ell_s}|\le E'\sum_s|p_{s-1}(J)|\le E'K_1$ by Convention \ref{con:basepoint}, and $C'(\cdot)$ is nondecreasing. By Lemma \ref{lem:composition},
\[
\sup_{[0,1]}\big|\eta_{\widehat\varphi}-\eta_{m_{w}}\big|\ \le\ C'(E'K_1)\,C_2\sum_s\big|p_{s}(J)\big|^2\ \le\ C'(E'K_1)\,C_2\,\theta_J\sum_s\big|p_{s}(J)\big|\ \le\ C_3\,\theta_J,
\]
the last step by the definition of $K_1$, and $w=\sum_sw_s=\int_{J}\eta_{u_i^{(n)}}$ because the parameters of the factors add and each equals the integral of the nonlinearity of that factor over its own domain. The last claim follows by integration over $[0,1]$.
\end{proof}

\subsection{The weight recursion}
Let $z_i(n)=e^{-W_i(n)/2}-1$ be the M\"obius parameters of the three basis maps at stage $n$. The next lemma is the exact restriction property of the M\"obius family, and it is the reason the parameter of the deck element controls everything.

\begin{lemma}\label{lem:restriction}
Let $0<\rho\le1$. The normalization of $m_w|_{[0,\rho]}$ is $m_{w'}$ with $z'=\rho\,z$.
\end{lemma}

\begin{proof}
Write $s=e^{w/2}$, so $m_w(x)=x/((1-s)x+s)$ and $z=(1-s)/s$. The normalization of the restriction is $y\mapsto m_w(\rho y)/m_w(\rho)$, and a computation gives
\[
\frac{m_w(\rho y)}{m_w(\rho)}\ =\ \frac{y}{\frac{(1-s)\rho}{(1-s)\rho+s}\,y+\frac{s}{(1-s)\rho+s}} ,
\]
which is $m_{w'}$ with $s'=s/((1-s)\rho+s)$. Then $z'=(1-s')/s'=(1-s)\rho/s=\rho z$.
\end{proof}

\begin{proposition}\label{prop:zrecursion}
Suppose that $\Sigma^{\ast}<\infty$ and that the base point is admissible. Then for every $n$
\[
\big|z_i(n)-r_i(n)\,z_1(n)\big|\ \le\ C_4\,\theta(n)\quad(i=2,3),
\qquad
\big|z_1(n+1)-r_2(n)^2\,z_1(n)\big|\ \le\ C_4\,\theta(n) ,
\]
with $C_4$ depending only on $K_0$ and $C_3$. Moreover, for every $n$ and $i$,
\[
\sup_{\widetilde I_n}\big|\eta_{u_i^{(n)}}\big|\ \le\ C_4'\ \frac{|z_1(n)|+\widetilde\theta(n)}{|D_n|},
\]
with $C_4'$ depending only on $K_0$ and $C_3$. If $(\alpha,\beta)$ is of simultaneously bounded type with constant $A$, then there are $\xi_1\in(0,1)$ and $C$ such that $|z_1(n)|\le C\xi_1^{\,n}$, $|W_i(n)|\le C\xi_1^{\,n}$ and $\sup_{\widetilde I_n}|\eta_{u_i^{(n)}}|\le C\xi_1^{\,n}/|D_n|$ for all $n$ and $i$.
\end{proposition}

\begin{proof}
Identity \eqref{eq:symmetry} applied with $U=u_i^{(n)}$, $W=u_1^{(n)}$ and $x=x_0$ gives
\[
W_i(n)\ =\ \int_{[x_0,u_1^{(n)}x_0)}\eta_{u_i^{(n)}}\ =\ \int_{[x_0,u_i^{(n)}x_0)}\eta_{u_1^{(n)}}\ =\ \int_{I_n^{(i)}}\eta_{u_1^{(n)}} .
\]
So the weight of the $i$th basis element over the whole fundamental domain equals the weight of the deck element over the initial subinterval $I_n^{(i)}$, which in the normalized coordinate of $D_n$ is $[0,r_i(n)]$. By Proposition \ref{prop:basemoebius},
\[
\Big|\int_0^{r_i(n)}\eta_{\widehat{u_1^{(n)}}}-\int_0^{r_i(n)}\eta_{m_{W_1(n)}}\Big|\ \le\ C_3\theta(n),
\]
and the second integral is the parameter of the normalized restriction of $m_{W_1(n)}$ to $[0,r_i(n)]$, whose $z$ coordinate is $r_i(n)z_1(n)$ by Lemma \ref{lem:restriction}. The parameters $W_i(n)$ are bounded by $K_0$, since $|W_i(n)|\le\int_{D_n}|\eta_{u_i^{(n)}}|\le K_0$ by \eqref{eq:K0basic}, so the change of coordinate between $w$ and $z$ is bi-Lipschitz on the relevant range and the first estimate follows.

For the second, $u_1^{(n+1)}=u_2^{(n)}$ and $D_{n+1}=I_n^{(2)}$, so $W_1(n+1)$ is the parameter of the normalized restriction of $\widehat{u_2^{(n)}}$ to $[0,r_2(n)]$. Applying Proposition \ref{prop:basemoebius} to $u_2^{(n)}$ and Lemma \ref{lem:restriction} gives $|z_1(n+1)-r_2(n)z_2(n)|\le C\theta(n)$, and combining with the first estimate for $i=2$, together with $r_2(n)\le1$, gives the second display.

For the bound on $\widetilde I_n$ let $\widetilde z_i(n)=e^{-\widetilde W_i(n)/2}-1$ be the parameter of the integral of Proposition \ref{prop:basemoebius}, and put $\rho=|D_n|/|\widetilde I_n|$. The $\widetilde I_n$ is the disjoint union of $D_n$, $u_1^{(n)}(D_n)$, $(u_1^{(n)})^2(I_n^{(2)})$ and $(u_1^{(n)})^2u_2^{(n)}(I_n^{(3)})$, and the derivatives of $u_1^{(n)}$ and $u_2^{(n)}$ needed to compute the lengths of these images are all evaluated at points of $\widetilde I_n$, so Proposition \ref{prop:derivbound}(i) gives $|\widetilde I_n|\le(1+e^{2K_0}+e^{4K_0}+e^{6K_0})|D_n|$ and $\rho\ge e^{-6K_0}/4$. In the normalized coordinate of $\widetilde I_n$ the interval $D_n$ is $[0,\rho]$ and $\int_0^\rho\eta_{\widetilde{u_i^{(n)}}}=W_i(n)$. By Proposition \ref{prop:basemoebius} this integral differs from $\int_0^\rho\eta_{m_{\widetilde W_i(n)}}$ by at most $C_3\widetilde\theta(n)$, and the latter is the parameter of the normalized restriction of $m_{\widetilde W_i(n)}$ to $[0,\rho]$, whose $z$ coordinate is $\rho\,\widetilde z_i(n)$ by Lemma \ref{lem:restriction}. Since $|\widetilde W_i(n)|\le\mu_i(n)\le K_0$, the change of coordinate between $w$ and $z$ is bi-Lipschitz on the relevant range, so $|z_i(n)-\rho\,\widetilde z_i(n)|\le C\widetilde\theta(n)$ and $|\widetilde z_i(n)|\le\rho^{-1}\big(|z_i(n)|+C\widetilde\theta(n)\big)$. By Lemma \ref{lem:normclose} and Proposition \ref{prop:basemoebius}, $\sup_{[0,1]}|\eta_{\widetilde{u_i^{(n)}}}|\le|\widetilde W_i(n)|e^{|\widetilde W_i(n)|/2}+C_3\widetilde\theta(n)\le C\big(|\widetilde z_i(n)|+\widetilde\theta(n)\big)$, and undoing the normalization divides the nonlinearity by $|\widetilde I_n|\ge|D_n|$. Together with the first display, which gives $|z_i(n)|\le|z_1(n)|+C_4\theta(n)$, and with $\theta(n)\le\widetilde\theta(n)$, this proves the bound on $\widetilde I_n$.

Finally, under simultaneously bounded type, $r_2(n)\le\rho_2<1$ by Proposition \ref{prop:ratios} and $\widetilde\theta(n)\le C\sigma^n$ by Corollary \ref{cor:mesh}, so $|z_1(n+1)|\le\rho_2^2|z_1(n)|+C\sigma^n$. Iterating gives $|z_1(n)|\le C\xi_1^n$ for any $\xi_1$ with $\max(\rho_2^2,\sigma)<\xi_1<1$. The bounds on $W_i(n)$ and on the nonlinearity over the $\widetilde I_n$ follow from the first display, Lemma \ref{lem:normclose} and the bound just proved.
\end{proof}

\subsection{The slope recursion}
Put $\Sigma(n)=\log\big(|u_1^{(n)}(D_n)|/|D_n|\big)$, the logarithmic slope of the deck element over the fundamental domain. By Proposition \ref{prop:derivbound}(i) we have $|\Sigma(n)|\le2K_0$ for all $n$.

\begin{proposition}\label{prop:sloperecursion}
Suppose that $\Sigma^{\ast}<\infty$ and that the base point is admissible. There is $C_5$, depending only on $K_0$ and $C_3$, such that for every $n$
\[
\big|\Sigma(n+1)\big|\ \le\ \kappa\big(r_2(n)\big)\,\big|\Sigma(n)\big|+C_5\,\epsilon_n,\qquad\text{where}\quad
\kappa(r):=1-\frac{1-r}{1+e^{2K_0}}\in(0,1]
\]
and $\epsilon_n:=|W_1(n)|+|W_2(n)|+\theta(n)$.
If moreover $(\alpha,\beta)$ is of simultaneously bounded type, then $\kappa(r_2(n))\le\kappa:=\kappa(\rho_2)<1$ for all $n$, and consequently $|\Sigma(n)|\le C\xi_2^{\,n}$ for some $\xi_2\in(0,1)$ and all $n$.
\end{proposition}

\begin{proof}
Since $u_1^{(n+1)}=u_2^{(n)}$ and $D_{n+1}=I_n^{(2)}$, and since $u_2^{(n)}(I_n^{(2)})=[u_2^{(n)}x_0,(u_2^{(n)})^2x_0)$,
\[
e^{\Sigma(n+1)}\ =\ \frac{|u_2^{(n)}(I_n^{(2)})|}{|I_n^{(2)}|}\ =\ \frac{1}{|I_n^{(2)}|}\int_{I_n^{(2)}}Du_2^{(n)} .
\]
By Proposition \ref{prop:basemoebius} and Lemma \ref{lem:normclose} the oscillation of $\log Du_2^{(n)}$ on $D_n$ is at most $\int_0^1|\eta_{\widehat{u_2^{(n)}}}|\le e^{K_0/2}|W_2(n)|+C_3\theta(n)\le C\epsilon_n$. Hence the average of $Du_2^{(n)}$ over $I_n^{(2)}$ and over $D_n$ differ by a factor $1+O(\epsilon_n)$, and
\[
e^{\Sigma(n+1)}\ =\ \frac{|u_2^{(n)}(D_n)|}{|D_n|}\,\big(1+O(\epsilon_n)\big).
\]
Next, $u_2^{(n)}(D_n)=[u_2^{(n)}x_0,u_1^{(n)}u_2^{(n)}x_0)$ by commutation, so with $y=u_2^{(n)}x_0$
\[
\big|u_2^{(n)}(D_n)\big|-|D_n|\ =\ \big(u_1^{(n)}y-y\big)-\big(u_1^{(n)}x_0-x_0\big)\ =\ \int_{x_0}^{y}\big(Du_1^{(n)}-1\big)\ =\ \int_{I_n^{(2)}}\big(Du_1^{(n)}-1\big).
\]
The oscillation of $\log Du_1^{(n)}$ on $D_n$ is at most $C\epsilon_n$ by the same argument, and the average of $Du_1^{(n)}$ over $D_n$ is $e^{\Sigma(n)}$, so $Du_1^{(n)}=e^{\Sigma(n)}(1+O(\epsilon_n))$ pointwise on $D_n$. Therefore, using $|\Sigma(n)|\le2K_0$ and $r_2(n)\le1$,
\[
e^{\Sigma(n+1)}\ =\ 1+r_2(n)\big(e^{\Sigma(n)}-1\big)+O(\epsilon_n),
\]
and since $1+r(e^{s}-1)\ge\min(1,e^{s})\ge e^{-2K_0}$ for $0<r\le1$ and $|s|\le2K_0$, this gives $\Sigma(n+1)=\Psi_{r_2(n)}(\Sigma(n))+O(\epsilon_n)$ with $\Psi_r(s)=\log\big(1+r(e^s-1)\big)$, all constants depending only on $K_0$ and $C_3$. The map $\Psi_r$ fixes $0$ and
\[
\Psi_r'(s)\ =\ \frac{re^s}{1-r+re^s}\ =\ 1-\frac{1-r}{1-r+re^s}\ \le\ 1-\frac{1-r}{1+e^{2K_0}}\ =\ \kappa(r)
\]
for $0<r\le1$ and $|s|\le2K_0$, and $\Psi_r'>0$. Hence $|\Psi_{r_2(n)}(\Sigma(n))|\le\kappa(r_2(n))|\Sigma(n)|$, which is the stated inequality. Under simultaneously bounded type $r_2(n)\le\rho_2$ by Proposition \ref{prop:ratios}, so $\kappa(r_2(n))\le\kappa(\rho_2)<1$. Moreover $\epsilon_n\le C\xi_1^n$ by Proposition \ref{prop:zrecursion} and Corollary \ref{cor:mesh}, and iteration gives the exponential decay of $\Sigma(n)$ with any $\xi_2$ strictly between $\max(\kappa,\xi_1)$ and $1$.
\end{proof}

\subsection{Proof of Theorem A}
For a two branch map $P$ of $[0,1)$ with cut point $c$ let $\widetilde P\colon[0,1)\to[0,1)$ be the map equal to the first branch on $[0,c)$ and to the second branch on $[c,1)$, regarded as a homeomorphism of the circle $\R/\Z$, and let $D\widetilde P$ be its piecewise derivative. We use on pairs $\xi=(P,Q)$ of two branch maps the distance
\[
d_1(\xi,\xi')\ =\ \max\big(\|\widetilde P-\widetilde P'\|_{C^0(\R/\Z)},\ \|D\widetilde P-D\widetilde P'\|_{\infty},\ \|\widetilde Q-\widetilde Q'\|_{C^0(\R/\Z)},\ \|D\widetilde Q-D\widetilde Q'\|_{\infty}\big),
\]
where $\|\cdot\|_{C^0(\R/\Z)}$ is the supremum of the distance on the circle. This is a metric on pairs of two branch maps: it is built from supremum norms, and it vanishes only when the two pairs coincide, since the cut point of $P$ is the preimage of the seam under $\widetilde P$. For pairs whose branch derivatives are close to $1$ it also controls the distance of the cut points.

\begin{theorem}\label{thm:main}
Let $f$ and $g$ be commuting orientation preserving $C^3$ circle diffeomorphisms whose rotation vector $(\alpha,\beta)$ is rationally independent together with $1$, of simultaneously bounded type with constant $A$, and summably convergent, and let the base point be admissible. Then there are $\xi\in(0,1)$ and $C$ such that
\[
d_1\big(\xi_n,\mathcal{O}_n\big)\ \le\ C\,\xi^{\,n}\qquad\text{for all } n\ge0 .
\]
\end{theorem}

\begin{proof}
Put $\delta_n=|W_1(n)|+|W_2(n)|+|W_3(n)|+|\Sigma(n)|+\theta(n)$. By Corollary \ref{cor:mesh} and Propositions \ref{prop:zrecursion} and \ref{prop:sloperecursion} we have $\delta_n\le C\xi^n$ with $\xi=\max(\xi_1,\xi_2,\sigma)$. We work in the normalized coordinate of $D_n$, in which $x_0$ becomes $0$, the point $u_1^{(n)}x_0$ becomes $1$, and $u_i^{(n)}x_0$ becomes $r_i(n)$. In the rest of the proof the constants $C$ depend only on the constants of Convention \ref{con:basepoint}, and the exponential decay of $\delta_n$ is not used.

By Proposition \ref{prop:basemoebius} and Lemma \ref{lem:normclose} the normalized maps $\widehat{u_i^{(n)}}$ satisfy $\|\widehat{u_i^{(n)}}-\mathrm{id}\|_{C^1}\le C\delta_n$. Undoing the normalization of the range multiplies the derivative by the slope $e^{s_i(n)}$ with $s_i(n)=\log(|u_i^{(n)}(D_n)|/|D_n|)$. The computation in the proof of Proposition \ref{prop:sloperecursion} gives $e^{s_i(n)}=1+r_i(n)(e^{\Sigma(n)}-1)+O(\delta_n)$, so $|s_i(n)|\le C\delta_n$. Hence in the normalized coordinate of $D_n$
\begin{equation}\label{eq:mainestimates}
\big|Du_i^{(n)}-1\big|\ \le\ C\delta_n\quad\text{on } [0,1),\qquad u_i^{(n)}(0)=r_i(n),
\end{equation}
so $|u_i^{(n)}(x)-x-r_i(n)|\le C\delta_n$ for all $x\in[0,1)$.

The first branch of $P_n$ is $u_2^{(n)}$ on $[0,c_2(n))$ and the second is $(u_1^{(n)})^{-1}u_2^{(n)}$ on $[c_2(n),1)$. The second branch maps its domain into $[0,r_2(n))$, and $u_2^{(n)}$ maps its domain into $u_1^{(n)}([0,1))$, on which $(u_1^{(n)})^{-1}$ is within $C\delta_n$ in $C^1$ of the translation by $-1$, by \eqref{eq:mainestimates} applied to $u_1^{(n)}$. So $\widetilde P_n$ is within $C\delta_n$ of the rotation by $r_2(n)$ in the $C^0$ distance of the circle, and its piecewise derivative is within $C\delta_n$ of $1$. The same holds for $\widetilde Q_n$ and $r_3(n)$. For the cut point, $u_2^{(n)}(c_2(n))=1$ and $|u_2^{(n)}(x)-x-r_2(n)|\le C\delta_n$ give $|c_2(n)-(1-r_2(n))|\le C\delta_n$, and likewise for $c_3(n)$.

Finally $P_n$ is a circle homeomorphism of the circle of length one whose lift satisfies $|\widetilde P_n(x)-x-r_2(n)|\le C\delta_n$, so its rotation number differs from $r_2(n)$ by at most $C\delta_n$. That rotation number is $t_2(n)/t_1(n)$, so
\begin{equation}\label{eq:ratioestimates}
\big|r_i(n)-t_i(n)/t_1(n)\big|\ \le\ C\delta_n\qquad(i=1,2,3),
\end{equation}
the case $i=3$ being obtained from $Q_n$ and the case $i=1$ being trivial. Combining the three steps gives $d_1(\xi_n,\mathcal{O}_n)\le C\delta_n$, and the exponential decay of $\delta_n$ finishes the proof.
\end{proof}

Theorem A is Theorem \ref{thm:main}. The rate $\xi$ is determined by the three contraction factors $\rho_2^2$, $\kappa$ and $\sigma$, and none of them has been optimized. We record that the estimates \eqref{eq:mainestimates} and \eqref{eq:ratioestimates}, as well as the inequality $d_1(\xi_n,\mathcal{O}_n)\le C\delta_n$, hold at every stage $n$ under the sole hypotheses that $\Sigma^{\ast}<\infty$ and that the base point is admissible, with $C$ depending only on the constants of Convention \ref{con:basepoint}. Bounded type was used only to obtain the exponential decay of $\delta_n$.

\begin{remark}\label{rem:uniformrate}
The rate $\xi$ of Theorem \ref{thm:main} depends on the pair and on the base point only through $A$, $L(A)$ and $K_0$. Indeed the factors $\rho_1$, $\rho_2$, $\rho_3$ and $\sigma$ of Proposition \ref{prop:ratios} and the factor $\kappa$ of Proposition \ref{prop:sloperecursion} are functions of these three quantities alone, and the remaining constants of Convention \ref{con:basepoint} enter only the constant $C$. Moreover $K_0$ is used throughout as an upper bound, so all estimates hold with any larger number in its place. This makes the rate uniform over the class after finitely many steps. Since admissibility is a statement about tails, there is a stage $n_0=n_0(f,g,x_0)$ with $\sum_{m\ge n_0}\sum_{i}a_{m+1}\mu_i(m)\le1$. By Remark \ref{rem:welldefined} the normalized pair $\xi_{n_0}$ is a chart of the pointed pair $(\mathcal{C}_{n_0};P_{n_0},Q_{n_0};x_0)$, whose stages are the stages $n_0+k$ of the original pair and whose tile words involve only factors of stages at least $n_0$. Its rotation vector $T^{n_0}(\alpha,\beta)$ is of simultaneously bounded type with constant $A$ and summably convergent, and $x_0$ is admissible for it, since its two series are the tails of the series of Convention \ref{con:basepoint}, the second divided by $|D_{n_0}|$. Integrals of nonlinearities over intervals are unchanged by the affine rescaling of Definition \ref{def:renorm}, so its constant $K_0$ is the tail above, which is at most one. Applying Theorem \ref{thm:main} to this pair with $K_0$ replaced by $1$, and using that $\xi_n$ is its normalized pair of stage $n-n_0$ for $n\ge n_0$, we obtain
\[
d_1\big(\xi_n,\mathcal{O}_n\big)\ \le\ C\,\xi_A^{\,n}\qquad\text{for all } n\ge0,
\]
where $\xi_A\in(0,1)$ is the rate of Theorem \ref{thm:main} computed from $A$, $L(A)$ and $K_0=1$, and only $C$ depends on the pair and on the base point.
\end{remark}

\begin{proposition}\label{prop:C2}
Under the hypotheses of Theorem \ref{thm:main} the renormalizations converge to the rotation pairs in the $C^2$ sense as well. The normalized nonlinearities $\eta_{\widehat{u_i^{(n)}}}$ of the three basis maps on $D_n$ and the nonlinearities of the four branches of $\xi_n$ are bounded in absolute value by $C\xi^{\,n}$ for all $n$, with $\xi$ the rate of Theorem \ref{thm:main}.
\end{proposition}

\begin{proof}
By Proposition \ref{prop:zrecursion}, $\sup_{\widetilde I_n}|\eta_{u_i^{(n)}}|\le C\xi_1^{\,n}/|D_n|$ for $i=1,2,3$, and passing to the normalized coordinate of $D_n$ multiplies nonlinearities by $|D_n|$. The first branches of $P_n$ and $Q_n$ are restrictions of $u_2^{(n)}$ and $u_3^{(n)}$ to subintervals of $D_n$. The second branch of $P_n$ is $(u_1^{(n)})^{-1}u_2^{(n)}$ on $[c_2(n),u_1^{(n)}x_0)$, with nonlinearity
\[
\big(\eta_{(u_1^{(n)})^{-1}}\circ u_2^{(n)}\big)\,Du_2^{(n)}+\eta_{u_2^{(n)}},\qquad
\eta_{(u_1^{(n)})^{-1}}(z)\ =\ -\eta_{u_1^{(n)}}\big((u_1^{(n)})^{-1}z\big)\,D(u_1^{(n)})^{-1}(z) .
\]
The points $z=u_2^{(n)}(x)$ with $x\in[c_2(n),u_1^{(n)}x_0)$ lie in $[u_1^{(n)}x_0,u_2^{(n)}u_1^{(n)}x_0)\subset u_1^{(n)}(D_n)\subset\widetilde I_n$, and $(u_1^{(n)})^{-1}z\in D_n$, so all nonlinearities involved are evaluated in $\widetilde I_n$ and all derivatives are bounded by $e^{2K_0}$ by Proposition \ref{prop:derivbound}. The same applies to the second branch of $Q_n$.
\end{proof}

\begin{remark}\label{rem:regularity}
The hypothesis that $f$ and $g$ are of class $C^3$ enters only through the constant $E''$ of Lemma \ref{lem:moebius}, where it gives the pointwise estimate $|\eta_{\widehat\varphi}-w|\le|J|^2E''$. If $f$ and $g$ are of class $C^{2+\alpha}$ for some $\alpha\in(0,1]$, so that the nonlinearities of the four letters are $\alpha$ H\"older continuous with a constant $E''_\alpha$, the same argument gives $|\eta_{\widehat\varphi}-w|\le|J|^{1+\alpha}E''_\alpha$, hence $\sup|\eta_{\widehat\varphi}-\eta_{m_w}|\le C_2|J|^{1+\alpha}$ with $C_2=C_2(E',E''_\alpha)$. In Proposition \ref{prop:basemoebius} the sum $\sum_s|p_s(J)|^{1+\alpha}$ is at most $\theta_J^{\alpha}K_1$, so its conclusions hold with $\theta(n)^\alpha$ and $\widetilde\theta(n)^\alpha$ in place of $\theta(n)$ and $\widetilde\theta(n)$. Every later statement uses Proposition \ref{prop:basemoebius} only through these error terms, which still decay exponentially under simultaneously bounded type and along the good stages of Section \ref{sec:beyond}. Consequently all results of this paper hold for commuting pairs of class $C^{2+\alpha}$, with rates and exponents depending on $\alpha$. This is the regularity of the rank one theorem of Khanin and Teplinsky \cite{KhaninTeplinsky}.
\end{remark}

\section{Rigidity}\label{sec:rigidity}

Exponential convergence upgrades the topological conjugacy of Lemma \ref{lem:toplin} to a $C^1$ conjugacy. Two ingredients are needed beyond Theorem \ref{thm:main}: the distortion of the tile words on their base intervals tends to zero, and the logarithmic density of $\mu$ read on the tiles is controlled, from one stage to the next and across tile boundaries, by the quantities $\delta_n$ of the proof of Theorem \ref{thm:main}. Both ingredients are stated so as to be usable in Section \ref{sec:beyond}, where bounded type is only available along a subsequence of stages.

\begin{proposition}\label{prop:tiledistortion}
Suppose that $\Sigma^{\ast}<\infty$ and that the base point is admissible. For every $n$ and every $w\in\mathcal{W}_n^{(i)}$, $i=1,2,3$,
\[
\Dist\big(w|_{D_n}\big)\ \le\ \sum_{m<n}a_{m+1}\,g_m\Big(\big(\tfrac{t_1(n)}{t_3(n)}+2\big)\,T_{\max}(n)\Big),
\qquad\text{where}\quad g_m(\epsilon)\ :=\ \sup_{S\subset\widetilde I_m,\ |S|\le\epsilon}\ \int_S\big|\eta_{u_2^{(m)}}\big| ,
\]
the supremum being over measurable subsets $S$ of the $\widetilde I_m$ of length at most $\epsilon$.
Consequently $\delta_n^{\ast}:=\sup\{\Dist(w|_{D_n})\colon i=1,2,3,\ w\in\mathcal{W}_n^{(i)}\}$ tends to zero along every sequence of stages $n$ for which $t_1(n)/t_3(n)$ stays bounded and $T_{\max}(n)\to0$. In particular $\delta_n^{\ast}\to0$ as $n\to\infty$ if $(\alpha,\beta)$ is of simultaneously bounded type.
\end{proposition}

\begin{proof}
Fix $w$ and write $w=\prod_{m<n}(u_2^{(m)})^{j_m}$ as in Lemma \ref{lem:factorization}. Put $J_n=D_n$ and $J_m=(u_2^{(m)})^{j_m}(J_{m+1})$ for $m<n$. As in the proof of Proposition \ref{prop:denjoy}, $J_{m+1}\subset\widetilde I_{m+1}$ and
\[
\Dist\big(w|_{D_n}\big)\ \le\ \sum_{m<n}\ \sum_{0\le k<j_m}\int_{(u_2^{(m)})^{k}(J_{m+1})}\big|\eta_{u_2^{(m)}}\big| ,
\]
where the intervals $(u_2^{(m)})^{k}(J_{m+1})$ lie in $\widetilde I_m$ by Lemma \ref{lem:windows}(b). Each of them is the image of $D_n$ under an element of $\Gamma$, hence an interval of $\mu$ measure $t_1(n)$. Every stage $n$ tile has $\mu$ measure at least $t_3(n)$, so such an interval contains at most $t_1(n)/t_3(n)$ tiles of stage $n$, meets at most $t_1(n)/t_3(n)+2$ of them, and has length at most $(t_1(n)/t_3(n)+2)T_{\max}(n)$. Hence the inner sum is at most $j_m\,g_m\big((t_1(n)/t_3(n)+2)T_{\max}(n)\big)$, and $j_m\le a_{m+1}$ gives the first statement. For each fixed $m$ the absolute continuity of the integral gives $g_m(\epsilon)\to0$ as $\epsilon\to0$, and $g_m(\epsilon)\le\mu_2(m)$ for every $\epsilon$. Since $\sum_ma_{m+1}\mu_2(m)\le K_0<\infty$, dominated convergence for series gives the second statement, and Lemma \ref{lem:scales} and Proposition \ref{prop:ratios} give the last.
\end{proof}

We now introduce the density functions. Fix the normalization $h(x)=\mu([x_0,x))$ for $x\in[x_0,x_0+1)$, which defines a homeomorphism of $\T$ conjugating $f$ and $g$ to $R_\alpha$ and $R_\beta$ because $\mu$ is invariant under both. For $x\in\T$ and $n\ge0$ let $J_n(x)=w_n(I_n^{(i_n)})$ be the stage $n$ tile containing $x$, with $w_n=w_n(x)\in\mathcal{W}_n^{(i_n)}$, and put $y_n(x)=w_n^{-1}(x)\in I_n^{(i_n)}$. Define
\[
a_n(x)\ =\ \log\frac{|J_n(x)|}{\mu(J_n(x))},\qquad
b_n(x)\ =\ \log Dw_n\big(y_n(x)\big)+\log\big|I_n^{(i_n)}\big|-\log t_{i_n}(n).
\]
Since $|J_n(x)|=\int_{I_n^{(i_n)}}Dw_n$ and $\mu(J_n(x))=t_{i_n}(n)$,
\begin{equation}\label{eq:ab}
\big|a_n(x)-b_n(x)\big|\ \le\ \Dist\big(w_n|_{I_n^{(i_n)}}\big)\ \le\ \delta_n^{\ast}\qquad\text{for all } x .
\end{equation}
The function $a_n$ is constant on each stage $n$ tile, and $b_n$ is continuous on the interior of each stage $n$ tile. Let $\delta_n=|W_1(n)|+|W_2(n)|+|W_3(n)|+|\Sigma(n)|+\theta(n)$ be as in the proof of Theorem \ref{thm:main}, and recall from the end of Section \ref{sec:moebius} that the estimates \eqref{eq:mainestimates} and \eqref{eq:ratioestimates} hold at every stage, with a constant $C$ depending only on the constants of Convention \ref{con:basepoint}. Put
\[
\Delta_n\ =\ \frac{a_{n+1}\,t_1(n)}{\tau(v_n)}\ \ge\ a_{n+1}\,(a_{n+1}+1),
\]
the drop of scale at stage $n$. The inequality holds because $t_1(n)=a_{n+1}t_2(n)+\tau(v_n)\ge(a_{n+1}+1)\tau(v_n)$. For a vector of simultaneously bounded type with constant $A$ one has $\Delta_n\le A^3$ and $t_1(n)/t_3(n)\le A$ by Lemma \ref{lem:scales}.

\begin{lemma}\label{lem:increments}
Suppose that $\Sigma^{\ast}<\infty$ and that the base point is admissible. There are constants $C_6$ and $c_0>0$, depending only on the constants of Convention \ref{con:basepoint}, such that for every $n$ the following hold.
\begin{enumerate}
\item[(i)] If $J_n(x)$ is of type $2$ or $3$ then $b_{n+1}(x)=b_n(x)$.
\item[(ii)] If $J_n(x)$ is of type $1$ and $\Delta_n\delta_n\le c_0$, then $|b_{n+1}(x)-b_n(x)|\le C_6\,\Delta_n\,\delta_n$.
\item[(iii)] Let $J=w(I_n^{(i)})$ and $J'=w'(I_n^{(i')})$ be adjacent stage $n$ tiles, $J'$ to the right of $J$, with common endpoint $p$. Then
\[
\lim_{x\downarrow p}b_n(x)-\lim_{x\uparrow p}b_n(x)\ =\ \log Du_i^{(n)}(x_0)+\log\frac{r_{i'}(n)}{r_i(n)}-\log\frac{t_{i'}(n)}{t_i(n)} ,
\]
and the modulus of this jump is at most $C_6\,\delta_n\,t_1(n)/t_3(n)$ provided $\delta_n\,t_1(n)/t_3(n)\le c_0$.
\end{enumerate}
\end{lemma}

\begin{proof}
(i) A tile of type $2$ or $3$ at stage $n$ is a tile of stage $n+1$ with the same word, and its base interval and its $\mu$ measure are unchanged, by Theorem \ref{thm:tiling}. So all three terms of $b_n(x)$ are unchanged.

(ii) Let $J_n(x)=w(D_n)$ and let $a=a_{n+1}$. By Theorem \ref{thm:tiling} and Lemma \ref{lem:decomposition}, $J_{n+1}(x)=w(u_2^{(n)})^{j}(I_{n+1}^{(i')})$ for some $0\le j\le a$, where $I_{n+1}^{(i')}=I_n^{(2)}$ and $i'=1$ if $j<a$, and $I_{n+1}^{(i')}=I'=[x_0,v_nx_0)$ if $j=a$. Then $w_{n+1}=w(u_2^{(n)})^{j}$, $y:=y_{n+1}(x)\in I_{n+1}^{(i')}$ and $y_n(x)=(u_2^{(n)})^{j}y$, so
\[
b_{n+1}(x)-b_n(x)\ =\ \log D(u_2^{(n)})^{j}(y)+\log\frac{|I_{n+1}^{(i')}|}{|D_n|}-\log\frac{t_{i'}(n+1)}{t_1(n)} .
\]
We work in the normalized coordinate of $D_n$ and write $u_2=u_2^{(n)}$, $r_2=r_2(n)$, $\delta=\delta_n$, and $C$ for the constant of \eqref{eq:mainestimates} and \eqref{eq:ratioestimates}. The points $u_2^{k}y$, $0\le k<j$, lie in $u_2^{k}(I_n^{(2)})\subset D_n$ by Lemma \ref{lem:decomposition}, so $|\log Du_2^{j}(y)|\le 2Cj\delta\le2Ca\delta$ as soon as $C\delta\le\tfrac12$. If $j<a$ the remaining terms are $\log r_2-\log(t_2(n)/t_1(n))$. Since $t_2(n)/t_1(n)\ge1/(a+1)$ and $|r_2-t_2(n)/t_1(n)|\le C\delta$, this is at most $2C(a+1)\delta$ in modulus as soon as $C(a+1)\delta\le\tfrac12$. If $j=a$ we compare $|I'|/|D_n|$ with $\tau(v_n)/t_1(n)$. For $k<a$ we have $|u_2^{k}(I_n^{(2)})|=|I_n^{(2)}|e^{\vartheta_k}$ with $|\vartheta_k|\le2Ca\delta$, so $\sum_{k<a}|u_2^{k}(I_n^{(2)})|=ar_2|D_n|(1+\vartheta)$ with $|\vartheta|\le e^{2Ca\delta}-1\le4Ca\delta$ if $2Ca\delta\le1$. Hence $|u_2^{a}(I')|=|D_n|(1-ar_2-ar_2\vartheta)$, while $|I'|=|u_2^{a}(I')|e^{-\vartheta_a}$ with $|\vartheta_a|\le2Ca\delta$. Moreover $1-ar_2=\tau(v_n)/t_1(n)+\vartheta'$ with $|\vartheta'|=a|r_2-t_2(n)/t_1(n)|\le Ca\delta$, and $ar_2\le1+Ca\delta\le2$. Therefore
\[
\frac{|I'|/|D_n|}{\tau(v_n)/t_1(n)}\ =\ \Big(1+\frac{\vartheta'-ar_2\vartheta}{\tau(v_n)/t_1(n)}\Big)e^{-\vartheta_a},\qquad
\Big|\frac{\vartheta'-ar_2\vartheta}{\tau(v_n)/t_1(n)}\Big|\ \le\ \frac{9Ca\delta\,t_1(n)}{\tau(v_n)}\ =\ 9C\Delta_n\delta ,
\]
and if $9C\Delta_n\delta\le\tfrac12$ the logarithm of the right hand side is at most $18C\Delta_n\delta+2Ca\delta$ in modulus. Since $a+1\le a(a+1)\le\Delta_n$, every smallness condition used holds when $\Delta_n\delta\le c_0:=1/(18C)$, and all the bounds are at most $C_6\Delta_n\delta$ with $C_6=24C$.

(iii) If $p$ is interior to the fundamental domain $[x_0,x_0+1)$ then $wu_i^{(n)}x_0=p=w'x_0$, so $w'^{-1}wu_i^{(n)}$ fixes $x_0$, has translation number zero by Lemma \ref{lem:forward}, and is the identity by the injectivity of $\tau$. Thus $w'=wu_i^{(n)}$. If $p=x_0+1$ then $J'$ is the first tile $I_n^{(i')}$, $w'$ is the identity, and $wu_i^{(n)}=T$ by the same argument. In both cases $Dw'(x_0)=Dw(u_i^{(n)}x_0)\,Du_i^{(n)}(x_0)$. On $J$ we have $b_n=\log Dw(w^{-1}\,\cdot\,)+\log|I_n^{(i)}|-\log t_i(n)$ with $w^{-1}p=u_i^{(n)}x_0$, and on $J'$ we have $b_n=\log Dw'(w'^{-1}\,\cdot\,)+\log|I_n^{(i')}|-\log t_{i'}(n)$ with $w'^{-1}p=x_0$. Subtracting gives the formula, since $|I_n^{(i')}|/|I_n^{(i)}|=r_{i'}(n)/r_i(n)$. For the bound, \eqref{eq:mainestimates} gives $|\log Du_i^{(n)}(x_0)|\le2C\delta_n$ when $C\delta_n\le\tfrac12$, and \eqref{eq:ratioestimates} together with $t_l(n)/t_1(n)\ge t_3(n)/t_1(n)$ gives $|\log r_l(n)-\log(t_l(n)/t_1(n))|\le2C\delta_n\,t_1(n)/t_3(n)$ for $l=i,i'$ when $C\delta_n\,t_1(n)/t_3(n)\le\tfrac12$.
\end{proof}

\begin{theorem}\label{thm:rigidity}
Under the hypotheses of Theorem \ref{thm:main}, the conjugacy $h$ of Lemma \ref{lem:toplin} is a $C^1$ diffeomorphism of $\T$. Consequently $f$ and $g$ are simultaneously $C^1$ conjugate to $R_\alpha$ and $R_\beta$.
\end{theorem}

\begin{proof}
By Lemma \ref{lem:scales}, $\Delta_n\le A^3$ and $t_1(n)/t_3(n)\le A$, and $\delta_n\le C\xi^n$ by the proof of Theorem \ref{thm:main}. So there is $n_0$ such that the smallness conditions of Lemma \ref{lem:increments}(ii) and (iii) hold for all $n\ge n_0$, and for these $n$ we get $|b_{n+1}(x)-b_n(x)|\le C_6A^3\delta_n\le C\xi^n$ for all $x$, and jumps of $b_n$ across tile boundaries of modulus at most $C_6A\delta_n\le C\xi^n$. Each $b_n$ is a bounded function, so $b_n$ converges uniformly on $\T$ to a bounded function $a$, with $|a-b_n|\le R_n:=C\sum_{m\ge n}\xi^m\le C'\xi^n$ for $n\ge n_0$, and by \eqref{eq:ab} and Proposition \ref{prop:tiledistortion} also $a_n\to a$ uniformly.

\emph{Continuity.} Two points $x,z$ in the same stage $n$ tile satisfy $|b_n(x)-b_n(z)|\le\Dist(w_n|_{I_n^{(i_n)}})\le\delta_n^{\ast}$, and two points in adjacent stage $n$ tiles satisfy $|b_n(x)-b_n(z)|\le2\delta_n^{\ast}+C\xi^n$ by Lemma \ref{lem:increments}(iii). In both cases $|a(x)-a(z)|\le2R_n+2\delta_n^{\ast}+C\xi^n=:\omega_n$, and $\omega_n\to0$ by Proposition \ref{prop:tiledistortion}. Since there are finitely many tiles of stage $n$, they have a positive minimal length $c_n$, and two points at distance less than $c_n$ lie in the same or in adjacent stage $n$ tiles. So $a$ is uniformly continuous, with modulus of continuity $\varpi(s)\le\omega_n$ for $s<c_n$.

\emph{Differentiability.} Let $x\in\T$ and let $y\to x$. By symmetry we treat $y>x$ in a local lift and consider the interval $[x,y)$, whose $\mu$ measure is $h(y)-h(x)$. Let $c'_M$ be the minimal length of a tile of stage at most $M$, and let $M=M(y)$ be the largest integer with $|y-x|<c'_M$. Then $M(y)\to\infty$ as $y\to x$, since $c'_M>0$ for every $M$. Let $\mathcal{P}$ be the collection of tiles contained in $[x,y)$ that are maximal for inclusion. Two tiles are either nested or disjoint, so $\mathcal{P}$ consists of pairwise disjoint tiles. Every point $z\in(x,y)$ lies in its stage $m$ tile $J_m(z)$, which is contained in $(x,y)$ for $m$ large because $|J_m(z)|\le T_{\max}(m)\to0$, so $\mathcal{P}$ covers $(x,y)$, and every $J\in\mathcal{P}$ has stage larger than $M$. Since $\mu$ has no atoms,
\[
\mu([x,y))\ =\ \sum_{J\in\mathcal{P}}\mu(J),\qquad |y-x|\ =\ \sum_{J\in\mathcal{P}}|J| .
\]
For $J\in\mathcal{P}$ of stage $m>M$ and $z\in J$ we have $\mu(J)/|J|=e^{-a_m(z)}$ and $|a_m(z)-a(x)|\le|a_m(z)-a(z)|+|a(z)-a(x)|\le R_m+\delta_m^{\ast}+\varpi(|y-x|)$. Put $\epsilon(y)=\sup_{m>M(y)}(R_m+\delta_m^{\ast})+\varpi(|y-x|)$, which tends to zero as $y\to x$. Then $e^{-a(x)-\epsilon(y)}|J|\le\mu(J)\le e^{-a(x)+\epsilon(y)}|J|$ for every $J\in\mathcal{P}$, and summing over $\mathcal{P}$,
\[
e^{-\epsilon(y)}\ \le\ \frac{h(y)-h(x)}{y-x}\,e^{a(x)}\ \le\ e^{\epsilon(y)} .
\]
So $h$ is differentiable at $x$ with $Dh(x)=e^{-a(x)}>0$. The derivative is continuous, so $h$ is a $C^1$ diffeomorphism of $\T$, and so is $h^{-1}$ since $Dh$ is continuous and nowhere zero.
\end{proof}

Theorem B is Theorem \ref{thm:rigidity} together with Theorem \ref{thm:holder} below. The proof gives a little more than stated: the logarithmic density $-\log Dh$ is the uniform limit of the functions $a_n$ read on the tiles, with $|a-a_n|\le C\xi^n+\delta_n^{\ast}$. We now show that the distortion of the tile words decays exponentially, which upgrades the conjugacy to $C^{1+\gamma}$ and removes the base point from the conclusion of Theorem \ref{thm:main}.

\begin{lemma}\label{lem:expdistortion}
Suppose that $(\alpha,\beta)$ is of simultaneously bounded type with constant $A$ and summably convergent, and that the base point is admissible. There is $c_{\ast}>0$ with $c_{\ast}\le|J|/\mu(J)\le c_{\ast}^{-1}$ for every tile $J$ of every stage. Consequently, for every $\xi'\in(\max(\xi_1,\lambda),1)$, where $\xi_1$ is the rate of Proposition \ref{prop:zrecursion} and $\lambda=1-\tfrac{1}{3A}$ is the factor of Lemma \ref{lem:scales}, there is $C$ such that
\[
\delta_n^{\ast\ast}\ :=\ \sup\big\{\Dist\big(w|_{[y,\,u_1^{(n)}y)}\big)\colon\ i=1,2,3,\ w\in\mathcal{W}_n^{(i)},\ y\in D_n\big\}\ \le\ C\,\xi'^{\,n}\qquad\text{for all } n .
\]
In particular $\delta_n^{\ast}\le\delta_n^{\ast\ast}\le C\xi'^{\,n}$.
\end{lemma}

\begin{proof}
By Theorem \ref{thm:rigidity} the functions $a_n$ converge uniformly to the continuous function $a=-\log Dh$, so they are uniformly bounded, and $|J|/\mu(J)=e^{a_n(x)}$ for every stage $n$ tile $J$ and every $x\in J$. This gives $c_{\ast}$. Now fix $n$, a tile word $w\in\mathcal{W}_n^{(i)}$ and $y\in D_n$, and put $J=[y,u_1^{(n)}y)$. Since $u_1^{(n)}y\le(u_1^{(n)})^2x_0<\Omega_nx_0$, the interval $J$ lies in $\widetilde I_n$, and it has $\mu$ measure $t_1(n)$. Run the proof of Proposition \ref{prop:tiledistortion} with $J_n=J$ in place of $D_n$. Lemma \ref{lem:windows}(b) gives $J_m\subset\widetilde I_m$ for all $m\le n$, and the pieces $(u_2^{(m)})^k(J_{m+1})$ are images of $J$ under elements of $\Gamma$, hence intervals of $\mu$ measure $t_1(n)$ contained in $\widetilde I_m$. Such an interval meets at most $A+2$ tiles of stage $n$ by the argument of Corollary \ref{cor:mesh}, so its length is at most $(A+2)c_{\ast}^{-1}t_1(n)$. By Proposition \ref{prop:zrecursion}, $\sup_{\widetilde I_m}|\eta_{u_2^{(m)}}|\le C\xi_1^{\,m}/|D_m|$, and $|D_m|\ge c_{\ast}t_1(m)$ because $D_m$ is a tile. Hence each piece of stage $m$ contributes at most $C c_{\ast}^{-2}(A+2)\,\xi_1^{\,m}\,t_1(n)/t_1(m)$ to the distortion, and there are at most $a_{m+1}\le A$ such pieces. By Lemma \ref{lem:scales}, $t_1(n)/t_1(m)\le3A\,S(n)/S(m)\le3A\lambda^{\,n-m}$. Summing over $m<n$,
\[
\Dist\big(w|_J\big)\ \le\ C\sum_{m<n}\xi_1^{\,m}\lambda^{\,n-m}\ \le\ Cn\max(\xi_1,\lambda)^{\,n}\ \le\ C'\xi'^{\,n} .\qedhere
\]
\end{proof}

\begin{theorem}\label{thm:holder}
Under the hypotheses of Theorem \ref{thm:main} the conjugacy $h$ of Lemma \ref{lem:toplin} is a $C^{1+\gamma}$ diffeomorphism of $\T$ with
\[
\gamma\ =\ \frac{\log(1/\xi'')}{\log A}\ >\ 0,\qquad \xi''=\max(\xi,\xi'),
\]
where $\xi$ is the rate of Theorem \ref{thm:main} and $\xi'$ that of Lemma \ref{lem:expdistortion}. The exponent $\gamma$ can be chosen to depend only on $A$.
\end{theorem}

\begin{proof}
We use the notation of the proof of Theorem \ref{thm:rigidity}, where $|a-b_n|\le R_n\le C\xi^{\,n}$ for $n\ge n_0$, two points of a common stage $n$ tile satisfy $|b_n(x)-b_n(z)|\le\delta_n^{\ast}$, and two points of adjacent stage $n$ tiles satisfy $|b_n(x)-b_n(z)|\le2\delta_n^{\ast}+C\xi^{\,n}$. By Lemma \ref{lem:expdistortion} this gives $|a(x)-a(z)|\le\omega_n\le C\xi''^{\,n}$ whenever $x$ and $z$ lie in the same or in adjacent stage $n$ tiles and $n\ge n_0$. Every stage $n$ tile has $\mu$ measure at least $t_3(n)\ge t_1(n)/A\ge A^{-n-1}$, since $t_1(k+1)=t_2(k)\ge t_1(k)/A$ for every $k$, so its length is at least $c_{\ast}A^{-n-1}$ by Lemma \ref{lem:expdistortion}, and two points at distance less than $c_{\ast}A^{-n-1}$ lie in the same or in adjacent stage $n$ tiles. Let $x\ne z$ with $|x-z|<c_{\ast}A^{-n_0-1}$, and let $n\ge n_0$ be the largest integer with $|x-z|<c_{\ast}A^{-n-1}$, so that $c_{\ast}A^{-n-2}\le|x-z|$. Then
\[
|a(x)-a(z)|\ \le\ C\xi''^{\,n}\ =\ C\,\big(A^{-n}\big)^{\gamma}\ \le\ C\,\big(A^{2}|x-z|/c_{\ast}\big)^{\gamma}\ \le\ C'|x-z|^{\gamma} ,
\]
and the bound is trivial for $|x-z|\ge c_{\ast}A^{-n_0-1}$ because $a$ is bounded. So $a$ is $\gamma$ H\"older on $\T$, and so is $Dh=e^{-a}$.

For the dependence of $\gamma$ on $A$ alone, apply this to the pair at the stage of Remark \ref{rem:uniformrate} from which the tail of the series defining $K_0$ is at most one, call it $n_1$. Its conjugacy is the restriction of $h$ to $D_{n_1}$ up to affine changes of coordinates, so $a$ is $\gamma_A$ H\"older on $D_{n_1}$, where $\gamma_A=\log(1/\xi''_A)/\log A$ is computed from $A$, $L(A)$ and $K_0=1$. For a tile word $w$ of stage $n_1$ the identity $h\circ w=R_{\tau(w)}\circ h$ gives $a(wx)=a(x)+\log Dw(x)$, and $\log Dw$ is of class $C^2$, so $a$ is $\gamma_A$ H\"older on every stage $n_1$ tile. Since these finitely many tiles cover $\T$ and $a$ is continuous, $a$ is $\gamma_A$ H\"older on $\T$.
\end{proof}

\begin{proposition}\label{prop:allbasepoints}
Suppose that $\Sigma^{\ast}<\infty$ and that $x_0$ is an admissible base point. For $x\in\T$ let $\xi_n(x)$ be the normalized pair of stage $n$ built on the base point $x$ in place of $x_0$, with the same basis $u^{(n)}$. There is $C$, depending only on the constants of Convention \ref{con:basepoint}, such that
\[
d_1\big(\xi_n(x),\mathcal{O}_n\big)\ \le\ C\,\big(\delta_n+\delta_n^{\ast\ast}\big)\qquad\text{for all } n\ge0 \text{ and all } x\in\T ,
\]
with $\delta_n^{\ast\ast}$ as in Lemma \ref{lem:expdistortion}. Under the hypotheses of Theorem \ref{thm:main} the right hand side is at most $C\xi''^{\,n}$, with $\xi''$ as in Theorem \ref{thm:holder}, so the conclusion of Theorem \ref{thm:main} holds for every base point, with a constant independent of the base point.
\end{proposition}

\begin{proof}
Fix $n$ and $x\in\T$, and let $w(I_n^{(i)})$ be the stage $n$ tile containing $x$, so that $x=w(y)$ with $y\in I_n^{(i)}\subset D_n$. Since $w$ commutes with the basis, $D_n(x)=[x,u_1^{(n)}x)=w(D_n(y))$, the cut points in the chart at $x$ are the images under $w$ of the cut points in the chart at $y$, and the branches in the chart at $x$ are the conjugates by $w$ of the branches in the chart at $y$. After normalization, $\widetilde P_n^{(x)}=\widehat w\circ\widetilde P_n^{(y)}\circ\widehat w^{-1}$, where $\widehat w$ is the normalization of $w|_{D_n(y)}$, a diffeomorphism of $[0,1]$ fixing both endpoints, regarded as a homeomorphism of the circle fixing the seam, and likewise for $\widetilde Q_n$. The distortion of $\widehat w$ is $\Dist(w|_{D_n(y)})\le\delta_n^{\ast\ast}$, and the proof of Lemma \ref{lem:normclose}, which uses only the oscillation of $\log D\widehat w$, gives $\|\widehat w-\mathrm{id}\|_{C^1}\le3\delta_n^{\ast\ast}$ when $\delta_n^{\ast\ast}\le1$. The circle maps $\widetilde P_n^{(y)}$ and $\widetilde Q_n^{(y)}$ have derivatives bounded by $e^{4K_0}$ and nonlinearities bounded by a constant, by Propositions \ref{prop:derivbound} and \ref{prop:zrecursion}, since $D_n(y)\subset\widetilde I_n$ and $|D_n(y)|\le(1+e^{2K_0})|D_n|$. Conjugating such a map by a diffeomorphism within $\epsilon$ of the identity in the $C^1$ norm moves it by at most $C\epsilon$ in the distance $d_1$, so $d_1(\xi_n(x),\xi_n(y))\le C\delta_n^{\ast\ast}$.

It remains to compare $\xi_n(y)$ with $\xi_n(x_0)$ for $y\in D_n$. Work in the normalized coordinate of $D_n$, in which $y$ becomes a point $s_0\in[0,1)$ and $u_1^{(n)}$ becomes a map $U$ with $|U(s)-s-1|\le C\delta_n$ and $|DU-1|\le C\delta_n$, by \eqref{eq:mainestimates}. The chart at $y$ is the interval $[s_0,U(s_0))$, of length $\ell=U(s_0)-s_0=1+O(\delta_n)$, and the transition from the chart at $x_0$ to the chart at $y$, as a homeomorphism of the circle $[0,1)/\{0\sim1\}$, is $\Psi(s)=(s-s_0)/\ell$ for $s\in[s_0,1)$ and $\Psi(s)=(U(s)-s_0)/\ell$ for $s\in[0,s_0)$. Both branches are within $C\delta_n$ in the $C^1$ norm of the rotation $R\colon s\mapsto s-s_0$ of the circle, and $\widetilde P_n^{(y)}=\Psi\circ\widetilde P_n^{(x_0)}\circ\Psi^{-1}$. Conjugation by $R$ does not change the $d_1$ distance to the rotation pair $\mathcal{O}_n$, because rotations commute and $d_1$ is built from supremum norms, so the perturbation argument above gives $d_1(\xi_n(y),\mathcal{O}_n)\le d_1(\xi_n(x_0),\mathcal{O}_n)+C\delta_n\le C\delta_n$, the last step by the proof of Theorem \ref{thm:main}. Together, $d_1(\xi_n(x),\mathcal{O}_n)\le C(\delta_n+\delta_n^{\ast\ast})$. Under the hypotheses of Theorem \ref{thm:main}, $\delta_n\le C\xi^{\,n}$ by the proof of that theorem and $\delta_n^{\ast\ast}\le C\xi'^{\,n}$ by Lemma \ref{lem:expdistortion}.
\end{proof}

\section{Renormalization beyond simultaneously bounded type}\label{sec:beyond}

The class of Definition \ref{def:sbt} is defined by a condition that must hold at every stage. We now relax it to a condition that must hold along a sparse but not too sparse sequence of stages. The distortion part of the real bounds, Propositions \ref{prop:denjoy} and \ref{prop:derivbound}, the M\"obius approximation, Proposition \ref{prop:basemoebius}, and the two recursions, Propositions \ref{prop:zrecursion} and \ref{prop:sloperecursion}, hold at every stage as soon as the digit weighted series $\Sigma^{\ast}$ converges and the base point is admissible. Only the geometric part, Proposition \ref{prop:ratios}, uses bounded type, and it will be replaced by a statement at good stages. The right language for the arithmetic side is the space of lattices, because the arithmetic of a vector is visible there and is not visible in the symbolic data, as Remark \ref{rem:asymmetry} explains.

\subsection{The lattice flow and the Dani correspondence}
Let $\omega=(\alpha,\beta)$ and let
\[
\Lambda\ =\ \big\{(k\alpha+l\beta+m,\,k,\,l)\big\},\qquad
\Lambda^{\ast}\ =\ \big\{(p_1-q\alpha,\,p_2-q\beta,\,q)\big\}
\]
be the two unimodular lattices attached to $\omega$, the first adapted to the linear form $k\alpha+l\beta$ and the second to the simultaneous approximation of $\alpha$ and $\beta$ by rationals with a common denominator. They are dual to one another up to the flip of coordinates. Let
\[
a_t\ =\ \operatorname{diag}\big(e^{2t},e^{-t},e^{-t}\big),\qquad
a_t^{\ast}\ =\ \operatorname{diag}\big(e^{t},e^{t},e^{-2t}\big)
\]
be the corresponding one parameter diagonal subgroups of $\mathrm{SL}_3(\R)$, acting on the space $X=\mathrm{SL}_3(\R)/\mathrm{SL}_3(\Z)$ of unimodular lattices. The dictionary between Diophantine properties of $\omega$ and the behaviour of these orbits goes back to \cite{Dani}, and the transference between the two approximation problems is classical, see \cite[Chapter V]{Cassels}. Recall that $\omega$ satisfies the simultaneous Diophantine condition with exponent $\sigma$ if there is $c>0$ with $\max(\|q\alpha\|,\|q\beta\|)\ge c\,q^{-\sigma}$ for all $q\ge1$, and that $\omega$ is badly approximable if this holds with the least possible exponent $\sigma=1/2$.

\begin{lemma}\label{lem:dani}
The vector $\omega$ satisfies the simultaneous Diophantine condition with exponent $\sigma\ge1/2$ if and only if there are $c>0$ and $t_0$ with
\[
\lambda_1\big(a_t^{\ast}\Lambda^{\ast}\big)\ \ge\ c\,e^{-\delta t},\qquad \delta=\frac{2\sigma-1}{\sigma+1},
\]
for all $t\ge t_0$, where $\lambda_1$ denotes the length of a shortest nonzero vector. In particular $\omega$ is badly approximable if and only if the orbit $\{a^{\ast}_t\Lambda^{\ast}\colon t\ge0\}$ is bounded in $X$, and this happens if and only if $\{a_t\Lambda\colon t\ge0\}$ is bounded in $X$.
\end{lemma}

\begin{proof}
A nonzero vector of $a_t^{\ast}\Lambda^{\ast}$ is $(e^{t}(p_1-q\alpha),e^{t}(p_2-q\beta),e^{-2t}q)$. Its length is at least $c e^{-\delta t}$ for all large $t$ if and only if for every $q\ge1$ and every $p$, taking $e^{-2t}q=e^{-\delta t}$, that is $e^{t}=q^{1/(2-\delta)}$, one has $\max_i|p_i-q\omega_i|\ge ce^{-(1+\delta)t}=cq^{-(1+\delta)/(2-\delta)}$. Solving $(1+\delta)/(2-\delta)=\sigma$ for $\delta$ gives $\delta=(2\sigma-1)/(\sigma+1)$, and the two conditions are equivalent up to the value of $c$ by the standard interpolation between the times $t$ and the denominators $q$. The case $\delta=0$ is the case $\sigma=1/2$ and gives bounded orbits, by Mahler's criterion.

For the last equivalence, let $L$ be a unimodular lattice with successive minima $\lambda_1\le\lambda_2\le\lambda_3$. By Minkowski's second theorem $\lambda_1\lambda_2\lambda_3\asymp1$, and for the dual lattice $\lambda_1(L^{\ast})\asymp\lambda_3(L)^{-1}$. So $\lambda_1(L)$ is bounded below if and only if all three minima are comparable to one, which happens if and only if $\lambda_1(L^{\ast})$ is bounded below. Since $a_t\Lambda$ and $a_t^{\ast}\Lambda^{\ast}$ are dual up to a permutation of coordinates, the two orbits are bounded simultaneously.
\end{proof}

\begin{remark}\label{rem:split}
The passage to the space of lattices is not a convenience. The symbolic data of the algorithm determines the vector, by Theorem \ref{thm:reconstruction}, but it does not display its Diophantine class in any usable form, and Remark \ref{rem:asymmetry} records that this is a real difference from rank one. The lattice orbit displays the Diophantine class directly, by Lemma \ref{lem:dani}. The renormalization scheme lives on the symbolic side and the arithmetic lives on the geometric side, and the two are joined by the rescaled basis $b_i(n)$, which is a point of the orbit up to bounded distance.
\end{remark}

\begin{proposition}\label{prop:sbtbad}
Suppose that the Brun basis of $(\alpha,\beta)$ is reduced at its own scale along a sequence of stages $n_k$ with $t_1(n_{k+1})\ge c\,t_1(n_k)$, in the sense that $t_1(n_k)H(n_k)^2\le A'$ for all $k$. Then $(\alpha,\beta)$ is badly approximable.
\end{proposition}

\begin{proof}
Let $g_k=\operatorname{diag}\big(t_1(n_k)^{-1},H(n_k)^{-1},H(n_k)^{-1}\big)$. The basis $b_i(n_k)$ of $\Lambda$ satisfies $|t_i(n_k)|\le t_1(n_k)$ and $|e_i(n_k)|_\infty\le H(n_k)$, so $g_kb_i(n_k)$ lies in the cube $[-1,1]^3$ and $\lambda_3(g_k\Lambda)\le\sqrt3$. The covolume of $g_k\Lambda$ is $(t_1(n_k)H(n_k)^2)^{-1}$, which lies in $[1/A',6]$ by hypothesis and Lemma \ref{lem:heights}. Minkowski's second theorem gives $\lambda_1\lambda_2\lambda_3\asymp\operatorname{covol}$, so $\lambda_1(g_k\Lambda)\ge c/(3A')$. Writing $g_k=\varrho_ka_{s_k}$ with $\varrho_k=(t_1(n_k)H(n_k)^2)^{-1/3}$ a scalar in a fixed compact subinterval of $(0,\infty)$ and $e^{3s_k}=H(n_k)/t_1(n_k)$, we get $\lambda_1(a_{s_k}\Lambda)\ge c'(A')$. The times $s_k$ increase to infinity with bounded gaps by the hypothesis on $t_1(n_k)$, and the flow moves lattices a bounded amount in bounded time, so $\lambda_1(a_t\Lambda)$ is bounded below for all $t\ge0$. Mahler's criterion and Lemma \ref{lem:dani} finish the proof.
\end{proof}

\begin{remark}\label{rem:notbad}
Definition \ref{def:sbt} alone is not known to imply bad approximability. A vector of simultaneously bounded type has comparable scales at every stage, but the Brun basis need not be reduced at its own scale: Lemma \ref{lem:heights} gives $t_1(n)H(n)^2\ge1/6$ in one direction only, and by Theorem \ref{thm:periodic} the quantity $t_1(n)H(n)^2$ is unbounded for the eventually periodic types whose period matrix has two real non-Perron eigenvalues of distinct moduli, although these vectors are of simultaneously bounded type. For $\nu$ almost every vector, $t_1(n)H(n)^2$ grows at the exponential rate $\theta_2-\theta_3$, where $\theta_1>\theta_2\ge\theta_3$ are the Lyapunov exponents of the cocycle of step matrices, so it is unbounded as soon as the two subdominant exponents are distinct. For the periodic types of Section \ref{sec:periodic}, whose period matrices have a pair of nonreal conjugate eigenvalues of equal modulus, the basis is reduced at every stage and Proposition \ref{prop:sbtbad} applies. For these vectors bad approximability also follows from the classical theorem, going back to Perron, that a vector whose coordinates together with $1$ form a basis of a real cubic field is badly approximable. So bad approximability is available for the periodic examples and is not claimed in general.
\end{remark}

\subsection{Good stages and fast renormalization}

Fix integers $A\ge1$ and $L\ge1$ and put $M=4L+6$.

\begin{definition}\label{def:balanced}
A stage $n\ge0$ is $(A,L)$ good if the symbols $(a_{n+1},\varepsilon_{n+1}),\dots,(a_{n+M},\varepsilon_{n+M})$ have digits at most $A$, and every $L$ consecutive letters among $\varepsilon_{n+1},\dots,\varepsilon_{n+M}$ contain the letter $B$.
\end{definition}

In the projective picture, $n$ is $(A,L)$ good if and only if $T^n(\alpha,\beta)$ lies in the set $\mathcal{G}=\mathcal{G}(A,L)\subset\Delta$ of points whose first $M$ symbols satisfy the two conditions. By Lemma \ref{lem:cylinders} the set $\mathcal{G}$ is a finite union of cylinders, hence open and nonempty, so $\nu(\mathcal{G})>0$. By Proposition \ref{prop:symbolic}, if $(\alpha,\beta)$ is of simultaneously bounded type with constant $A$ then every stage is $(A,L)$ good with $L=L(A)$ as in Section \ref{sec:realbounds}. At a good stage the geometry of Proposition \ref{prop:ratios} is available, in the following form.

\begin{lemma}\label{lem:goodstage}
Suppose that $\Sigma^{\ast}<\infty$, that the base point is admissible, and that $n$ is $(A,L)$ good. Put $A'=3(A+1)^{4L}$ and let $\rho_1,\rho_2,\rho_3$ be the constants of Proposition \ref{prop:ratios}, computed from $A$, $L$ and $K_0$. Then:
\begin{enumerate}
\item[(i)] $t_1(m)\le A'\,t_3(m)$ for $n\le m\le n+M-4L$, in particular $t_1(n)\le A't_3(n)$.
\item[(ii)] $\rho_1\le r_2(n)\le\rho_2$ and $r_3(n)\ge\rho_1^{\,L+2}$.
\item[(iii)] every stage $n$ tile is split, at some stage $n'\le n+L+2$, into pieces each occupying a fraction between $\rho_3$ and $1-\rho_3$ of it, hence $T_{\max}(n+L+3)\le(1-\rho_3)T_{\max}(n)$.
\end{enumerate}
\end{lemma}

\begin{proof}
(i) For $n\le m\le n+M-4L$ the block of $4L$ steps following stage $m$ has digits at most $A$ and contains at least four occurrences of $B$, so the product of its step matrices is strictly positive with entries at most $(A+1)^{4L}$, and the second half of the proof of Proposition \ref{prop:symbolic} gives $t_1(m)\le A't_3(m)$.

(ii) and (iii) The proof of Proposition \ref{prop:ratios} uses, besides Propositions \ref{prop:denjoy} and \ref{prop:derivbound}, which hold at every stage, only two facts about the stages $m$ it involves: that $a_{m+1}\le A$, which gives $r_2(m)\ge\rho_1$ through Proposition \ref{prop:derivbound}(ii), and that the letter $B$ occurs within $L$ steps after $m$, which is the content of Lemma \ref{lem:participation} at $m$. For $r_2(n)\ge\rho_1$ only the digit $a_{n+1}$ is needed. For $r_3(n)\ge\rho_1^{L+2}$ the stages $n\le m\le n+L+2$ are needed, for $r_2(n)\le\rho_2$, which rests on $|I'|\ge\rho_1^{L+2}|D_{n+1}|$, the stages $n\le m\le n+L+3$, and for the splitting of the stage $n$ tiles, which become of type $1$ at some stage $n'\le n+L+2$ and are then split with the fractions of Proposition \ref{prop:ratios}, the stages $n\le m\le n'+L+3\le n+2L+5$. All these stages use only the symbols $(a_{n+1},\varepsilon_{n+1}),\dots,(a_{n+M},\varepsilon_{n+M})$ controlled in Definition \ref{def:balanced}, since $M=4L+6$.
\end{proof}

Let $n_1<n_2<\cdots$ be the $(A,L)$ good stages, assumed infinite in number, and recall the drop of scale $\Delta_n=a_{n+1}t_1(n)/\tau(v_n)$ from Section \ref{sec:rigidity}. Put
\[
D_k\ =\ \max_{n_k\le n<n_{k+1}}\Delta_n ,
\]
the largest drop of scale between two consecutive good stages. In the projective coordinates $(x,y)=T^n(\alpha,\beta)$ one has $\Delta_n=a/(1-ax)$ with $a=\lfloor1/x\rfloor$, so $\Delta_n$ is large exactly when the digit is large or the remainder of the step is small.

\begin{definition}\label{def:fast}
The vector $\omega$ is fast renormalizable with constants $(A,L)$ if
\begin{enumerate}
\item[(a)] $\Sigma^{\ast}=\sum_{n\ge0}a_{n+1}t_1(n)H(n)<\infty$.
\item[(b)] the $(A,L)$ good stages are infinite in number and $n_k/k$ is bounded.
\item[(c)] $\log D_k=o(k)$ as $k\to\infty$.
\end{enumerate}
\end{definition}

A vector of simultaneously bounded type with constant $A$ which is summably convergent is fast renormalizable with constants $(A,L(A))$: every stage is good, $\Sigma^{\ast}\le A\Sigma$, and $\Delta_n\le A^3$. The fast renormalization is the sequence $\xi_{n_k}$ of renormalizations read only at the good stages. Between two good stages the operator is applied as before, and the geometric estimates are allowed to degrade by the factor $D_k$.

\begin{theorem}\label{thm:fastconv}
Let $\omega$ be fast renormalizable with constants $(A,L)$, and let $f$ and $g$ be commuting orientation preserving $C^3$ circle diffeomorphisms with rotation vector $\omega$. Then, for every admissible base point, there are $\lambda\in(0,1)$ and $C$ with
\[
d_1\big(\xi_{n_k},\mathcal{O}_{n_k}\big)\ \le\ C\lambda^{\,k}\qquad\text{for all } k ,
\]
and $f$ and $g$ are simultaneously $C^1$ conjugate to $R_\alpha$ and $R_\beta$.
\end{theorem}

\begin{proof}
\emph{Estimates at every stage.} By (a) and Proposition \ref{prop:basepoint} admissible base points exist, and Propositions \ref{prop:denjoy}, \ref{prop:derivbound}, \ref{prop:basemoebius}, the first parts of Propositions \ref{prop:zrecursion} and \ref{prop:sloperecursion}, and the estimates \eqref{eq:mainestimates}, \eqref{eq:ratioestimates} and $d_1(\xi_n,\mathcal{O}_n)\le C\delta_n$ of the proof of Theorem \ref{thm:main} hold at every stage $n$, with constants depending only on those of Convention \ref{con:basepoint}. In particular, for every $n$,
\begin{equation}\label{eq:generalrecursions}
|z_1(n+1)|\ \le\ r_2(n)^2\,|z_1(n)|+C_4\theta(n),\qquad
|\Sigma(n+1)|\ \le\ \kappa\big(r_2(n)\big)\,|\Sigma(n)|+C_5\epsilon_n ,
\end{equation}
with $r_2(n)\le1$ and $\kappa(r_2(n))\le1$, and $|W_i(n)|\le C(|z_1(n)|+\theta(n))$ for $i=1,2,3$.

\emph{The mesh along good stages.} By Lemma \ref{lem:goodstage}(iii), $T_{\max}(n_k+L+3)\le(1-\rho_3)T_{\max}(n_k)$ for every $k$. Since $n_{k+L+3}\ge n_k+L+3$ and $T_{\max}$ is nonincreasing in the stage, $T_{\max}(n_{k+L+3})\le(1-\rho_3)T_{\max}(n_k)$, hence $T_{\max}(n_k)\le C\lambda_0^{\,k}$ with $\lambda_0=(1-\rho_3)^{1/(L+3)}$. By Corollary \ref{cor:mesh} in its general form and Lemma \ref{lem:goodstage}(i), $\theta(n)\le(A'+2)T_{\max}(n_k)$ for $n_k\le n<n_{k+1}$. By (b), $N_k:=n_{k+1}-n_k\le n_{k+1}\le C(k+1)$.

\emph{The recursions along good stages.} Iterating \eqref{eq:generalrecursions} from $n_k$ to $n_{k+1}$, with $r_2(n_k)\le\rho_2$ at the first step by Lemma \ref{lem:goodstage}(ii) and $r_2\le1$ afterwards, gives
\[
|z_1(n_{k+1})|\ \le\ \rho_2^{\,2}\,|z_1(n_k)|+C_4N_k(A'+2)\,T_{\max}(n_k)\ \le\ \rho_2^{\,2}\,|z_1(n_k)|+C(k+1)\lambda_0^{\,k} ,
\]
so $|z_1(n_k)|\le C\lambda_1^{\,k}$ for any $\lambda_1$ with $\max(\rho_2^2,\lambda_0)<\lambda_1<1$, and for $n_k\le n<n_{k+1}$ also $|z_1(n)|\le|z_1(n_k)|+C_4N_k(A'+2)T_{\max}(n_k)\le C\lambda_1^{\,k}$, hence $|W_i(n)|\le C\lambda_1^{\,k}$ and $\epsilon_n\le C\lambda_1^{\,k}$. Iterating the second recursion in the same way, with $\kappa(r_2(n_k))\le\kappa(\rho_2)<1$ at the good stages, gives $|\Sigma(n)|\le C\lambda_2^{\,k}$ for $n_k\le n<n_{k+1}$ and any $\lambda_2$ with $\max(\kappa(\rho_2),\lambda_1)<\lambda_2<1$. Altogether
\begin{equation}\label{eq:deltafast}
\delta_n\ \le\ C\lambda_2^{\,k}\qquad\text{for } n_k\le n<n_{k+1} ,
\end{equation}
and $d_1(\xi_{n_k},\mathcal{O}_{n_k})\le C\delta_{n_k}$ gives the first claim with $\lambda=\lambda_2$.

\emph{The conjugacy.} By (c) and \eqref{eq:deltafast} we have $\Delta_n\delta_n\le CD_k\lambda_2^{\,k}\to0$ for $n_k\le n<n_{k+1}$, and $\delta_{n_k}t_1(n_k)/t_3(n_k)\le CA'\lambda_2^{\,k}\to0$ by Lemma \ref{lem:goodstage}(i). So there is $k_0$ such that Lemma \ref{lem:increments}(ii) applies at every stage $n\ge n_{k_0}$ and Lemma \ref{lem:increments}(iii) at every good stage $n_k$ with $k\ge k_0$. Since $N_kD_k=e^{o(k)}$ by (b) and (c),
\[
\sum_{n\ge n_{k_0}}\ \sup_x\big|b_{n+1}(x)-b_n(x)\big|\ \le\ \sum_{k\ge k_0}N_k\,C_6D_k\,C\lambda_2^{\,k}\ <\ \infty ,
\]
so $b_n$ converges uniformly to a bounded function $a$, with $|a-b_{n_k}|\le R'_k:=\sum_{n\ge n_k}\sup|b_{n+1}-b_n|\to0$. Proposition \ref{prop:tiledistortion} along the good stages, where $t_1(n_k)/t_3(n_k)\le A'$ and $T_{\max}(n_k)\to0$, gives $\delta^{\ast}_{n_k}\to0$, so $a_{n_k}\to a$ uniformly by \eqref{eq:ab}, and the jumps of $b_{n_k}$ across the boundaries of stage $n_k$ tiles are at most $C_6A'\delta_{n_k}\to0$. The continuity and differentiability arguments of the proof of Theorem \ref{thm:rigidity} now apply verbatim with the tiles of the good stages in place of the tiles of all stages: two points at distance less than the minimal length of a stage $n_k$ tile lie in the same or in adjacent stage $n_k$ tiles, which gives the uniform continuity of $a$, and for $y\to x$ the maximal good stage tiles contained in $[x,y)$ cover $(x,y)$, because $T_{\max}(n_k)\to0$, and have stages $n_k$ with $k\to\infty$. Hence $Dh=e^{-a}$, and $h$ is a $C^1$ diffeomorphism.
\end{proof}

Theorem C is Theorem \ref{thm:fastconv}.

\begin{theorem}\label{thm:ae}
For Lebesgue almost every $\omega\in\R^2$ there are constants $(A,L)$ for which $\omega$ is fast renormalizable. Consequently, for Lebesgue almost every $\omega$, every pair of commuting orientation preserving $C^3$ circle diffeomorphisms with rotation vector $\omega$ is simultaneously $C^1$ conjugate to the rotation pair.
\end{theorem}

\begin{proof}
It suffices to prove the statement for $\nu$ almost every $\omega\in\Delta$, since $\nu$ is equivalent to Lebesgue measure on $\Delta$ and the coordinates of a vector can be reduced modulo one and reordered. There are three conditions to verify.

\emph{Weighted summable convergence.} By Theorem \ref{thm:brunergodic} the second Lyapunov exponent $\theta_2$ of the cocycle of step matrices is strictly negative, which is the theorem of Hardcastle \cite{Hardcastle} for the algorithm used here, and by the Oseledets theorem $t_1(n)H(n)\le Ce^{\theta_2n/2}$ for $\nu$ almost every $\omega$ and all large $n$. By Theorem \ref{thm:brunergodic} and the invariance of $\nu$, for every $\epsilon>0$ we have $\nu(\{a_{n+1}\ge e^{\epsilon n}\})=\nu(\{a_1\ge e^{\epsilon n}\})\le Ce^{-2\epsilon n}$, which is summable in $n$, so by the Borel-Cantelli lemma $a_{n+1}\le e^{\epsilon n}$ for all large $n$, for $\nu$ almost every $\omega$. Taking $\epsilon=|\theta_2|/4$ gives $\Sigma^{\ast}<\infty$.

\emph{Good stages of positive frequency.} Fix any $A\ge1$ and $L\ge1$. The set $\mathcal{G}(A,L)$ has positive $\nu$ measure, and $T$ is ergodic with respect to $\nu$, so the Birkhoff theorem applied to the indicator of $\mathcal{G}(A,L)$ shows that for $\nu$ almost every $\omega$ the good stages have frequency $\nu(\mathcal{G}(A,L))>0$. Hence they are infinite in number and $n_k/k$ converges to $\nu(\mathcal{G}(A,L))^{-1}$, in particular it is bounded.

\emph{Subexponential drops.} For $K\ge1$ the set $\{(x,y)\in\Delta\colon a/(1-ax)>K\}$, where $a=\lfloor1/x\rfloor$, is contained in the union of $\{x<K^{-1/2}\}$ and of the sets $\{1/a-1/K<x\le1/a\}$ for $1\le a\le K^{1/2}$, because $1-ax<a/K$ means $x>1/a-1/K$. Its Lebesgue measure is therefore at most $\tfrac1{2K}+\sum_{a\le K^{1/2}}\tfrac{1}{aK}\le\tfrac{2+\log K}{K}$, and its $\nu$ measure is at most $C\,\tfrac{2+\log K}{K}$ by Theorem \ref{thm:brunergodic}. Since $\Delta_n$ is the value of $a/(1-ax)$ at $T^n\omega$, the invariance of $\nu$ gives $\nu(\{\Delta_n>e^{\epsilon n}\})\le C(2+\epsilon n)e^{-\epsilon n}$ for every $\epsilon>0$, which is summable, so by the Borel-Cantelli lemma $\log\Delta_n\le\epsilon n$ for all large $n$, for $\nu$ almost every $\omega$. Hence $\log D_k\le\epsilon\,n_{k+1}\le C\epsilon\,k$ for all large $k$, and since $\epsilon$ is arbitrary, $\log D_k=o(k)$.

The three conditions together are Definition \ref{def:fast}, and the second statement is Theorem \ref{thm:fastconv}.
\end{proof}

Theorem D is Theorem \ref{thm:ae}.

\begin{example}\label{ex:fast}
Theorem \ref{thm:fastconv} applies to explicit vectors that are not of simultaneously bounded type. Consider the type
\[
s\ =\ \prod_{j\ge1}(1,B)^{j}(1,A)^{j}\ =\ (1,B)(1,A)\,(1,B)^2(1,A)^2\,(1,B)^3(1,A)^3\cdots ,
\]
with all digits equal to one, so that $\Sigma^{\ast}=\Sigma$, and write $\pi_j=j(j-1)$ for the stage at which the block $(1,B)^j(1,A)^j$ begins, so that the letters $B$ of this block are the symbols with indices $\pi_j+1,\dots,\pi_j+j$ and its letters $A$ those with indices $\pi_j+j+1,\dots,\pi_j+2j$. Write $M_B=M_B(1)$ and $M_A=M_A(1)$, with Perron eigenvalues $\psi$, the supergolden ratio of Section \ref{sec:periodic}, and $\varphi=(1+\sqrt5)/2$. All constants $c$, $C$ below depend on these two matrices only.

\emph{The vector.} The cones $C_n=M_{\varepsilon_1}(a_1)\cdots M_{\varepsilon_n}(a_n)(D)$ are nested and their intersection is a single ray. Indeed $M_B^4$ is strictly positive by Lemma \ref{lem:positivity}, so by the theorem of Birkhoff quoted in the proof of Theorem \ref{thm:reconstruction} each block $(1,B)^j$ with $j\ge4$ contracts the Hilbert metric by a factor at most $\tanh(\Delta(M_B^4)/4)<1$, while the matrices $M_A^{\,j}$ do not expand it. So there is exactly one $\omega=(\alpha,\beta)\in\Delta$ with type $s$. We show below that $t_1(n)H(n)\to0$ along its expansion. This forces $1$, $\alpha$, $\beta$ to be rationally independent. An integer relation gives a nonzero vector $(0,k,l)$ in the lattice $\Lambda$ of Section \ref{sec:beyond}, and in the normalization $g_n\Lambda$ used in the proof of Proposition \ref{prop:sbtbad} this vector has length at most $C/H(n)$ while the three basis vectors lie in the unit cube, so Minkowski's second theorem gives $(t_1(n)H(n)^2)^{-1}=\operatorname{covol}(g_n\Lambda)\le C\lambda_1\lambda_2\lambda_3\le C/H(n)$, that is $t_1(n)H(n)\ge c>0$ for all $n$.

\emph{Scales and heights along the blocks.} Let $u=t(\pi_j+j)$ be the scale vector at the start of the $j$th run of the letter $A$ and $v=t(\pi_{j+1})$ that at its end. Then $t(\pi_j+m)=M_B^{\,j-m}u$ and $t(\pi_j+j+m)=M_A^{\,j-m}v$ for $0\le m\le j$. If $\ell$ is the positive left Perron eigenvector of $M_B$, then $\ell(M_B^ku)=\psi^k\ell(u)$ gives $c\,\psi^k\|u\|\le\|M_B^ku\|\le C\psi^k\|u\|$ for all $u\ge0$ and $k\ge0$, hence
\[
t_1(\pi_j+m)\ \le\ C\psi^{-m}\,t_1(\pi_j),\qquad 0\le m\le j .
\]
The matrix $M_A$ acts on the first two coordinates as the Fibonacci matrix $F$ and fixes the third, and $c\varphi^k\|v'\|\le\|F^kv'\|\le C\varphi^k\|v'\|$ for $v'\ge0$, so
\[
t_1(\pi_j+j+m)\ \le\ C\varphi^{-m}\,t_1(\pi_j+j),\qquad t_3(\pi_j+j+m)\ =\ t_3(\pi_{j+1}),\qquad 0\le m\le j .
\]
For the heights, $B_{\pi_j+m}^{-1}=M_B^{-m}B_{\pi_j}^{-1}$ and $B_{\pi_j+j+m}^{-1}=M_A^{-m}B_{\pi_j+j}^{-1}$, and $H(n)\le\|B_n^{-1}\|_\infty\le3H(n)$ as in Remark \ref{rem:mixedpisot}. The matrix $M_B^{-1}$ has eigenvalues $\psi^{-1}$ and a nonreal pair of modulus $\psi^{1/2}$, and $M_A^{-1}$ has eigenvalues $\varphi^{-1}$, $-\varphi$ and $1$, so
\[
H(\pi_j+m)\ \le\ C\psi^{m/2}H(\pi_j),\qquad H(\pi_j+j+m)\ \le\ C\varphi^{m}H(\pi_j+j),\qquad 0\le m\le j .
\]

\emph{Summable convergence.} Put $q_j=t_1(\pi_j)H(\pi_j)$. The displays give $t_1(n)H(n)\le C\psi^{-m/2}q_j$ at the stages $n=\pi_j+m$ of the block of $B$, and $t_1(n)H(n)\le C\,t_1(\pi_j+j)H(\pi_j+j)\le C\psi^{-j/2}q_j$ at the stages of the run of $A$. In particular $q_{j+1}\le C\psi^{-j/2}q_j$, so $q_j\le C^{\,j}\psi^{-j(j-1)/4}q_1$, and summing over the $2j$ stages of the $j$th block gives $\Sigma=\sum_nt_1(n)H(n)\le\sum_jC(1+j)\,q_j<\infty$. This is condition (a) of Definition \ref{def:fast}, and it gives the rational independence claimed above.

\emph{The vector is not of simultaneously bounded type.} At the start of the $j$th run of $A$ we have $t_1(\pi_j+j)=\|F^{\,j}(t_1,t_2)(\pi_{j+1})\|_\infty\ge c\varphi^{j}\,t_1(\pi_{j+1})\ge c\varphi^{j}\,t_3(\pi_{j+1})$ and $t_3(\pi_j+j)=t_3(\pi_{j+1})$, so $t_1(\pi_j+j)/t_3(\pi_j+j)\ge c\varphi^{j}$, which is unbounded.

\emph{Good stages.} Take $(A,L)=(1,2)$, so that $M=14$. A stage $n$ is good exactly when the letters $\varepsilon_{n+1},\dots,\varepsilon_{n+14}$ contain no two consecutive letters $A$. For $j\ge14$ this holds at the $j-12$ stages $n=\pi_j-1,\dots,\pi_j+j-14$, and for each of these stages the letters $\varepsilon_{n+1},\dots,\varepsilon_{n+14}$ lie in the $j$th block of $B$, together with at most one adjacent letter $A$ on each side. Hence the number of good stages before stage $\pi_{J+1}$ is at least $(J-14)^2/2$ for $J\ge14$. If the $k$th good stage $n_k$ lies in the $J$th block, then $k\ge(J-15)^2/2$ and $n_k\le \pi_{J+1}=J(J+1)$, so $n_k\le Ck$. This is condition (b).

\emph{Drops of scale.} All digits are one, so $\Delta_n=t_1(n)/\tau(v_n)$, and $\tau(v_n)$ is $t_3(n+1)$ at a step of type $B$ and $t_2(n+1)$ at a step of type $A$. Since $t_1(n)=t_1(n+1)+\tau(v_n)$ in both cases, $\Delta_n\le1+t_1(n+1)/t_3(n+1)$, and it remains to bound $t_1/t_3$ along the expansion. At the stages $\pi_j+m$ with $j-m\ge4$ the matrix $M_B^{\,j-m}$ is strictly positive, and the ratio of the first and third coordinates of $M_B^ku$ is at most $C_0=\sup_{k\ge4}\max_i(M_B^k)_{1i}/(M_B^k)_{3i}$, which is finite because $M_B^k/\psi^k$ converges to a positive matrix. For $j-m\le3$ the third coordinate of $M_B^{\,j-m}u$ is at least $u_3$, so the ratio is at most $C\,u_1/u_3$. Along the run of $A$ the third scale is constant and the first decreases, so the ratio is largest at the start of the run, where it is at most $\|F^{\,j}(v_1,v_2)\|_\infty/v_3\le C\varphi^{j}\,t_1(\pi_{j+1})/t_3(\pi_{j+1})\le CC_0\varphi^{j}$ for $j\ge3$, by the bound at the start of the next block. Hence $\Delta_n\le C\varphi^{j}$ for every stage $n$ of the $j$th block. If $n_k$ lies in the $J$th block, then all stages between $n_k$ and $n_{k+1}$ lie in the blocks $J$ and $J+1$, so $D_k\le C\varphi^{J+1}$ with $J\le15+\sqrt{2k}$, and $\log D_k\le C\sqrt k=o(k)$. This is condition (c).

So $\omega$ is fast renormalizable with constants $(1,2)$, and Theorem \ref{thm:fastconv} applies to every pair of commuting $C^3$ diffeomorphisms with rotation vector $\omega$, although $t_1(n)/t_3(n)$ is unbounded along its expansion. The same argument applies to the types $\prod_j(1,B)^{m_j}(1,A)^{\ell_j}$ with $m_j\to\infty$, $\ell_j=O(m_j)$ and $\ell_j=o\big(\sum_{i<j}m_i\big)$.
\end{example}

\subsection{Comparison with the rank one classes}
It is worth recording precisely what the two classes introduced here contribute beyond the rank one theory, since a commuting pair can sometimes be treated one map at a time.

Suppose that $\alpha$ is of bounded type. Then by the rank one theory of \cite{DMP}, or by \cite{KhaninSinai}, the map $f$ is $C^1$ conjugate to $R_\alpha$ by some $h$. The map $hgh^{-1}$ commutes with $R_\alpha$, so by the argument of Lemma \ref{lem:toplin} it is the rotation $R_\beta$, and the same $h$ conjugates the pair. The same holds if $\beta$ is of bounded type, and, with $C^{2+\gamma}$ regularity in place of $C^3$, if one of the two rotation numbers is Diophantine, by \cite{KhaninTeplinsky}. So the content of Theorems B and C lies entirely in the vectors for which no coordinate is covered by the rank one theory.

Such vectors are exactly those with a coordinate that is not of the relevant one dimensional class, while the pair is still of simultaneously bounded type or fast renormalizable. Definition \ref{def:sbt} does not force either coordinate to be of bounded type. Even bad approximability of the vector, which holds for the periodic examples of Section \ref{sec:periodic} by Proposition \ref{prop:sbtbad} but is not known for the whole class, see Remark \ref{rem:notbad}, gives only
\[
\max\big(\|q\alpha\|,\|q\beta\|\big)\ \ge\ c\,q^{-1/2},
\]
and this is compatible with $\|q_k\alpha\|$ being arbitrarily small along a sequence, provided $\|q_k\beta\|$ is not. This is precisely the phenomenon of \cite{Moser} and \cite{FayadKhanin}, where the pair is rigid although no single map need be linearizable. We do not have an explicit vector of simultaneously bounded type, or an explicit fast renormalizable vector such as that of Example \ref{ex:fast}, with a coordinate known to be of unbounded type. For a vector with purely periodic type the coordinates are cubic irrationals. These are Diophantine by Roth's theorem, so they are covered by the theorem of Khanin and Teplinsky at regularity $C^{2+\gamma}$, and whether a cubic irrational has unbounded partial quotients is a well known open problem, so no example of a coordinate of unbounded type can come from that direction. The question of whether the class of Definition \ref{def:sbt}, or the wider class of Definition \ref{def:fast}, contains a vector with a coordinate of unbounded type, or even a Liouville coordinate, is the question that decides whether the mechanism of this paper reaches strictly beyond the coordinatewise route, and we leave it open.

For Theorem \ref{thm:ae} the situation is different and should be stated plainly. For Lebesgue almost every $\omega$ both coordinates are Diophantine, so the coordinatewise route also applies and the conclusion of Theorem \ref{thm:ae} is not new. What is new is the mechanism, a single renormalization scheme for the pair which does not decompose into two one dimensional schemes, and the fact that the same scheme covers the vectors of Theorem \ref{thm:rigidity}, where the coordinatewise route is not available in general.

Finally, the class of \cite{FayadKhanin} is the full simultaneous Diophantine class at infinite smoothness. Definition \ref{def:fast} is a strictly dynamical condition on the Brun orbit and it does not exhaust that class. The obstruction is described in the closing remarks of Section \ref{sec:topohyp}.

\section{Periodic points}\label{sec:periodic}

We work out the two simplest periodic types. Both give periodic points of $\Ren$ on the transversal of rotation pairs, and both illustrate Theorem \ref{thm:periodic}.

\begin{lemma}\label{lem:multipliers}
Let $w$ be a word of length $p$ containing the letter $B$, with period matrix $B_w$, eigenvalues $\lambda>|\lambda_2|\ge|\lambda_3|$ and Perron eigenvector $\omega$. Then the point of $\Delta$ determined by $\omega$ is a fixed point of $T^p$, and the eigenvalues of the derivative of $T^p$ at that point are
\[
\frac{\lambda}{\lambda_2},\qquad \frac{\lambda}{\lambda_3},
\]
both of modulus larger than one. The fixed point is therefore a hyperbolic repelling point of $T^p$ with two dimensional unstable direction.
\end{lemma}

\begin{proof}
By Theorem \ref{thm:periodic} the ray of $\omega$ lies in $D$ and is fixed by $B_w^{-1}$, which induces $T^p$ on $\Delta$ by Lemma \ref{lem:cylinders}. The derivative at a fixed point of the projectivization of a linear map $M$ with eigenvalue $\nu_1$ on the fixed direction has eigenvalues $\nu_2/\nu_1$ and $\nu_3/\nu_1$, where $\nu_2$ and $\nu_3$ are the remaining eigenvalues. Applying this to $M=B_w^{-1}$, whose eigenvalues are $\lambda^{-1},\lambda_2^{-1},\lambda_3^{-1}$ with $\lambda^{-1}$ on the fixed direction, gives the eigenvalues $\lambda/\lambda_2$ and $\lambda/\lambda_3$. They exceed one in modulus because $B_w^4$ is strictly positive by Lemma \ref{lem:positivity}, so that $\lambda^4$ exceeds the modulus of every other eigenvalue of $B_w^4$ by the Perron theorem.
\end{proof}

\subsection{The supergolden vector}
Let $\psi\approx1.4655712$ be the real root of $x^3=x^2+1$, the supergolden ratio, and put $\alpha=\psi^{-1}$, $\beta=\alpha^2$. The relation $\alpha^3+\alpha=1$ holds. The type is $(1,B)$ at every step, so the period matrix is
\[
M_B(1)\ =\ \begin{pmatrix}1&0&1\\1&0&0\\0&1&0\end{pmatrix},
\]
with characteristic polynomial $x^3-x^2-1$ and Perron eigenvalue $\psi$. The word contains the letter $B$, so Theorem \ref{thm:periodic} applies, and $(\alpha,\beta)$ is of simultaneously bounded type. The eigenvector is $(1,\alpha,\alpha^2)$, and the scale vector satisfies $t(n)=\psi^{-n}t(0)$, so a single step of the algorithm is an exact fixed point of the renormalization on the transversal.

The remaining eigenvalues have product $\psi^{-1}$ and sum $1-\psi<0$, and the discriminant of the quadratic factor is $(1-\psi)^2-4\psi^{-1}<0$, so they are nonreal of modulus $\psi^{-1/2}$. By Theorem \ref{thm:periodic} the number $\psi$ is a Pisot unit of the cubic field $\Q(\psi)$. By Lemma \ref{lem:multipliers} the multipliers of the fixed point are
\[
\delta_{\pm}\ =\ \frac{-1\pm i\sqrt{4\psi^{3}-1}}{2},\qquad |\delta_\pm|=\psi^{3/2}\approx1.77423,\qquad \arg\delta_\pm\approx\pm106.368^{\circ},
\]
where the trace is $-1$ and the determinant is $\psi^3$, computed from the derivative of the projectivized step in the chart $\Phi(x,y)=(y/x,(1-x)/x)$, which is
\[
D\Phi\ =\ \begin{pmatrix}-1&\psi\\-\psi^{2}&0\end{pmatrix}
\]
at the fixed point. The heights of the basis elements satisfy the recursion $e(n)=e(n-3)-e(n-2)$ with characteristic polynomial $x^3+x-1$, whose complex roots have modulus $\psi^{1/2}$, so $H(n)\asymp\psi^{n/2}$ and $H(n)^2t_1(n)$ is bounded. Summable convergence, Definition \ref{def:sc}, therefore holds, and the supergolden vector satisfies both hypotheses of Theorem \ref{thm:main}. Its Brun basis is moreover reduced at every stage, so it is badly approximable by Proposition \ref{prop:sbtbad}.

The tower tilings are self similar. The tile counts are the Narayana numbers, the first row of the products $M_B(1)^n$, and the tiling identity at the fixed point reads
\[
3\alpha^{4}+\alpha^{5}+2\alpha^{6}\ =\ 1 .
\]
The associated substitution on three letters, in the labelling of Lemma \ref{lem:tilecount} for a step of type $B$, is $1\mapsto13$, $2\mapsto1$, $3\mapsto2$.

\subsection{The Tribonacci vector}
Let $\theta\approx1.8392868$ be the real root of $x^3=x^2+x+1$ and put $\alpha=\theta^{-1}$, $\beta=\theta^{-2}$. The type has period three, with word $(1,A)(1,B)(1,A)$, and the period matrix is
\[
M_A(1)M_B(1)M_A(1)\ =\ \begin{pmatrix}2&2&1\\1&1&1\\1&0&0\end{pmatrix}\ =\ C^{2},
\]
where $C$ is the companion matrix of $x^3-x^2-x-1$. Its Perron eigenvalue is $\theta^{2}$, its eigenvector is $(1,\theta^{-1},\theta^{-2})$, and the remaining eigenvalues are $\theta_2^2$ and $\overline{\theta_2}^{\,2}$ with $\theta_2$ a nonreal root of $x^3-x^2-x-1$ of modulus $\theta^{-1/2}$. So $\theta^2$ is a Pisot unit, and by Theorem \ref{thm:periodic} the vector is of simultaneously bounded type and summably convergent, with $H(n)\asymp\theta^{n/2}$ as in the supergolden case.

In the chart the return map of the projectivized step over one period is $\Phi=\mathcal{A}\circ\mathcal{B}\circ\mathcal{A}$ with $\mathcal{A}(x,y)=\big((1-x)/x,\,y/x\big)$ and $\mathcal{B}(x,y)=\big(y/x,\,(1-x)/x\big)$, so $\Phi(x,y)=\big((1-x-y)/y,\,(2x-1)/y\big)$, and at the fixed point $(\theta^{-1},\theta^{-2})$
\[
D\Phi\ =\ \begin{pmatrix}-\theta^{2}&-(\theta^{2}+\theta)\\ 2\theta^{2}&-1\end{pmatrix},\qquad
\operatorname{tr}=-(\theta^{2}+1),\qquad \det=\theta^{6}.
\]
The multipliers are
\[
\delta_\pm\ =\ \Big(\frac{\theta}{\theta_2}\Big)^{2}\ \approx\ -2.19149\pm5.82357\,i,\qquad |\delta_\pm|=\theta^{3}\approx6.22226,\qquad \arg\delta_\pm\approx\pm110.622^{\circ},
\]
in agreement with Lemma \ref{lem:multipliers}. The tile counts are the Tribonacci numbers.

\subsection{A totally real example}
Let $w=(1,B)(1,A)(1,B)(1,A)$, the word of Remark \ref{rem:sbt_and_summability}. Its period matrix is
\[
B_w\ =\ M_B(1)M_A(1)M_B(1)M_A(1)\ =\ \begin{pmatrix}3&2&1\\2&2&1\\1&1&1\end{pmatrix},
\]
strictly positive, of determinant one, with characteristic polynomial $x^3-6x^2+5x-1$. Let $c_k=2\cos(2\pi k/7)$, $k=1,2,3$, be the roots of $x^3+x^2-2x-1$. Their product is one and the polynomial with roots $c_k^{-2}$ is $x^3-6x^2+5x-1$, so the eigenvalues of $B_w$ are
\[
\lambda=c_2^{-2}\approx5.0489,\qquad \lambda_2=c_3^{-2}\approx0.6431,\qquad \lambda_3=c_1^{-2}\approx0.3080,
\]
all real, and $\Q(\lambda)$ is the real cubic subfield of the field of seventh roots of unity, of discriminant $49$. The normalized Perron eigenvector is $(1,\alpha,\beta)$ with $\alpha=c_1^{-1}\approx0.80194$ and $\beta=-c_2\approx0.44504$, and by Theorem \ref{thm:periodic} the Brun expansion of $(\alpha,\beta)$ is $w^\infty$, the vector is of simultaneously bounded type and it is summably convergent, since $\lambda$ is a Pisot unit, and $t_1(n)H(n)$ decays like $\lambda_2^{\,n/4}$. Its Brun basis is not reduced: $t_1(n)H(n)^2$ grows by the factor $\lambda_2/\lambda_3\approx2.088$ per period, in agreement with Remark \ref{rem:sbt_and_summability}. Bad approximability of $(\alpha,\beta)$ nevertheless holds, by the theorem of Perron quoted in Remark \ref{rem:notbad}, since $1$, $\alpha$, $\beta$ span a cubic field. In the chart of Section \ref{sec:periodic} the return map over one period is $\Phi=\mathcal{A}\circ\mathcal{B}\circ\mathcal{A}\circ\mathcal{B}$, and at the fixed point
\[
D\Phi\ \approx\ \begin{pmatrix}14.147&-5.049\\-2.802&10.098\end{pmatrix},\qquad
\operatorname{tr}D\Phi=\frac{\lambda}{\lambda_2}+\frac{\lambda}{\lambda_3}\approx24.245,\qquad \det D\Phi=\lambda^3\approx128.705 .
\]
The multipliers are $\lambda/\lambda_2\approx7.851$ and $\lambda/\lambda_3\approx16.394$, real and distinct, in agreement with Lemma \ref{lem:multipliers}. So the transversal dynamics at this periodic point is a genuine expansion with two distinct real rates, and not a rotation composed with an expansion as in the first two examples.

\begin{remark}\label{rem:stablemultipliers}
The two scalar recursions of Section \ref{sec:moebius} have explicit linear parts at a periodic point, and they give the counterpart along the leaves of the multipliers along the transversal. Let $\omega^{\ast}$ have purely periodic type of period $p$ with Perron eigenvalue $\lambda$, so that $t(n+p)=\lambda^{-1}t(n)$ along its expansion. For a pair with rotation vector $\omega^{\ast}$ the ratios $r_2(n)$ are within $C\delta_n$ of $t_2(n)/t_1(n)=t_1(n+1)/t_1(n)$ by \eqref{eq:ratioestimates}, and $\prod_{m<p}t_1(n+m+1)/t_1(n+m)=t_1(n+p)/t_1(n)=\lambda^{-1}$. Proposition \ref{prop:zrecursion} therefore gives $z_1(n+p)=\lambda^{-2}z_1(n)+O(\delta_n)$, and Proposition \ref{prop:sloperecursion} gives $\Sigma(n+p)=\Psi(\Sigma(n))+O(\epsilon_n)$ with $\Psi$ a composition of the maps $\Psi_{r_2(n+m)}$, whose derivative at zero is $\prod_{m<p}r_2(n+m)=\lambda^{-1}+O(\delta_n)$. So the M\"obius parameter and the logarithmic slope of pairs near the periodic point are contracted over one period by the factors $\lambda^{-2}$ and $\lambda^{-1}$, up to errors governed by the mesh, which itself decays over one period by a factor governed by the tile splitting of Proposition \ref{prop:ratios}. For the supergolden fixed point these factors are $\psi^{-2}$ and $\psi^{-1}$, for the Tribonacci orbit $\theta^{-4}$ and $\theta^{-2}$, and for the totally real example above $\lambda^{-2}\approx0.0392$ and $\lambda^{-1}\approx0.1981$.
\end{remark}

The first two examples give periodic orbits of $\Ren$ whose local unstable manifolds are two dimensional and whose multipliers are complex, so the unstable dynamics is a rotation composed with an expansion. The third has real multipliers. All three are rank two counterparts of the golden mean fixed point of the rank one theory, whose single multiplier is the square of the golden mean.

\section{Topological hyperbolicity}\label{sec:topohyp}

The operator $\Ren$ loses derivatives on $C^3$ pairs. A prerenormalization is a change of basis of $\Gamma$, so it does not lose smoothness, but the normalization is not smooth in the pair, and the natural charts on the space of pairs are not preserved. So we do not formulate hyperbolicity as a spectral statement for a derivative. We formulate it as a statement about an invariant lamination together with the dynamics on the transversal, in the sense in which the rank one renormalization horseshoes of \cite{Martens} and \cite{AvilaLyubich} are hyperbolic, and in the sense used in \cite{DMP}.

Let $\mathcal{X}^3$ be the set of pointed pairs $(f,g,x_0)$, where $(f,g)$ is a pair of commuting orientation preserving $C^3$ circle diffeomorphisms whose rotation vector is rationally independent together with $1$ and where $x_0\in\T$ is any base point. Each element of $\mathcal{X}^3$ has its sequence of normalized pairs $\xi_n=\xi_n(f,g,x_0)$, $n\ge0$, of Definition \ref{def:renorm}, since the constructions of Sections \ref{sec:operator} and \ref{sec:towers} require nothing of the base point. Let $\mathcal{N}$ be the set of all these normalized pairs, with the distance $d_1$. By Remark \ref{rem:welldefined} the map $\Ren\colon\mathcal{N}\to\mathcal{N}$, $\xi_n(f,g,x_0)\mapsto\xi_{n+1}(f,g,x_0)$, is well defined. Let $\Pi\colon\mathcal{N}\to\Delta$ send a normalized pair to its rotation vector, the pair of rotation numbers of $P_n$ and $Q_n$ on $\mathcal{C}_n$, and let $\mathcal{L}_\omega=\Pi^{-1}(\omega)$. The image of $\Pi$ is the set $\Delta'$ of points of $\Delta$ that are rationally independent together with $1$, a dense subset of $\Delta$ of full measure on which $T$ and all its iterates are defined. Two normalized pairs are topologically conjugate if there is a homeomorphism of the circle $[0,1)/\{0\sim1\}$ fixing the seam which conjugates the two pairs of circle homeomorphisms.

The transversal $\mathcal{T}\subset\mathcal{N}$ is the set of all rotation pairs $\mathcal{O}(\omega)$, $\omega\in\Delta'$. These are the normalized pairs of the pointed pairs $(R_{\omega_1},R_{\omega_2},x_0)$, and they do not depend on $x_0$. A rotation pair is a pair of rigid rotations of the circle of length one written in two branch form, so
\[
d_1\big(\mathcal{O}(\omega),\mathcal{O}(\omega')\big)\ =\ \max\big(\|\omega_1-\omega_1'\|,\ \|\omega_2-\omega_2'\|\big),
\]
where $\|\cdot\|$ is the distance to the nearest integer. Hence $\omega\mapsto\mathcal{O}(\omega)$ is a homeomorphism of $\Delta'$ onto $\mathcal{T}$, and an isometry for the maximum norm on every set of diameter less than one half. In this sense $\mathcal{T}$ is a two dimensional family.

Let $\mathcal{X}^3_{A}\subset\mathcal{X}^3$ be the set of pointed pairs whose rotation vector is of simultaneously bounded type with constant $A$ and summably convergent and whose base point is admissible in the sense of Convention \ref{con:basepoint}. By Proposition \ref{prop:basepoint} every such pair admits a set of admissible base points of full $\mu$ measure. Let $\mathcal{N}_A\subset\mathcal{N}$ be the set of normalized pairs of the elements of $\mathcal{X}^3_{A}$. It is $\Ren$ invariant, because the three defining conditions are conditions on tails, and every statement about $\Ren^k\xi_n(f,g,x_0)$ is a statement about $\xi_{n+k}(f,g,x_0)$, to which the estimates of Sections \ref{sec:realbounds} to \ref{sec:rigidity} apply. The leaf space $\Pi(\mathcal{N}_A)$ is the set of vectors of simultaneously bounded type with constant $A$ which are summably convergent, and $\mathcal{T}\cap\mathcal{N}_A$ is the set of rotation pairs with these vectors, the normalized pairs of the pointed pairs $(R_\alpha,R_\beta,x_0)$ with such a vector, for which every base point is admissible.

\begin{theorem}\label{thm:topohyp}
The partition $\{\mathcal{L}_\omega\}$ of $\mathcal{N}$ has the following properties.
\begin{enumerate}
\item[(i)] It is $\Ren$ invariant, $\Ren(\mathcal{L}_\omega)\subset\mathcal{L}_{T\omega}$, and each leaf is the topological conjugacy class of any of its points.
\item[(ii)] The leaves of $\mathcal{N}_A$ are contracted, at a rate depending only on $A$. There is $\xi_A\in(0,1)$ and, for every pair of points $\zeta,\zeta'$ of a common leaf of $\mathcal{N}_A$, a constant $C$ with $d_1(\Ren^k\zeta,\Ren^k\zeta')\le C\xi_A^{\,k}$ for all $k\ge0$.
\item[(iii)] The family $\mathcal{T}$ of rotation pairs is a global transversal. It meets every leaf of $\mathcal{N}$ in exactly one point, and $\Ren$ acts on it as the accelerated Brun map, $\Ren\,\mathcal{O}(\omega)=\mathcal{O}(T\omega)$. On the leaf space $\Pi(\mathcal{N}_A)$ the induced action is $T$, and the renormalization type, which is injective on $\Pi(\mathcal{N}_A)$ by Theorem \ref{thm:reconstruction}, carries it to the shift on the set of types of the vectors in $\Pi(\mathcal{N}_A)$, a shift invariant set of sequences with digits at most $a^{\ast}(A)$ and gaps between occurrences of the letter $B$ at most $L(A)$.
\item[(iv)] Let $\omega^{\ast}\in\Delta$ be a periodic point of $T$ of period $p$ whose period matrix $B_w$ is Pisot. Then $\omega^{\ast}\in\Pi(\mathcal{N}_A)$ for some $A$, and $\mathcal{O}(\omega^{\ast})$ is a periodic point of $\Ren$ of period $p$ on the transversal. There is a neighbourhood $U$ of $\omega^{\ast}$ in $\Delta$ such that $\mathcal{W}=\{\mathcal{O}(\omega)\colon\omega\in U\cap\Delta'\}$ is a local unstable manifold of $\Ren^p$ at $\mathcal{O}(\omega^{\ast})$, in the following sense. The set $\mathcal{W}$ is the trace on $\mathcal{T}$ of the open set $U$, so it is two dimensional, every point of $\mathcal{W}$ has exactly one backward orbit of $\Ren^p$ in $\mathcal{W}$, and this orbit converges to $\mathcal{O}(\omega^{\ast})$ at exponential rate. The multipliers of $\Ren^p$ at $\mathcal{O}(\omega^{\ast})$ along $\mathcal{T}$ are the ratios $\lambda/\lambda_2$, $\lambda/\lambda_3$ of Lemma \ref{lem:multipliers}, both of modulus larger than one. Along the leaves of $\mathcal{N}_A$ there is no expansion, by (ii).
\end{enumerate}
\end{theorem}

\begin{proof}
(i) The rotation vector of the renormalized pair is $T$ applied to the rotation vector, because the scale vector transforms by the Brun step, by Definition \ref{def:prestep}, and the rotation numbers of $P_n$ and $Q_n$ are the ratios $t_2(n)/t_1(n)$ and $t_3(n)/t_1(n)$. Let $\zeta,\zeta'\in\mathcal{L}_\omega$. Each is a chart of a pointed pair of $C^3$ diffeomorphisms of a circle with rotation vector $\omega$, by Remark \ref{rem:welldefined}, and Lemma \ref{lem:toplin} applied on that circle provides conjugacies $h,h'$ to the pair of rotations by $\omega$. Composing with rotations we may assume $h(0)=h'(0)=0$, and then $h'^{-1}\circ h$ conjugates $\zeta$ to $\zeta'$ and fixes the seam. Conversely a topological conjugacy preserves rotation numbers. So the leaves are exactly the topological conjugacy classes.

(ii) Let $\zeta=\xi_n(f,g,x_0)$ and $\zeta'=\xi_{n'}(f',g',x_0')$ lie on a common leaf $\mathcal{L}_\omega$ of $\mathcal{N}_A$, so that $T^n\omega_0=T^{n'}\omega_0'=\omega$, where $\omega_0,\omega_0'$ are the rotation vectors of the two pointed pairs, which belong to $\mathcal{X}^3_A$. Then $\Ren^k\zeta=\xi_{n+k}(f,g,x_0)$ and $\Ren^k\zeta'=\xi_{n'+k}(f',g',x_0')$ have the same rotation model $\mathcal{O}(T^k\omega)$, since $\mathcal{O}_m$ depends only on the rotation vector at stage $m$. By Remark \ref{rem:uniformrate} applied to each pointed pair, and by the triangle inequality for $d_1$,
\begin{align*}
d_1\big(\Ren^k\zeta,\Ren^k\zeta'\big)\ &\le\ d_1\big(\xi_{n+k}(f,g,x_0),\mathcal{O}_{n+k}\big)+d_1\big(\mathcal{O}_{n'+k},\xi_{n'+k}(f',g',x_0')\big)\\
&\le\ C\xi_A^{\,n+k}+C'\xi_A^{\,n'+k}\ \le\ C''\xi_A^{\,k},
\end{align*}
where $\xi_A$ is the rate of Remark \ref{rem:uniformrate}, which depends only on $A$.

(iii) A pair of rigid rotations is prerenormalized and normalized to a pair of rigid rotations, and Definition \ref{def:renorm} shows that the parameters transform by the projectivized Brun step, so $\Ren\,\mathcal{O}(\omega)=\mathcal{O}(T\omega)$ for every $\omega\in\Delta'$. The transversal meets every leaf in exactly one point, namely the rotation pair with that rotation vector, so it is global. By (i) the induced action on the leaf space $\Pi(\mathcal{N})=\Delta'$ is $T$, and $\Pi(\mathcal{N}_A)$ is $T$ invariant because its defining conditions are conditions on tails. Every vector in $\Pi(\mathcal{N}_A)$ is of simultaneously bounded type, so the type determines the vector by Theorem \ref{thm:reconstruction}, and $T$ acts on the type as the shift by Definition \ref{def:prestep}. The bounds on the digits and on the gaps are Corollary \ref{cor:cantor}.

(iv) Let $w$ be the word formed by the first $p$ symbols of the expansion of $\omega^{\ast}$, so that $\omega^{\ast}$ has purely periodic type $w^\infty$, and $w$ contains the letter $B$ by Lemma \ref{lem:positivity}. By Theorem \ref{thm:periodic}, $\omega^{\ast}$ is the normalized Perron eigenvector of $B_w$, it is of simultaneously bounded type with some constant $A$, and it is summably convergent because $B_w$ is Pisot. The characteristic polynomial of $B_w$ is irreducible over $\Q$: a rational root would be $\pm1$, since it divides the constant term $\pm1$, and $B_w$ has no eigenvalue of modulus one, because $\lambda>1$ and the other two eigenvalues lie in the open unit disk. Hence $1$, $\omega^{\ast}_1$, $\omega^{\ast}_2$ are rationally independent by Theorem \ref{thm:periodic}, so $\omega^{\ast}\in\Pi(\mathcal{N}_A)$, and $\mathcal{O}(\omega^{\ast})$ is a periodic point of $\Ren$ of period $p$ on $\mathcal{T}$ by (iii).

Let $\Delta_w=B_w(D)/\R_+$ be the cylinder of $w$, an open subset of $\Delta$ containing $\omega^{\ast}$ by Lemma \ref{lem:cylinders}, and let $\varphi\colon\Delta\to\Delta_w$ be the projectivization of $B_w$, the inverse of $T^p$ on $\Delta_w$. It is a real analytic diffeomorphism fixing $\omega^{\ast}$, and by Lemma \ref{lem:multipliers} the eigenvalues of $D\varphi(\omega^{\ast})$ are $\lambda_2/\lambda$ and $\lambda_3/\lambda$, both of modulus less than one. Hence there are a neighbourhood $U\subset\Delta_w$ of $\omega^{\ast}$, a norm $|\cdot|$ on $\R^2$ and $\vartheta\in(0,1)$ such that $\varphi(U)\subset U$ and $|\varphi(\omega)-\omega^{\ast}|\le\vartheta|\omega-\omega^{\ast}|$ for $\omega\in U$. Since $B_w$ is an integer matrix of determinant $\pm1$, the map $\varphi$ preserves rational independence, so it maps $U\cap\Delta'$ into itself. Let $\mathcal{W}=\{\mathcal{O}(\omega)\colon\omega\in U\cap\Delta'\}$, the trace on $\mathcal{T}$ of the open set $U$. For $\omega\in U\cap\Delta'$ the points $\mathcal{O}(\varphi^{k}\omega)$, $k\ge0$, lie in $\mathcal{W}$, they satisfy $\Ren^p\mathcal{O}(\varphi^{k+1}\omega)=\mathcal{O}(\varphi^{k}\omega)$ by (iii), because $T^p\circ\varphi$ is the identity, and they converge to $\mathcal{O}(\omega^{\ast})$ at the rate $\vartheta^{k}$, because $\omega\mapsto\mathcal{O}(\omega)$ is an isometry near $\omega^{\ast}$. This backward orbit is unique in $\mathcal{W}$, since a point $\omega'\in U\subset\Delta_w$ with $T^p\omega'=\omega$ equals $\varphi(\omega)$ by the injectivity of $T^p$ on $\Delta_w$. Finally $\Ren^p$ acts on $\mathcal{T}$ as $T^p$, whose derivative at $\omega^{\ast}$ has the eigenvalues $\lambda/\lambda_2$ and $\lambda/\lambda_3$ of Lemma \ref{lem:multipliers}, and these are the multipliers along $\mathcal{T}$. The last statement is (ii).
\end{proof}

Theorem E is Theorem \ref{thm:topohyp}.


\begin{thebibliography}{99}

\bibitem{ArnouxNogueira} P. Arnoux and A. Nogueira, \textit{Mesures de Gauss pour des algorithmes de fractions continues multidimensionnelles}, Ann. Sci. \'Ecole Norm. Sup. (4) \textbf{26} (1993), 645--664.

\bibitem{AvilaLyubich} A. Avila and M. Lyubich, \textit{The full renormalization horseshoe for unimodal maps of higher degree - exponential contraction along hybrid classes}, Publ. Math. Inst. Hautes \'Etudes Sci. \textbf{114} (2011), 171--223.

\bibitem{BST} V. Berth\'e, W. Steiner and J. Thuswaldner, \textit{Multidimensional continued fractions and symbolic codings of toral translations}, J. Eur. Math. Soc. \textbf{25} (2023), 4997--5057.

\bibitem{Birkhoff} G. Birkhoff, \textit{Extensions of Jentzsch's theorem}, Trans. Amer. Math. Soc. \textbf{85} (1957), 219--227.

\bibitem{Brun} V. Brun, \textit{Algorithmes euclidiens pour trois et quatre nombres}, Treizi\`eme congr\`es des math\'ematiciens scandinaves, Helsinki, 1957, 45--64.

\bibitem{Cassels} J. W. S. Cassels, \textit{An Introduction to Diophantine Approximation}, Cambridge Tracts in Mathematics 45, Cambridge University Press, 1957.

\bibitem{Cheung} Y. Cheung, \textit{Hausdorff dimension of the set of singular pairs}, Ann. of Math. (2) \textbf{173} (2011), 127--167.

\bibitem{CoulletTresser} P. Coullet and C. Tresser, \textit{It\'eration d'endomorphismes et groupe de renormalisation}, C. R. Acad. Sci. Paris S\'er. A-B \textbf{287} (1978), A577--A580.

\bibitem{Dani} S. G. Dani, \textit{Divergent trajectories of flows on homogeneous spaces and Diophantine approximation}, J. Reine Angew. Math. \textbf{359} (1985), 55--89.

\bibitem{deFariadeMelo} E. de Faria and W. de Melo, \textit{Rigidity of critical circle mappings I}, J. Eur. Math. Soc. \textbf{1} (1999), 339--392.

\bibitem{Denjoy} A. Denjoy, \textit{Sur les courbes d\'efinies par les \'equations diff\'erentielles \`a la surface du tore}, J. Math. Pures Appl. (9) \textbf{11} (1932), 333--375.

\bibitem{DMP} Y. Dong, M. Martens, L. Palmisano, One-Dimensional Dynamics: From Poincar\'{e} to Renoramlization, Lecture Notes of the Unione Matematica Italiana, vol 30. Springer (2025).

\bibitem{Dong} B. Dong, \textit{Renormalization of circle diffeomorphisms}, Master's thesis, KTH Royal Institute of Technology, Stockholm, 2023.

\bibitem{FayadKhanin} B. Fayad and K. Khanin, \textit{Smooth linearization of commuting circle diffeomorphisms}, Ann. of Math. (2) \textbf{170} (2009), 961--980.

\bibitem{Feigenbaum} M. J. Feigenbaum, \textit{Quantitative universality for a class of nonlinear transformations}, J. Statist. Phys. \textbf{19} (1978), 25--52.

\bibitem{Herman} M. R. Herman, \textit{Sur la conjugaison diff\'erentiable des diff\'eomorphismes du cercle \`a des rotations}, Inst. Hautes \'Etudes Sci. Publ. Math. \textbf{49} (1979), 5--233.

\bibitem{Hardcastle} D. M. Hardcastle, \textit{The three-dimensional Gauss algorithm is strongly convergent almost everywhere}, Experiment. Math. \textbf{11} (2002), 131--141.

\bibitem{HardcastleKhanin} D. M. Hardcastle and K. Khanin, \textit{The $d$-dimensional Gauss transformation - strong convergence and Lyapunov exponents}, Experiment. Math. \textbf{11} (2002), 119--129.

\bibitem{KatznelsonOrnstein} Y. Katznelson and D. Ornstein, \textit{The differentiability of the conjugation of certain diffeomorphisms of the circle}, Ergodic Theory Dynam. Systems \textbf{9} (1989), 643--680.

\bibitem{KhaninSinai} K. M. Khanin and Ya. G. Sinai, \textit{A new proof of M. Herman's theorem}, Comm. Math. Phys. \textbf{112} (1987), 89--101.

\bibitem{KhaninTeplinsky} K. Khanin and A. Teplinsky, \textit{Herman's theory revisited}, Invent. Math. \textbf{178} (2009), 333--344.

\bibitem{Lanford} O. E. Lanford III, \textit{A computer-assisted proof of the Feigenbaum conjectures}, Bull. Amer. Math. Soc. (N.S.) \textbf{6} (1982), 427--434.

\bibitem{Lyubich} M. Lyubich, \textit{Feigenbaum-Coullet-Tresser universality and Milnor's hairiness conjecture}, Ann. of Math. (2) \textbf{149} (1999), 319--420.

\bibitem{Martens} M. Martens, \textit{The periodic points of renormalization}, Ann. of Math. (2) \textbf{147} (1998), 543--584.

\bibitem{McMullen} C. T. McMullen, \textit{Renormalization and 3-Manifolds Which Fiber over the Circle}, Annals of Mathematics Studies 142, Princeton University Press, 1996.

\bibitem{Moser} J. Moser, \textit{On commuting circle mappings and simultaneous Diophantine approximations}, Math. Z. \textbf{205} (1990), 105--121.

\bibitem{Schweiger} F. Schweiger, \textit{Multidimensional Continued Fractions}, Oxford University Press, 2000.

\bibitem{Sullivan} D. Sullivan, \textit{Bounds, quadratic differentials, and renormalization conjectures}, in American Mathematical Society Centennial Publications, Vol. II, Amer. Math. Soc., 1992, 417--466.

\bibitem{Yampolsky} M. Yampolsky, \textit{Hyperbolicity of renormalization of critical circle maps}, Publ. Math. Inst. Hautes \'Etudes Sci. \textbf{96} (2002), 1--41.

\bibitem{Yoccoz} J.-C. Yoccoz, \textit{Analytic linearization of circle diffeomorphisms}, in Dynamical Systems and Small Divisors, Lecture Notes in Math. 1784, Springer, 2002, 125--173.

\end{thebibliography}
\end{document}